\documentclass[a4paper,11pt]{article}
\usepackage[colorlinks,bookmarksopen,bookmarksnumbered,citecolor=blue, linkcolor=black, urlcolor=blue]{hyperref}
\usepackage{graphicx}
\usepackage{indentfirst}
\usepackage{slashbox}
\usepackage[centertags]{amsmath}
\usepackage{amsfonts}
\usepackage{amssymb}
\usepackage{amsthm}
\usepackage{fancyhdr}
\usepackage{mathrsfs}
\usepackage{amsmath}
\usepackage{multirow}
\usepackage[misc]{ifsym}
\usepackage{booktabs}
\usepackage{stmaryrd}
\usepackage{algorithm}
\usepackage{algorithmic}
\usepackage{natbib}
\usepackage{enumerate}
\hypersetup{hypertex=true,
	colorlinks=true,
	linkcolor=blue,
	anchorcolor=blue,
	citecolor=blue,
	pdfstartview=Fit,
	breaklinks=true}

\allowdisplaybreaks[4]
\numberwithin{equation}{section}
\def\E{\operatorname*{E}}
\def\Var{\operatorname*{Var}}

\def\Cov{\operatorname*{Cov}}
\def\P{\operatorname*{P}}
\begin{document}
\newtheorem{theorem}{\bf Theorem\,}[section]
\newtheorem{lemma}{\bf Lemma\,}[section]
\newtheorem{remark}{\bf Remark\,}
\newtheorem{example}{\bf Example\,}
\newtheorem{proposition}{\bf Propositionâ\,}[section]
\newtheorem{corollary}{\bf Corollary\,}[section]
\newtheorem{assumption}{\bf Assumption\,}
\newtheorem*{definition}{\bf Definition\ }

\theoremstyle{remark}
\newtheorem*{notation}{\bf Notation}

\renewcommand{\tablename}{Table}
\renewcommand\refname{\large References}

\date{}

\title{\textbf{  Consistent complete independence test in high dimensions based on 
	  Chatterjee correlation coefficient }}
 \author{Liqi Xia$^1$, Ruiyuan Cao$^1$,  Jiang Du$^{1,2,*}$ and Jun Dai$^1$\bigskip\\
	1. School of Mathematics, Statistics and Mechanics, Beijing University of Technology,\\ Beijing 100124, China \bigskip\\
	2. Beijing Institute of Scientific and Engineering Computing, Beijing 100124, China\bigskip\\
	E-mail address: dujiang84@163.com (Jiang Du)}
\maketitle

\maketitle{\textbf{Abstract:}} 
In this article, we consider the complete independence test of high-dimensional data. Based on 
Chatterjee coefficient, we pioneer the development of quadratic test and extreme value test which possess good testing performance for oscillatory data, and establish  the corresponding  large sample properties  under both   null hypotheses and alternative hypotheses. In order to overcome the shortcomings of quadratic statistic and extreme value statistic, we propose a testing method termed as power enhancement test  by adding a screening statistic to the quadratic statistic. The proposed method do not reduce the testing power under dense alternative hypotheses, but   can  enhance the power significantly  under sparse alternative hypotheses. Three synthetic data examples and two real data examples are further used to illustrate the performance of our proposed methods.

\section{Introduction}

Let  $ \boldsymbol{X}=\left(X_{1}, \ldots, X_{p}\right)^{\top}  $ be a random vector taking values in  $ \mathbb{R}^{p}  $ with all univariates following continuous marginal distributions $F_{k}\left( X_k\right) $ $ (1\leqslant k \leqslant p $) and all bivariates following  continuous joint distributions $ F_{kl}\left( X_k,X_l\right) $ ($ 1\leqslant k \neq l\leqslant p $). A natural consideration for testing independence issue of multivariate variables is
\begin{eqnarray} \label{H1.1}
	 H_{0}: X_{1}, \ldots, X_{p} \quad \text { are mutually independent } 
\end{eqnarray}
based on  $ n $  independent identically distributed (i.i.d.)  realizations  $ \boldsymbol{X}_{1}, \ldots, \boldsymbol{X}_{n}$  from  $\boldsymbol{X}  $, where $\boldsymbol{X}_i=(X_{i1},\ldots,X_{ip})^\top,i=1,\ldots,n$. This article mainly focuses on the independence test of high-dimensional data, that is, the dimension $ p $ of the random vector $  \boldsymbol{X} $ exceeds the sample size $ n $. There are currently several prevalent asymptotic regimes regarding $ (n, p) $ for high-dimensional data. The most common case is that the sample size $ n $ and dimension $ p $ are comparable or $ p $ may go to infinity in some way as $ n $ does. For instance, as $ n $ and $ p $ go to infinity, $ n/p$ converges to a constant $\gamma \in(0,\infty)$, there is already a large amount of relevant research, see \cite{mao2014new}, \cite{Schott2005testing}, \cite{bao2015spectral}, \cite{mao2017robust}, etc. Obviously, the above restrictions on $ n $ and $ p $ are excessively stringent, therefore,  researches released the  constraint for both $ n $ and $ p $ for  tending towards infinity, relevant literatures see \cite{paindaveine2016high}, \cite{leung2018testing}, \cite{yao2018testing}, etc. Undoubtedly, only restricting  $ p $ to infinity, regardless of whether $ n $ tending to infinity or is fixed, is the most relaxed constraint, and this is exactly the asymptotic mechanism we pursue in this article to establish asymptotic null distribution of quadratic statistics, also see \cite{mao2015note}, \cite{mao2018testing}, etc.

According to different data architectures, statistical methods dedicated to high-dimensional independence testing are mainly divided into three major types. The first type is the quadratic type (i.e. $ \mathcal{L}_{2} $-type test) which exhibits good testing performance only under alternative hypotheses of dense data, latest representative works include \cite{mao2017robust},  \cite{mao2018testing}, \cite{yao2018testing}, \cite{leung2018testing}, etc. The second type is extreme value type (i.e. the maximum-type test or $ \mathcal{L}_\infty $-type test ), which requires a sparse data architecture. The performance of extreme value test will only validate when there are relatively few dependent paired variables present, refer to \cite{cai2011limiting}, \cite{han2017distribution}, \cite{drton2020high}, etc. Specially noteworthy, newly published \cite{shi2023max} has developed a third type test  that can simultaneously process two types of data. They minimize the $p$-value of the extreme value test and the $p$-value of the quadratic type test as the  new test statistic, called max-sum statistic  in their article. The  max-sum test method has been proven to have good superiority over other  methods which are tailored for individual data types under the dense or sparse alternative hypotheses. Meanwhile, we have been exploring both quadratic test and extreme value test, as well as pursuing a new type of test that can tackle both dense and sparse data simultaneously. To be specific, our approach is quite different from theirs. Consequently, we propose a new test method based on  quadratic type test by adding a screening statistic that can control extreme signal in sparse alternative data. The newly proposed test has been verified in the subsequent numerical simulation to have no decrease in test power under  dense alternative hypotheses, while enhance its power  significantly   under   sparse alternative hypotheses.

Lately, a new coefficient, Chatterjee rank correlation coefficient (\cite{chatterjee2021new}), has entered the family of correlation coefficients and has the following desirable merits, but not limited to them.
\begin{enumerate}
 \item [(a)]  It consistently estimates a population quantity (\cite{dette2013copula}) which is 0 if and only if two variables are independent and 1 if and only if one is a measurable function of the other. This population, also known as the correlation measure, can detect both linear and nonlinear dependencies  and is particularly friendly to oscillatory dependencies. 
 \item [(b)]  It follows a  normal distribution asymptotically under the hypothesis of independence. In fact, under the dependent alternative hypothesis, its normal limit distribution theory has also been derived (\cite{lin2022limit}). 
 \item [(c)]  The characteristic of its rank based statistic leads to the corresponding test fully distribution-free. This attractive feature is clearly ensured by rank-based tests for continuous distributions. 
\end{enumerate}

To our knowledge, Chatterjee's new coefficient has not been applied to the independence test of high-dimensional data. However, the test of high-dimensional oscillatory data in the fields of biology and medicine is rarely undertaken by ordinary correlation coefficients. Typically, numerous biological systems oscillate over time or space, from metabolic oscillations in yeast to the physiological cycle in mammals. Rapidly increasing, bioinformatics techniques are being utilized to quantify these biological oscillation systems, resulting in an increasing number of high-dimensional datasets with periodic oscillation signals.  Classic high-dimensional tests based on linear or monotonic correlation coefficients are difficult to generalize to handle these high-dimensional periodic data. Therefore, high-dimensional tests based on Chatterjee's  new coefficients are bound to be a refreshing trend.

Our work mainly involves three aspects, and to our knowledge,  similar works using Chatterjee rank correlation for these aspects are vacant. 
\begin{enumerate}
\item We derived the fourth-order moment and covariance of Chatterjee rank correlation under the null hypothesis through brute force, and perhaps there is a shortcut worth patiently exploring for their solution, similar to higher-order moments of Pearson or Spearman coefficients. Furthermore, we proved that the quadratic statistic of high-dimensional test follows a standard normal distribution after appropriate pairwise combinations. Under dense alternative hypotheses, our quadratic test power uniformly converges to 1. 

\item To address the issue of low power under sparse alternative hypotheses, we applied extreme value statistic to test the sparse alternative signal. Under $ H_0 $, the extreme value statistic was validated to follow Gumbel distribution with slight necessary adjustments. The theoretical and simulation results demonstrated that our extreme value test statistic has high power against  sparse alternative hypotheses. 

\item In order to overcome the shortcomings of quadratic statistic and extreme statistic, similar to  \cite{fan2015power}, we added a screening statistic to the quadratic test statistic  to  detect sparse noise. Consequently, the power of our test will not decrease under dense alternative hypotheses, but can be enhanced under sparse alternative hypotheses. The screening statistic can screen  out variables with significant  dependent relationships.
\end{enumerate}
	
The remaining of this article is arranged as follows. In Section \ref{section2}, the variance of the squared Chatterjee coefficient and the quadratic statistic are derived, and asymptotic theories of the proposed quadratic statistic under both the null hypothesis and alternative hypothesis are presented, respectively. The asymptotic distribution of the constructed extreme value test statistic  under the null hypothesis and the consistency of extreme value test under the alternative hypothesis are given in Section \ref{section3}. In Section \ref{section4}, we search for a threshold to control extreme distribution noise under mild conditions. Based on these theories, we further introduce a screening statistic   and a screening set  that can enhance the power of the quadratic test. In addition, we also explore the properties of the power enhanced test and the screening set  under both the null and alternative hypotheses. In Section \ref{section5}, we generate three synthetic data examples to verify the performance of our proposed tests and compare them with existing tests. In Section \ref{section6}, we analyze two real datasets, the leaf dataset and the circadian gene transcription dataset, to illustrate the practical application   of the proposed tests. We summarize the entire article in Section \ref{section7}. We defer all proofs to Appendix \ref{appendix}.

\begin{notation}
Superscript ``$ {\top} $"   denotes transposition. Throughout, $ c $, $ C $, $ C_1 $ and $ C_2 $ refer to positive constants whose values may differ in different parts of the paper.  The set of real numbers is denoted as $ \mathbb{R} $.   For any two real sequences $  \left\{a_{n}\right\} $  and  $ \left\{b_{n}\right\}  $, we write $ a_{n}=O\left(b_{n}\right) $ if there exists $  C>0 $  such that  $ \left|a_{n}\right| \leqslant C\left|b_{n}\right| $, $ a_{n} \asymp b_{n} $  if there exists $ c $, $ C>0 $  such that  $ c\left|a_{n}\right| \leqslant\left|b_{n}\right| \leqslant C\left|a_{n}\right| $  and $ a_{n}=o\left(b_{n}\right) $  if for any  $ c>0 $ such that  $ \left|a_{n}\right| \leqslant c\left|b_{n}\right|  $  for any large enough  $ n $.  A set consisting of $ n $ distinct elements $ x_1, \cdots , x_n $ is written as either $ \{x_1, \cdots , x_n\} $ or $ \{x_i\}^n_{i=1} $. Symbol ``$:= $" means ``define''.
\end{notation}

\section{High dimensional test based on sum of squared Chatterjee coefficient}\label{section2}

Assume $\{(X_{1k},X_{1l}),\cdots,(X_{nk},X_{nl}) \}$ is a finite i.i.d. sample from bivariate vector $ (X_k,X_l), 1\leqslant k\neq l\leqslant p $. Rearrange $(X_{1k},X_{1l}),\ldots,(X_{nk},X_{nl}) $ as  $(X_{(1)k},X_{[1]l}^k),\ldots,(X_{(n)k},X_{[n]l}^k)  $ such that $X_{(1)k}\leqslant\cdots\leqslant X_{(n)k}$, that is, if we denote by $ X_{(r)k} $ the $ r $th order statistic of the $ X_k $ sample value, then
the $ X_l $ value associated with $ X_{(r)k} $ is called the concomitant of the $ r $th-order statistic and is denoted
by $ X_{[r]l}^k $. Let  $ R_{il}=\sum_{j=1}^nI\left(X_{jl}\leq X_{il}\right)$ and $ R_{[i]l}^k=\sum_{j=1}^nI \left (X_{[j]l}^k\leq X_{[i]l}^k \right)$ be the rank of $ X_{il} $ and $X_{[i]l}^k $ with indicator function $I(\cdot)$, respectively. Then, Chatterjee's new coefficient proposed by \cite{chatterjee2021new} is as follows,
\begin{eqnarray}\label{xidefinition}
	 \hat{\xi}_{kl}:=\xi_n\left( \{(X_{ik},X_{il})\}_{i=1}^n\right) =1-\dfrac{3\sum^{n-1}_{i=1}\left|R_{[i+1],l}^k-R_{[i]l}^k\right|}{n^2-1},\ 1\leqslant k \neq l\leqslant p.
\end{eqnarray}
 $ \hat{\xi}_{kl} $ can be deemed as a consistent estimator of the population $ \xi_{kl} $ which is an associate measure used to detect functional dependency relationships between two components $ X_k $ and $ X_l $ in $  \boldsymbol{X} $.  Thus, the hypothesis test (\ref{H1.1}) considered in this article can be reset to the following form,
$$H_{0}:\xi_{kl}=0,\ \ \ \ \ \ \ 1\leqslant k \neq l\leqslant p \ \ \ \ \text{vs}$$ 
$$H_{a}:\text{there is at least one pair $ (k,l) $ such that }  \xi_{kl}>0,\ \ \ 1\leqslant k \neq l\leqslant p,$$
which is further equivalent to 
 $$H_{0}:\boldsymbol{\xi}_p=\boldsymbol{0},\ \ \ \ \ \ \text{vs}\ \ \ \ \ \ H_{a}:\boldsymbol{\xi}_p\neq\boldsymbol{0},$$
 where $ \boldsymbol{\xi}_p=\left(\xi_{12},\cdots,\xi_{1p},\xi_{21},\cdots,\xi_{2p},\cdots,\xi_{p-1,p} \right)^{\top}\in \mathbb{R}^{p(p-1)}$ with consistent estimator   presented as $ \boldsymbol{\hat{\xi}}_p=\left(\hat{\xi}_{12},\cdots,\hat{\xi}_{1p},\hat{\xi}_{21},\cdots,\hat{\xi}_{2p},\cdots,\hat{\xi}_{p-1,p} \right)^{\top}\in \mathbb{R}^{p(p-1)}$. All such $ H_a $'s constitute the alternative hypothesis space $ \boldsymbol{\Xi}_a$, i.e., $ \boldsymbol{\Xi}_a=\{\boldsymbol{\xi}_p\in\mathbb{R}^{p(p-1)}:\boldsymbol{\xi}_p\neq\boldsymbol{0}\}$, and further define $ \boldsymbol{\Xi}:=\boldsymbol{\Xi}_a\cup \{\boldsymbol{0}\}$  as the parameter space of  $ \boldsymbol{\xi}_p $  that covers the union of $  \{\mathbf{0}\} $  and the alternative set $  \boldsymbol{\Xi}_{a}  $.
 
We need to emphasize that  statistic $ \hat{\xi}_{kl} $ is not symmetric, which means that $ \hat{\xi}_{kl} $ and $ \hat{\xi}_{lk} $ are not equal. A reasonable explanation is given in Remark 1 of  \cite{chatterjee2021new}. Actually, $ \xi_{kl}$ seeks to determine whether $ X_l $ is a measure function of $ X_k $, and vice versa. Therefore, based on (\ref{xidefinition}), we construct the following quadratic statistic,
$$ T_{np}=\sum_{k\neq l}^{p} \hat{\xi}_{kl}^2 .$$
After being re-decomposed and combined, $ T_{np} $ can present a symmetric form, see  Appendix \ref{appendix} for details. We first provide the exact expectation and variance of $\hat{\xi}_{kl}$ under $ H_0 $ in the following lemma.
\begin{lemma} 	\label{lemma2.1}
	 Denote $ u_n=\E(\hat{\xi}_{kl}^2) $, $v_n^2=\Var(\hat{\xi}_{kl}^2).$  Under $ H_0 $,   it has 
	$$u_n=\dfrac{(n-2)(4n-7)}{10(n-1)^2(n+1)},$$
    $$v_n^2=\dfrac{224n^5-1792n^4+5051n^3-4969n^2-2458n+18128}{700(n-1)^4(n+1)^3}.$$
\end{lemma}

Under $ H_0 $, since $\hat{\xi}_{kl}$ is obtained by ordering, $ \hat{\xi}_{kl} $ and $ \hat{\xi}_{km}$ ($ k\neq l\neq m $) are independent. Similarly, there are also $ \hat{\xi}_{lk} $ and $ \hat{\xi}_{mk} $ ($ k\neq l\neq m $), $ \hat{\xi}_{kl} $ and $ \hat{\xi}_{mq} $ ($k\neq l\neq m \neq q$). It should be noted that $ \hat{\xi}_{kl}^2 $ and $ \hat{\xi}_{lk}^2 $ are not independent. Next, we provide the exact variance of $ \hat{\xi}_{kl}^2 $ and $ \hat{\xi}_{lk}^2 $  as well as the variance of $ T_{np} $, although it can be seen from Appendix \ref{appendix} that their derivation may seem cumbersome.
\begin{lemma} \label{lemma2.2}
	  Under $ H_0 $, denote $ \mu_{np}=\E(T_{np}) $, $ \sigma_{np}^2 = \Var(T_{np}) $, then
	  $$\mu_{np}=p(p-1)u_n $$ and
	 $$\Cov(\hat{\xi}_{kl}^2,\hat{\xi}_{lk}^2)=\dfrac{(n-2)(784n^5-8022n^4+27301n^3-24228n^2-5045n-44070)}{50n(n+1)^4(n-1)^5}.$$
	 Furthermore,
	\begin{eqnarray*}
	 &&\sigma_{np}^2 =p(p-1)\Var\left(\hat{\xi}_{12}^2\right)+p(p-1)\Cov(\hat{\xi}_{12}^2,\hat{\xi}_{21}^2)=p(p-1)\times\\
	 &&\dfrac{(224 n^8- 1792 n^7+ 15803 n^6- 137437 n^5 + 599321 n^4 - 1080523 n^3 + 610212 n^2- 493848 n +1233960)}{700 n(1 + n)^4(n-1)^5 }.\end{eqnarray*}
\end{lemma}
\begin{remark} \label{remark1}
   It is imperative to emphasize that the expectations, variances, and covariance specified in the aforementioned two lemmas are accurate.Therefore, they can be reliably utilized for the construction of the test statistic. Later simulations reveal that our empirical test size approximates the significance level closely, particularly in scenarios involving smaller sample sizes and dimensions.
\end{remark}

Although the calculation of the  variance  of $ T_{np} $ exhibits a certain degree of hardship, subtly, after some restructuring, the normalized $ T_{np} $ follows a standard normal distribution  by applying the classical central limit theorem.
\begin{theorem}  \label{theorem2.1}
Under $ H_0 $, no matter sample size $ n $ is fixed or tends to infinity, as long as $ p\xrightarrow{}\infty$, standardized $ T_{np} $ follows a standard normal distribution, that is,
	$$J_\xi=\dfrac{T_{np}-\mu_{np}}{\sigma_{np}}=\sigma_{np}^{-1}\sum_{k\neq l}^{p} (\hat{\xi}_{kl}^2-u_n) \xrightarrow{d}N(0,1),$$
where $ \xrightarrow{d} $ denotes convergence in distribution.
\end{theorem}
\begin{remark} \label{remark2}
	The conditions stated in Theorem  \ref{theorem2.1} are exceptionally lenient for testing the independence of high-dimensional data. With the specified significance level, we can directly derive the critical value for the quadratic test.
\end{remark}

By integrating Theorem \ref{theorem2.1} with the following theorem, the consistency of the proposed test method can be achieved under certain conditions.
\begin{theorem} \label{theorem2.2} For all $ 1\leqslant k \neq l\leqslant p $, define the set of indexes $ (k,l) $ of all nonzero elements in $ \boldsymbol{\xi}_p\in \boldsymbol{\Xi}_a $ as $ \mathcal{M}.$ Suppose that there is a constant $c_0>0$ such that for all  $(k,l)\in\mathcal{M}$,  $ \xi_{kl}>c_0 .$  Let $ M $ be the cardinality of set $ \mathcal{M} $. Under $ H_{a} $, as $ n,p\xrightarrow{}\infty $ and $\frac{nM}{p}\xrightarrow{}\infty $, 
$$ \P(J_\xi>z_q)\xrightarrow{}1,$$
where $ z_q$  is the upper $ q $th quantile of the standard normal distribution for significance level $ q \in(0,1)$.
\end{theorem}

\begin{remark} \label{remark3}
	Note that $ M  $ can be treated as the number of pairs with Chatterjee coefficient of non-zero in all $ p (p-1) $ pairs of variables. In fact, from the perspective of Theorem \ref{theorem2.2}, in the setting of high dimensions, where $ n/p $ tends to 0 as $ n,p\xrightarrow{}\infty $, achieving high power requires that $M$ increases at a higher rate than $ p/n $, hence the alternative set formed in this way presents a certain dense pattern. 
\end{remark}

The ensuing corollary demonstrates that if the signal is excessively weak, the proposed test statistic $J_\xi$ becomes invalid.

\begin{corollary}  \label{corollary2.1} For all $ 1\leqslant k \neq l\leqslant p $, define the set of indexes $ (k,l) $ of all nonzero elements in $ \boldsymbol{\xi}_p\in \boldsymbol{\Xi}_a $ as $ \mathcal{M}.$ Suppose that there is a constant $c_0>0$ such that for all  $(k,l)\in\mathcal{M}$,  $ \xi_{kl}>c_0 .$  Let $ M $ be the cardinality of set $ \mathcal{M} $. Under $ H_{a} $, if $ M $ is excessively small, such as $M=o(p/n) $ or $ M $ is fixed as $n,p\xrightarrow{}\infty $, it has 	
$$\lim _{n,p\xrightarrow{}\infty}\P(J_\xi>z_q)\leqslant q.$$
\end{corollary}

\section{Extreme value test} \label{section3}
When there are too few non-zero elements in $ \boldsymbol{\xi}_p $, i.e.,  $ \boldsymbol{\xi}_p $ belongs to sparse alternative hypotheses set, quadratic statistics will exhibit difficulties in high-dimensional independence test, typically encountering low power problems. The main reason is that under $ H_0 $, a large amount of estimation errors are accumulated in the quadratic statistics for high-dimensional data, which leads to excessively large critical values dominating the signal under  sparse alternative hypotheses.

A considerable number of scholars have formulated extreme value statistics to address the aforementioned challenges. In our approach, we explore an alternative form of extreme value statistic by substituting the Chatterjee correlation coefficient for the general correlation coefficient as
$$ L_{n p}=\max _{1 \leqslant k\neq l \leqslant p}\left|\hat{\xi}_{k l}\right|.$$

Theorem \ref{theorem3.1} provides that adjusted $ L_{n p}$ converges weakly to a Gumbel distribution  under null hypothesis. 
\begin{theorem} \label{theorem3.1}
	 Under $ H_{0} $, for any  $ y \in \mathbb{R} $, as  $ n,p \rightarrow \infty $, $ \operatorname{P}(  L_{n p}^{2} / u_{n}-c_p \leqslant y ) \rightarrow\exp \left(-e^{-y / 2} / \sqrt{8 \pi}\right) $.
	where $ c_p=4 \log ( \sqrt{2}p) -\log \log ( \sqrt{2}p) $ and $u_n$ is presented in Lemma \ref{lemma2.1}.
\end{theorem}

The following theorem indicates that the test based on $ L_{np} $ is consistent under the alternatives $ H_a $.
\begin{theorem} \label{theorem3.2}
	 For all $ 1\leqslant k \neq l\leqslant p $, define the set of indexes $ (k,l) $ of all nonzero elements in $ \boldsymbol{\xi}_p\in \boldsymbol{\Xi}_a $ as $ \mathcal{M}.$ Suppose that there is a constant $c_0>0$ such that for all  $(k,l)\in\mathcal{M}$,  $ \xi_{kl}>c_0 .$ Under $ H_{a} $, as $ n,p\xrightarrow{}\infty $ and $\frac{\log p}{n}\rightarrow0 $, $ L_{n p} $ has high power, that is,
	
	$$\P( L_{n p}^{2} / u_{n}-c_p>z'_q)\xrightarrow{}1,$$
	where  $ z'_q $ is the upper $ q $th quantile of the Gumbel distribution for  significance level $ q \in(0,1)$.
\end{theorem}

From Theorem \ref{theorem3.1} and Theorem \ref{theorem3.2}, we can conclude that the test method based on $ L_{n p}$ is consistent under mild conditions.

\section{Power enhancement test}\label{section4}
 Most existing extreme value tests exhibit good performance under sparse alternatives, however, they require either bootstrap or strict conditions to derive the limit null distribution which results in typically slow convergence rates and size distortions.  
 To address the aforementioned issues,  we further invoke  a novel technique first proposed by \cite{fan2015power} for high-dimensional testing problems named the power enhancement test. The power enhancement test does not reduce the testing power under dense alternatives, but  enhances the power under sparse alternatives.
 
 Assuming $ J_\xi $ is the test statistic with the correct asymptotic test size, but encounters low power under sparse alternatives. The test power can be enhanced by adding a screening component $ J_0 \geqslant 0 $, which satisfies the following three power enhancement properties:
 \begin{enumerate}[(a)]
 \item   Nonnegativity: $ J_0 \geqslant 0 $ a.s.. \label{propertya}
 \item No-size-distortion: $ \P(J_0 = 0|H_0)\xrightarrow{}1 $. \label{propertyb}
 \item Power enhancement: $ J_0 $ diverges in probability under some specific sparse alternative regions. \label{propertyc}
\end{enumerate}

Similar to  \cite{fan2015power}, we construct  the power enhancement test  in  the following form,
 \begin{eqnarray*}  
 	J_E = J_0 + J_\xi.
 \end{eqnarray*}
The pivot statistic $ J_\xi $ is a quadratic test statistic presented in Section \ref{section2} with the correct asymptotic size  and has high power against $ H_0 $ in a certain dense alternative region $ \boldsymbol{\Xi}( J_\xi) $. The screening component $ J_0 $ does not serve as a test statistic and is added  just to enhance test power. 

Property (\ref{propertya}) ensures that the test power of $ J_E $, due to the addition of a nonnegative component, will not be lower than that of $ J_\xi $. Property (\ref{propertyb}) indicates that the addition of this nonnegative component tending towards 0 will not affect the limit null distribution, and further will not cause size distortion. This also fully demonstrates that the limit distribution of $ J_E $ under the null hypothesis is asymptotically the same as that of $ J_\xi $, which does not require additional costs to derive the limit null distribution of $ J_E $. Property (\ref{propertyc}) shows that there will be a significant improvement for the test power in certain alternative regions where $ J_0 $ diverges.

Next, we need to establish relevant requirements for our estimator. Choose  a  high-criticism threshold $\delta_{np}$ such that, under both null and  alternative hypotheses, for any  $\boldsymbol{\xi}_p \in \boldsymbol{\Xi}$,  
\begin{eqnarray}\label{4.1}
	\inf _{\boldsymbol{\xi}_p \in \boldsymbol{\Xi}} \P\left(\max _{1\leqslant k\neq l \leqslant p}\left|\hat{\xi}_{kl}-\xi_{kl}\right| / u_n^{1 / 2}<\delta_{np} \mid \boldsymbol{\xi}_p\right) \rightarrow 1,
\end{eqnarray}
where $u_n>0$ is a normalizing constant  and taken as the variance of $\hat{\xi}_{kl}$ from Lemma \ref{lemma2.1} in Section \ref{section2}. The sequence $\delta_{np}$, depending on $(n, p)$, grows slowly as $n, p \rightarrow \infty$ and is chosen to dominate the maximum noise level $\max _{1\leqslant k\neq l \leqslant p}\frac{\left| \hat{\xi}_{kl}\right|}{\sqrt{u_n}}$.

The following lemma indicates that (\ref{4.1}) holds under mild conditions.
\begin{lemma}\label{lemma4.1}
	  Assume  that the joint distribution function $F_{kl}$ $( 1\leqslant k \neq l\leqslant p) $ of any bivariate $ (X_k, X_l) $ from $ \boldsymbol{\xi}_p $ is continuous. If there exist fixed constants $\beta, C, C_1, C_2>0$ such that for any $t \in \mathbb{R}$ and $x, x^{\prime} \in \mathbb{R}$,
$$ 	  \left|\P(X_l \geqslant t \mid X_k=x)-\P\left(X_l \geqslant t \mid X_k=x^{\prime}\right)\right| \leqslant C\left(1+|x|^\beta+\left|x^{\prime}\right|^\beta\right)\left|x-x^{\prime}\right| $$
	 and
$$ \P(|X_k| \geqslant t) \leqslant C_1 e^{-C_2 t}, 1\leqslant k \neq l\leqslant p,  $$	 
	 then, for any $ \boldsymbol{\xi}_p \in \boldsymbol{\Xi} $ with  $ \boldsymbol{\Xi}=\boldsymbol{\Xi}_a\cup\{\boldsymbol{0}\}$,
	$$ \inf _{\boldsymbol{\xi}_p \in \boldsymbol{\Xi}} \P\left(\max _{1\leqslant k\neq l \leqslant p}\left|\hat{\xi}_{kl}-\xi_{kl}\right| / u_{n}^{1 / 2}<\delta_{np} \mid \boldsymbol{\xi}_p\right) \rightarrow 1. $$
\end{lemma}
\begin{remark} 
\begin{enumerate} [(1)]
\item The two requirements stipulated in Lemma \ref{lemma4.1} for the distribution, namely the Lipschitz condition and tail probability, are notably lenient and can be met with considerable ease. Additionally,  variance $ u_n $ is accurate, not  relying on data for estimation, that is, its distribution is free.
\item (Statement on the selection of $\delta_{np}$) 
Throughout the article, the critical value $\delta_{np}$ takes $\delta_{np}=\sqrt{c_p}\log\log n$ with $ c_p=4 \log ( \sqrt{2}p) -\log \log ( \sqrt{2}p) $   presented in Theorem \ref{theorem3.1} in which $\max _{1\leqslant k\neq l \leqslant p}\frac{\left| \hat{\xi}_{kl}\right|}{\sqrt{u_n}}=O_p\left( \sqrt{c_p}\right) $. We choose to use the exact $ c_p $  instead of  $ \log p $ in \cite{fan2015power} in order to avoid biased results for screening in small samples. When $ n $ and $ p $ are large enough, there is no difference in their effects. The selection of $ \log\log n $ is to ensure that $ \delta_{np} $ grows slowly enough and slightly larger than $ c_p $.  In fact, other eligible options are also allowable.
\end{enumerate}
\end{remark} 

Before formally presenting $ J_0 $, we first define a screening set
\begin{eqnarray}\label{estimationS}
\hat{S}\left(\boldsymbol{\hat{\xi}}_p \right) =\left\{(k,l):\left|\hat{\xi}_{kl}\right|>u_n^{1 / 2} \delta_{np}, 1\leqslant k\neq l \leqslant p\right\},
\end{eqnarray}
with a population counterpart as
\begin{eqnarray}\label{populationS}
S\left(\boldsymbol{\xi}_p\right) =\left\{(k,l):\left|\xi_{kl}\right|>2u_n^{1 / 2} \delta_{np}, 1\leqslant k\neq l \leqslant p\right\}.
\end{eqnarray}
Obviously, under $ H_0 $, $ S\left(\boldsymbol{\xi}_p\right)=\emptyset $ , and according to Lemma \ref{lemma4.1}, 
 $$\P\left( \hat{S}\left(\boldsymbol{\hat{\xi}}_p \right)=\emptyset\right)=\P\left( \max _{1\leqslant k\neq l \leqslant p}\left| \hat{\xi}_{kl}\right|/\sqrt{u_n}<\delta_{np}\mid H_0\right)\rightarrow1.$$ 
 Moreover, if $ S\left(\boldsymbol{\xi}_p\right)\neq\emptyset $, for any  $(k,l)\in S\left(\boldsymbol{\xi}_p\right)$, $ \left|\xi_{kl}\right|>2u_n^{1 / 2} \delta_{np} $ implies $ \left|\hat{\xi}_{kl}\right|>u_n^{1 / 2} \delta_{np} $,  thus,
 $ S(\boldsymbol{\xi}_p) \subset \hat{S}\left(\boldsymbol{\hat{\xi}}_p \right) $ with high probability,  these will be formally presented in Theorem \ref{theorem4.1}.

Based on the above analysis, we  construct our screening statistic  $ J_0 $ as 
$$
J_0=\sqrt{p(p-1)} \sum_{(k,l) \in \hat{S}\left(\boldsymbol{\hat{\xi}}_p \right)} \frac{\hat{\xi}_{kl}^2}{u_n}=\sqrt{p(p-1)}\sum_{k\neq l}^{p}\frac{\hat{\xi}_{kl}^2}{u_n}I\left\lbrace\frac{\left| \hat{\xi}_{kl}\right|}{\sqrt{u_n}}>\delta_{np}  \right\rbrace.
$$ 
From the definition of $ J_0 $, it can be seen that $ J_0\geqslant0 $ clearly satisfies power enhancement property  (\ref{propertya}): Nonnegativity. We further set the sparse alternative set
\begin{eqnarray}\label{sparsedef}
 \boldsymbol{\Xi}_s =\left\lbrace \boldsymbol{\xi}_p\in \boldsymbol{\Xi }_a:\max_{1\leqslant k\neq l \leqslant p} \left|\xi_{kl}\right|>2u_n^{1 / 2} \delta_{np}\right\rbrace.
\end{eqnarray}
From the above discussion, the power of $ J_\xi $ is enhanced on $ \boldsymbol{\Xi}_s$ or its subset due to the addition of $ J_0 $. $  \hat{S}\left(\boldsymbol{\hat{\xi}}_p \right) $ not only reveals the alternative structure of $ \boldsymbol{\Xi}_s$, but also identifies non-zero elements, which is beneficial for us to screen out the desired dependent variables in simulation or practical applications.

The following theorem provides properties related to screening statistic  $ J_0 $ and screening set $ \hat{S}\left(\boldsymbol{\hat{\xi}}_p \right) $.

\begin{theorem} \label{theorem4.1}
	Assuming the conditions of Lemma \ref{lemma4.1} hold, as $ n, p \rightarrow \infty $, it has,
	\renewcommand{\theenumi}{\roman{enumi}}
	\renewcommand{\labelenumi}{(\theenumi)}
	\begin{enumerate}
	\item \label{theorem4.1(i)} under  $ H_{0}: \boldsymbol{\xi}_p=\mathbf{0}, \P\left(\hat{S}\left(\boldsymbol{\hat{\xi}}_p \right)=\emptyset \mid H_{0}\right) \rightarrow 1 $, hence
	$ \P\left(J_{0}=0 \mid H_{0}\right) \rightarrow 1 $.
	
	\item \label{theorem4.1(ii)} for any nonempty sparse alternative set $ \boldsymbol{\Xi}_s $ defined in (\ref{sparsedef}),
	$$\inf _{\boldsymbol{\xi}_p \in \boldsymbol{\Xi}_s} \P\left(J_0>\sqrt{p(p-1)} \mid \boldsymbol{\boldsymbol{\xi}_p}\right) \rightarrow 1.  $$
	
	\item \label{theorem4.1(iii)}  for any $ \boldsymbol{\xi}_p \in \boldsymbol{\Xi} $ with  $ \boldsymbol{\Xi}=\boldsymbol{\Xi}_a\cup\{\boldsymbol{0}\}$, $ \hat{S}\left(\boldsymbol{\hat{\xi}}_p\right)  $ and $ S\left( \boldsymbol{\xi}_p\right)  $ are defined in (\ref{estimationS}) and (\ref{populationS}), respectively,
	 $$ 
	\inf _{\boldsymbol{\xi}_p \in \boldsymbol{\Xi}} \P(S\left( \boldsymbol{\xi}_p\right) \subset \hat{S}\left(\boldsymbol{\hat{\xi}}_p \right) \mid \boldsymbol{\xi}_p) \rightarrow 1.$$
	\end{enumerate}
	
\end{theorem}
\begin{remark} 
\begin{enumerate} [(1)] 
   \item 	According to Theorem 4.1(\ref{theorem4.1(i)}),  $J_0$ satisfies power enhancement property  (\ref{propertyb}): No-size-distortion. Theorem 4.1(\ref{theorem4.1(ii)}) is a naturally expected result under sparse alternative signals. Theorem 4.1(\ref{theorem4.1(iii)}) guarantees all the significant signals are	contained in  $ \hat{S}(\boldsymbol{\hat{\xi}}_p ) $ with a high probability. However, it should be noted that if $ \boldsymbol{\xi}_p \in \boldsymbol{\Xi}_a $ but $ \boldsymbol{\xi}_p  \not\in \boldsymbol{\Xi}_s $, then $ S(\boldsymbol{\xi}_p)=\emptyset $ and $ \hat{S}(\boldsymbol{\hat{\xi}}_p )= \emptyset$ with high  probability and Theorem 4.1(\ref{theorem4.1(iii)}) also applies. If given some mild conditions, it can be further shown that  $ \P(\hat{S}(\boldsymbol{\hat{\xi}}_p )= S(\boldsymbol{\xi}_p)\mid \boldsymbol{\xi}_p) \rightarrow 1 $ uniformly in $ \boldsymbol{\xi}_p $. Hence the selection is consistent. Since the consistency of the selection is not a requirement for our power enhancement test, we no longer focus on its consistency.
   \item The sparse alternative set $ \boldsymbol{\Xi}_s $ defined in (\ref{sparsedef}) requires at least one component in $\boldsymbol{\xi}_p $ to have magnitude greater than $ 2u_n^{1 / 2} \delta_{np} $, and similar settings are also considered in Section 2.2 of \cite{fan2015power}, Section 4.2 of \cite{drton2020high} and Section 4.1 of \cite{han2017distribution}.
\end{enumerate}
\end{remark} 

The following theorem establishes the claimed  properties of power enhancement test.

\begin{theorem} \label{theorem4.2}
Assuming the conditions of Lemma \ref{lemma4.1} hold, as $n, p \rightarrow \infty$, the power enhancement test $J_E=$ $J_0+J_\xi$ admits the following properties: 
\renewcommand{\theenumi}{\roman{enumi}}
\renewcommand{\labelenumi}{(\theenumi)}
\begin{enumerate}
\item \label{theorem4.2(i)} Under   $H_0, J_E \xrightarrow{d} N(0,1)$.
\item \label{theorem4.2(ii)} The dense alternative set $ \boldsymbol{\Xi}\left(J_\xi\right) $  is defined as follows, for  a positive constant  $ C $,
\begin{eqnarray}\label{densedef}
	\boldsymbol{\Xi}( J_\xi)=\left\lbrace\boldsymbol{\xi}_p\in \boldsymbol{\Xi}_a:\sum_{k\neq l}^{p} \xi_{kl}^2\geqslant Cp^2u_n\delta_{np}^2\right\rbrace,
\end{eqnarray}
then $J_E$ has high power uniformly on the set
$ \boldsymbol{\Xi}_s \cup \boldsymbol{\Xi}\left(J_\xi\right)$, that is, $$\inf _{\boldsymbol{\xi}_p \in \boldsymbol{\Xi}_s \cup \boldsymbol{\Xi}\left(J_\xi\right)} P\left(J_E>z_q \mid \boldsymbol{\xi}_p\right) \rightarrow 1.$$ 
\end{enumerate}
\end{theorem}

\begin{remark} 
\begin{enumerate} [(1)]  
\item Theorem 4.2(\ref{theorem4.2(i)}) demonstrates that the addition of screening statistic does not change the limiting distribution of quadratic statistic under the null hypothesis and furthermore does not cause any size distortion for large sample size. Additionally, Theorem 4.2(\ref{theorem4.2(ii)}) ensures that the power for the power enhancement test does not decrease and tends to 1 after the addition of sparse alternative hypothesis. In essence, in high-dimensional tests, for quadratic statistic $ J_\xi $, $ \boldsymbol{\Xi}\left(J_\xi\right) $ is already a uniformly high power region, and the relevant proof is deferred to Lemma A.2 in Appendix \ref{appendix}. With the addition of $ J_0 $, the region is expanded to $ \boldsymbol{\Xi}_s \cup \boldsymbol{\Xi}\left(J_\xi\right) $. Obviously, due to the increase in the range of alternative regions, the power of $ J_\xi $ is naturally enhanced, which satisfies power enhancement property (\ref{propertyc}).

\item In contrast to the sparse alternative set  $ \boldsymbol{\Xi}_s $, the dense alternative set $\boldsymbol{\Xi}( J_\xi) $ defined in (\ref{densedef}) suggests that, for certain constant  $ C $, the magnitude of each non-zero component in $\boldsymbol{\xi}_p $, numbering in the order of $ O(p^2) $, surpasses $ u_n^{1 / 2} \delta_{np} $. This suggestion is indicative of a certain degree of density within set $\boldsymbol{\Xi}( J_\xi) $, which stands out in contrast to the sparsity of set $ \boldsymbol{\Xi}_s $.

\end{enumerate} 
\end{remark} 

\section{Simulations}\label{section5}
In this section, we generate synthetic data to investigate the performance of our proposed approaches, namely, the quadratic statistic $ J_\xi $ in Section \ref{section2}, the adjustment form of extreme value statistic  $ L_{np} $ (called $ M_\xi $) in Section \ref{section3}, i.e., $ M_\xi=L_{n p}^{2} / u_{n}-c_p $, and the power enhancement test $ J_E $ in Section \ref{section4}.

The following tests are used for comparison. Four types of quadratic tests are   $ S_r $ (\cite{Schott2005testing}), $ S_\rho $ (\cite{mao2017robust}), $ S_\tau $ (\cite{mao2018testing})  and $ S_\varphi $ (\cite{shi2023max}),  based on Pearson correlation, Spearman’s $ \rho $, Kendall’s $\tau$  and Spearman’s footrule, respectively. Three types of extreme value tests are $M_\rho $ (\cite{han2017distribution}),   $ M_\tau $ (\cite{han2017distribution})  and $ M_\varphi $ (\cite{shi2023max}),   based on Spearman’s $ \rho $, Kendall’s $\tau$ and Spearman’s footrule, respectively. In addition, we also add a comparison with the rank-based test ($ C_\varphi $) proposed by \cite{shi2023max} via minimizing the $p$-values of $ S_\varphi $  and $ M_\varphi $.

 Three examples including 12 models   are considered to generate the synthetic data from $\boldsymbol{X}=\left(X_1, \ldots, X_p\right)^{\top}$ with dimensions $p=100,200,400,800$  and sample sizes $n=50,100$. The four models in Example 1 mainly generate data under the null hypotheses to verify the validity of the proposed test methods.  The distributions of the four models include standard normal distribution and non normal heavy tailed distribution. Example 2 generates data under  dense alternative hypotheses, including linear, nonlinear and oscillatory dependent data. Example 3 generates data under various sparse alternative hypotheses. The significance level for each test is $q=0.05$. For  each model, we perform 1000 independent replicates. In the following, with a slight abuse of notation, we write $f(\boldsymbol{w})=$ $\left(f\left(w_1\right), \ldots, f\left(w_p\right)\right)^{\top}$ for any univariate function $f: \mathbb{R} \rightarrow \mathbb{R}$ and $\boldsymbol{w}=\left(w_1, \ldots, w_p\right)^{\top} \in \mathbb{R}^p$.  
 The empirical size in Example 1 is presented in Table \ref{table1}.  The empirical rejection rates of Example 2 and Example 3 are shown in Table \ref{table2} and Table  \ref{table3}, respectively. In addition, to test the ability of our screening statistic to screen out variables, we also record the frequency of the set $ \hat{S} $  being a nonempty set.

\subsection*{Example 1. Data generating under $ H_0 $}\label{example5.1}
\begin{enumerate}[(a)]
\item \label{example1(a)} $\boldsymbol{X} \sim  N_p\left(0, \mathbf{I}_p\right)$, where $\textbf{I}_p$ is the identity matrix of dimension $p$.
	
\item \label{example1(b)} $\boldsymbol{X}=\boldsymbol{W}^3$ with $\boldsymbol{W} \sim N_p\left(0, \mathbf{I}_p\right)$.
	
\item \label{example1(c)} $X_1, \ldots, X_p \stackrel{\text { i.i.d }}{\sim} Cauchy(0,1)$.
	
\item \label{example1(d)} $X_1, \ldots, X_p\stackrel{\text { i.i.d }}{\sim} t(3)$ which is a $t$-distribution with 3 degrees of freedom.
	
\end{enumerate}

\subsection*{Example 2. Data generating under dense alternative hypotheses} \label{example2}  
\begin{enumerate}[(a)]
\item \label{example2(a)} $ \boldsymbol{X} \sim N_p\left(0, \Sigma_\rho\right)$ with $\Sigma_\rho=\rho \textbf{I}_p+(1-\rho) \textbf{e}_p$, where $\textbf{e}_p$  is a $p\times p  $ matrix with all entries being 1, $\rho=0.1$.

\item \label{example2(b)} $ \boldsymbol{X} \sim N_p\left(0, \Sigma_\rho\right)$ with $\Sigma_\rho=\left(\sigma_{i j}\right)_{p \times p}$, $ \sigma_{i j}=\rho^{|i-j|} $ for $1 \leqslant i,j \leqslant p$, $\rho=0.3$.

\item\label{example2(c)} $ \boldsymbol{X}=\boldsymbol{V}+0.4\boldsymbol{U}$, $\boldsymbol{V}=\left(\boldsymbol{W^{\top}},\sin(2\pi \boldsymbol{W})^{\top},\cos(2\pi\boldsymbol{W})^{\top},\sin(4\pi\boldsymbol{W})^{\top},\cos(4\pi \boldsymbol{W})^{\top} \right)^{\top} $, where $ \boldsymbol{W} \sim N_{p/5}\left(0, \textbf{I}_{p/5}\right)$, and noise vector $\boldsymbol{U} \sim N_{p}\left(0, \textbf{I}_{p}\right) $ is independent of $ \boldsymbol{V} $ and $ \boldsymbol{W} $.

\item \label{example2(d)} $\boldsymbol{X}=\left(\boldsymbol{W^{\top}},\log( \boldsymbol{W^2})^{\top}+3\boldsymbol{V}^{\top} \right)^{\top} $, where $ \boldsymbol{W}$ and $\boldsymbol{V}$ are mutually independent and both from $ N_{p/2}\left(0, \textbf{I}_{p/2}\right)$.
\end{enumerate}  
 
\subsection*{Example 3. Data generating under sparse alternative hypotheses} \label{example3}   
\begin{enumerate}[(a)]
\item \label{example3(a)}  $\boldsymbol{X}\sim N_p\left( 0, \Sigma\right) $ with $\Sigma=\left(\sigma_{i j}\right)_{p \times p}, \sigma_{11}=$ $\cdots=\sigma_{p p}=1, \sigma_{12}=\sigma_{21}=2.7\sqrt{\frac{\log p}{n}}$ and the remaining elements being 0.

\item \label{example3(b)} \textbf{(Quadratic)} $\boldsymbol{X}=\left(U,V,\boldsymbol{W}^{\top}\right)^{\top}   $, with  $U=V^2+Z/3$, $ \left( V,\boldsymbol{W^{\top}}\right)^{\top}\sim N_{p-1}\left(  0, I_{p-1}\right),$  and  $Z\sim N(0,1)$.

\item \label{example3(c)} \textbf{(W-shaped)} $\boldsymbol{X}=\left(U,V,\boldsymbol{W}^{\top}\right)^{\top}  $, with  $U=|V+0.5|I(V<0)+|V-0.5|I(V\geqslant0)$ and $ \left( V,\boldsymbol{W^{\top}}\right)^{\top}\sim N_{p-1}\left(  0, I_{p-1}\right) $ independent of $U$.

\item \label{example3(d)} (\textbf{Sinusoid}) $\boldsymbol{X}=\left(U,V,\boldsymbol{W}^{\top}\right)^{\top} $, with  $  U=\cos\left(2\pi V \right)+\lambda\varepsilon$, $ \varepsilon \sim N(0,1) $ independent of $U$, $V$ and $\boldsymbol{W}$,  $ \left( V,\boldsymbol{W^{\top}}\right)^{\top}\sim N_{p-1}\left(  0, I_{p-1}\right) $ independent of $U$. In the setting, $ \lambda $ controls oscillatory strength and $ 0\leqslant\lambda\leqslant 1 $, as $ \lambda $ decreases, the oscillation becomes stronger, and set $ \lambda=0.05 $.
\end{enumerate}  

From Table \ref{table1}, we can see that all quadratic tests except   $ S_r $ can effectively control the size even when the sample size is relatively small, while $ S_r $ only exhibits normal empirical size in Example 1(\ref{example1(a)}) which is specifically designed for Gaussian models. As for all extreme value tests, due to their slow convergence to the extreme distribution, their empirical size distortion occurs as expected.

Table \ref{table2} presents the simulation results for Example 2(\ref{example2(a)})-2(\ref{example2(d)}).  In Example 2(\ref{example2(a)})-2(\ref{example2(b)}), the random vector under dense alternative model exhibits linear dependence. As expected,  $ \mathcal{L}_{2} $-type test methods would perform well. Surprisingly, test method  $ C_\varphi $  performs as well as  $ \mathcal{L}_{2} $-type test methods. Although slightly inferior to them, our proposed $ \mathcal{L}_{2} $-type test still maintains a certain level of linear detection ability. Compared with  $ \mathcal{L}_{2} $-type test methods,  all  $ \mathcal{L}_\infty $-type test methods perform bad under all circumstances we considered. For small sample size $n$, the empirical powers of  $M_\rho$ and $ M_\xi$  fall  below the given significance level, and even surprisingly near zero. In Example 2(\ref{example2(c)})-2(\ref{example2(d)}), the  proposed three tests outperform all the other tests, performing exceptionally well, even for the proposed extreme value test ($ M_\xi $) which  has a  high power when the sample size is relatively large ($ n=100 $). Surprisingly,  existing test methods are invalidated in Example 2(\ref{example2(c)})-2(\ref{example2(d)}). It should be emphasized that our screening set exhibits a certain degree of strength in terms of frequency of nonempty, especially in Example 2(\ref{example2(d)}), but has little effect on power enhancement. This is because our quadratic test has already shown favourable power under dense alternative hypotheses, while the screening statistic only shows significant performance under  sparse alternative hypotheses.

 Table \ref{table3}   presents the simulation results for the sparse alternatives. In the setting of sparse linearly dependent case in Example 3(\ref{example3(a)}), all extreme value tests show good performance as expected, while the proposed methods still hold a position behind them. In Example  3(\ref{example3(c)}) and Example  3(\ref{example3(d)}), specifically designed for sparse oscillatory models, the proposed extreme test wins with  marginal superiority. Additionally, it is worth noting that our screening statistic screens out a large amount of dependent signals  and the power of the quadratic test has been greatly enhanced, while all the remaining tests have almost lost their testing ability.
\begin{table}[htbp]
	\centering
\addtolength{\tabcolsep}{-3pt}
\caption{Empirical size for Example 1(a)-(d)}
\label{table1}
	\begin{tabular}{cccccccccccccc}
		\toprule
				$n$     & $p$     & $ S_r $    & $ S_\rho $  & $ S_\tau $  & $ S_\varphi $ & $ M_\rho $  & $ M_\tau $  & $ M_\varphi $ & $ C_\varphi $ & $ J_\xi $   &$  M_\xi $   & $ J_E $    & $ P (\hat{S}\neq\emptyset) $ \\
		\midrule
		\multicolumn{14}{c}{Example 1(\ref{example1(a)})} \\
		$ n $=50  & $ p $=100 & 0.045  & 0.048  & 0.043  & 0.054  & 0.003  & 0.015  & 0.025  & 0.036  & 0.040  & 0.029  & 0.040  & 0.000  \\
		& $ p $=200 & 0.061  & 0.054  & 0.054  & 0.054  & 0.003  & 0.012  & 0.020  & 0.045  & 0.052  & 0.027  & 0.052  & 0.000  \\
		& $ p $=400 & 0.047  & 0.039  & 0.037  & 0.043  & 0.002  & 0.012  & 0.014  & 0.033  & 0.063  & 0.013  & 0.063  & 0.000  \\
		& $ p $=800 & 0.059  & 0.057  & 0.060  & 0.050  & 0.000  & 0.007  & 0.019  & 0.030  & 0.055  & 0.018  & 0.055  & 0.000  \\
		&       &       &       &       &       &       &       &       &       &       &       &       &  \\
		$ n $=100 & $ p $=100 & 0.058  & 0.048  & 0.048  & 0.041  & 0.016  & 0.025  & 0.034  & 0.041  & 0.050  & 0.033  & 0.050  & 0.000  \\
		& $ p $=200 & 0.055  & 0.046  & 0.046  & 0.046  & 0.016  & 0.032  & 0.038  & 0.046  & 0.041  & 0.030  & 0.041  & 0.000  \\
		& $ p $=400 & 0.049  & 0.047  & 0.049  & 0.048  & 0.016  & 0.029  & 0.036  & 0.045  & 0.051  & 0.024  & 0.051  & 0.000  \\
		& $ p $=800 & 0.051  & 0.045  & 0.038  & 0.047  & 0.007  & 0.026  & 0.032  & 0.041  & 0.050  & 0.033  & 0.050  & 0.000  \\
		\multicolumn{14}{c}{} \\
		\multicolumn{14}{c}{Example 1(\ref{example1(b)})} \\
		$ n $=50  & $ p $=100 & 0.228  & 0.048  & 0.051  & 0.053  & 0.003  & 0.018  & 0.025  & 0.044  & 0.059  & 0.029  & 0.059  & 0.000  \\
		& $ p $=200 & 0.214  & 0.042  & 0.048  & 0.046  & 0.001  & 0.015  & 0.016  & 0.029  & 0.052  & 0.013  & 0.052  & 0.000  \\
		& $ p $=400 & 0.170  & 0.062  & 0.055  & 0.059  & 0.004  & 0.018  & 0.018  & 0.044  & 0.053  & 0.024  & 0.053  & 0.000  \\
		& $ p $=800 & 0.221  & 0.064  & 0.061  & 0.066  & 0.001  & 0.008  & 0.028  & 0.049  & 0.046  & 0.017  & 0.046  & 0.000  \\
		&       &       &       &       &       &       &       &       &       &       &       &       &  \\
		$ n $=100 & $ p $=100 & 0.196  & 0.048  & 0.046  & 0.046  & 0.029  & 0.042  & 0.048  & 0.052  & 0.047  & 0.039  & 0.047  & 0.000  \\
		& $ p $=200 & 0.227  & 0.049  & 0.049  & 0.048  & 0.013  & 0.030  & 0.029  & 0.035  & 0.055  & 0.039  & 0.055  & 0.000  \\
		& $ p $=400 & 0.227  & 0.043  & 0.045  & 0.041  & 0.013  & 0.031  & 0.032  & 0.037  & 0.045  & 0.036  & 0.045  & 0.000  \\
		& $ p $=800 & 0.214  & 0.049  & 0.050  & 0.038  & 0.009  & 0.026  & 0.026  & 0.038  & 0.040  & 0.021  & 0.040  & 0.000  \\
		\multicolumn{14}{c}{} \\
		\multicolumn{14}{c}{Example 1(\ref{example1(c)})} \\
		$ n $=50  & $ p $=100 & 0.443  & 0.059  & 0.059  & 0.049  & 0.009  & 0.021  & 0.025  & 0.038  & 0.052  & 0.025  & 0.052  & 0.000  \\
		& $ p $=200 & 0.444  & 0.052  & 0.053  & 0.045  & 0.003  & 0.030  & 0.027  & 0.047  & 0.040  & 0.022  & 0.040  & 0.000  \\
		& $ p $=400 & 0.475  & 0.057  & 0.055  & 0.059  & 0.002  & 0.017  & 0.021  & 0.044  & 0.048  & 0.015  & 0.048  & 0.000  \\
		& $ p $=800 & 0.454  & 0.055  & 0.054  & 0.058  & 0.001  & 0.009  & 0.013  & 0.039  & 0.051  & 0.010  & 0.051  & 0.000  \\
		&       &       &       &       &       &       &       &       &       &       &       &       &  \\
		$ n $=100 & $ p $=100 & 0.602  & 0.050  & 0.047  & 0.056  & 0.018  & 0.037  & 0.030  & 0.044  & 0.043  & 0.044  & 0.043  & 0.000  \\
		& $ p $=200 & 0.583  & 0.042  & 0.047  & 0.049  & 0.014  & 0.027  & 0.027  & 0.033  & 0.048  & 0.034  & 0.048  & 0.000  \\
		& $ p $=400 & 0.590  & 0.049  & 0.047  & 0.051  & 0.016  & 0.035  & 0.042  & 0.046  & 0.046  & 0.021  & 0.046  & 0.000  \\
		& $ p $=800 & 0.589  & 0.055  & 0.060  & 0.063  & 0.011  & 0.026  & 0.034  & 0.052  & 0.042  & 0.027  & 0.042  & 0.000  \\
		\multicolumn{14}{c}{} \\
		\multicolumn{14}{c}{Example 1(\ref{example1(d)})} \\
		$ n $=50  & $ p $=100 & 0.080  & 0.046  & 0.046  & 0.038  & 0.010  & 0.027  & 0.031  & 0.037  & 0.052  & 0.026  & 0.052  & 0.000  \\
		& $ p $=200 & 0.092  & 0.054  & 0.058  & 0.048  & 0.006  & 0.021  & 0.027  & 0.043  & 0.051  & 0.026  & 0.051  & 0.000  \\
		& $ p $=400 & 0.089  & 0.045  & 0.043  & 0.046  & 0.002  & 0.021  & 0.020  & 0.037  & 0.048  & 0.013  & 0.048  & 0.000  \\
		& $ p $=800 & 0.086  & 0.048  & 0.041  & 0.063  & 0.000  & 0.009  & 0.017  & 0.039  & 0.048  & 0.013  & 0.048  & 0.000  \\
		&       &       &       &       &       &       &       &       &       &       &       &       &  \\
		$ n $=100 & $ p $=100 & 0.071  & 0.057  & 0.046  & 0.044  & 0.022  & 0.042  & 0.038  & 0.035  & 0.054  & 0.032  & 0.054  & 0.000  \\
		& $ p $=200 & 0.084  & 0.043  & 0.044  & 0.037  & 0.013  & 0.029  & 0.026  & 0.036  & 0.049  & 0.024  & 0.049  & 0.000  \\
		& $ p $=400 & 0.083  & 0.055  & 0.056  & 0.058  & 0.015  & 0.027  & 0.041  & 0.039  & 0.052  & 0.028  & 0.052  & 0.000  \\
		& $ p $=800 & 0.081  & 0.064  & 0.062  & 0.054  & 0.012  & 0.027  & 0.032  & 0.043  & 0.047  & 0.027  & 0.047  & 0.000  \\
		\bottomrule
	\end{tabular}%
	\label{tab:addlabel}%
\end{table}%

\begin{table}[htbp]
	\centering
\addtolength{\tabcolsep}{-3pt}
\caption{Rejection frequencies for Example 2(a)-(d) under dense alternative hypotheses}
\label{table2}
	\begin{tabular}{cccccccccccccc}
		\toprule
		$ n $     & $ p $     & $ S_r $    & $ S_\rho $  & $ S_\tau $  & $ S_\varphi $ & $ M_\rho $  & $ M_\tau $  & $ M_\varphi $ & $ C_\varphi $ & $ J_\xi $   &$  M_\xi $   & $ J_E $    & $ P (\hat{S}\neq\emptyset) $ \\
		\midrule
		\multicolumn{14}{c}{Example 2(\ref{example2(a)})} \\
		$ n $=50  & $ p $=100 & 1.000  & 1.000  & 1.000  & 1.000  & 0.049  & 0.145  & 0.274  & 1.000  & 0.412  & 0.027  & 0.412  & 0.000  \\
		& $ p $=200 & 1.000  & 1.000  & 1.000  & 1.000  & 0.030  & 0.144  & 0.266  & 1.000  & 0.770  & 0.032  & 0.770  & 0.000  \\
		& $ p $=400 & 1.000  & 1.000  & 1.000  & 1.000  & 0.013  & 0.111  & 0.279  & 1.000  & 0.962  & 0.039  & 0.962  & 0.000  \\
		& $ p $=800 & 1.000  & 1.000  & 1.000  & 1.000  & 0.006  & 0.097  & 0.282  & 1.000  & 0.999  & 0.036  & 0.999  & 0.000  \\
		&       &       &       &       &       &       &       &       &       &       &       &       &  \\
		$ n $=100 & $ p $=100 & 1.000  & 1.000  & 1.000  & 1.000  & 0.399  & 0.513  & 0.596  & 1.000  & 0.546  & 0.064  & 0.546  & 0.000  \\
		& $ p $=200 & 1.000  & 1.000  & 1.000  & 1.000  & 0.417  & 0.581  & 0.694  & 1.000  & 0.897  & 0.044  & 0.897  & 0.000  \\
		& $ p $=400 & 1.000  & 1.000  & 1.000  & 1.000  & 0.384  & 0.581  & 0.748  & 1.000  & 0.997  & 0.071  & 0.997  & 0.000  \\
		& $ p $=800 & 1.000  & 1.000  & 1.000  & 1.000  & 0.359  & 0.629  & 0.809  & 1.000  & 1.000  & 0.069  & 1.000  & 0.000  \\
		&       &       &       &       &       &       &       &       &       &       &       &       &  \\
		\multicolumn{14}{c}{Example 2(\ref{example2(b)})} \\
		$ n $=50  & $ p $=100 & 0.985  & 0.970  & 0.975  & 0.962  & 0.149  & 0.320  & 0.387  & 0.967  & 0.155  & 0.036  & 0.155  & 0.000  \\
		& $ p $=200 & 0.996  & 0.974  & 0.978  & 0.957  & 0.068  & 0.219  & 0.307  & 0.966  & 0.161  & 0.029  & 0.161  & 0.000  \\
		& $ p $=400 & 0.990  & 0.974  & 0.977  & 0.969  & 0.034  & 0.147  & 0.232  & 0.971  & 0.146  & 0.019  & 0.146  & 0.000  \\
		& $ p $=800 & 0.992  & 0.976  & 0.977  & 0.975  & 0.010  & 0.080  & 0.167  & 0.977  & 0.145  & 0.016  & 0.145  & 0.000  \\
		&       &       &       &       &       &       &       &       &       &       &       &       &  \\
		$ n $=100 & $ p $=100 & 1.000  & 1.000  & 1.000  & 1.000  & 0.986  & 0.995  & 0.991  & 1.000  & 0.308  & 0.077  & 0.308  & 0.000  \\
		& $ p $=200 & 1.000  & 1.000  & 1.000  & 1.000  & 0.964  & 0.988  & 0.993  & 1.000  & 0.312  & 0.056  & 0.312  & 0.000  \\
		& $ p $=400 & 1.000  & 1.000  & 1.000  & 1.000  & 0.934  & 0.985  & 0.980  & 1.000  & 0.280  & 0.052  & 0.280  & 0.000  \\
		& $ p $=800 & 1.000  & 1.000  & 1.000  & 1.000  & 0.905  & 0.984  & 0.986  & 1.000  & 0.292  & 0.047  & 0.292  & 0.000  \\
		&       &       &       &       &       &       &       &       &       &       &       &       &  \\
		\multicolumn{14}{c}{Example 2(\ref{example2(c)})} \\
		$ n $=50  & $ p $=100 & 0.045  & 0.047  & 0.044  & 0.061  & 0.007  & 0.025  & 0.022  & 0.047  & 0.269  & 0.148  & 0.269  & 0.001  \\
		& $ p $=200 & 0.045  & 0.046  & 0.043  & 0.046  & 0.002  & 0.027  & 0.017  & 0.046  & 0.274  & 0.106  & 0.274  & 0.000  \\
		& $ p $=400 & 0.048  & 0.039  & 0.039  & 0.047  & 0.004  & 0.020  & 0.028  & 0.050  & 0.262  & 0.088  & 0.262  & 0.000  \\
		& $ p $=800 & 0.043  & 0.045  & 0.039  & 0.042  & 0.000  & 0.006  & 0.015  & 0.035  & 0.241  & 0.081  & 0.241  & 0.000  \\
		&       &       &       &       &       &       &       &       &       &       &       &       &  \\
		$ n $=100 & $ p $=100 & 0.053  & 0.046  & 0.049  & 0.063  & 0.017  & 0.022  & 0.031  & 0.051  & 0.609  & 0.648  & 0.611  & 0.004  \\
		& $ p $=200 & 0.064  & 0.050  & 0.053  & 0.065  & 0.015  & 0.031  & 0.036  & 0.069  & 0.602  & 0.638  & 0.602  & 0.000  \\
		& $ p $=400 & 0.055  & 0.057  & 0.060  & 0.061  & 0.020  & 0.033  & 0.027  & 0.053  & 0.619  & 0.648  & 0.619  & 0.000  \\
		& $ p $=800 & 0.054  & 0.048  & 0.048  & 0.063  & 0.010  & 0.026  & 0.025  & 0.060  & 0.601  & 0.643  & 0.601  & 0.000  \\
		&       &       &       &       &       &       &       &       &       &       &       &       &  \\
		\multicolumn{14}{c}{Example 2(\ref{example2(d)})} \\
		$ n $=50  & $ p $=100 & 0.039  & 0.051  & 0.048  & 0.055  & 0.009  & 0.020  & 0.036  & 0.047  & 0.485  & 0.397  & 0.489  & 0.010  \\
		& $ p $=200 & 0.035  & 0.050  & 0.053  & 0.050  & 0.007  & 0.019  & 0.021  & 0.036  & 0.494  & 0.375  & 0.495  & 0.001  \\
		& $ p $=400 & 0.050  & 0.059  & 0.059  & 0.060  & 0.001  & 0.011  & 0.015  & 0.045  & 0.457  & 0.305  & 0.457  & 0.000  \\
		& $ p $=800 & 0.042  & 0.041  & 0.046  & 0.055  & 0.001  & 0.019  & 0.029  & 0.053  & 0.488  & 0.261  & 0.488  & 0.000  \\
		&       &       &       &       &       &       &       &       &       &       &       &       &  \\
		$ n $=100 & $ p $=100 & 0.058  & 0.053  & 0.057  & 0.050  & 0.022  & 0.037  & 0.036  & 0.048  & 0.939  & 0.976  & 0.939  & 0.021  \\
		& $ p $=200 & 0.043  & 0.047  & 0.049  & 0.058  & 0.021  & 0.035  & 0.041  & 0.065  & 0.950  & 0.991  & 0.950  & 0.005  \\
		& $ p $=400 & 0.044  & 0.045  & 0.048  & 0.054  & 0.013  & 0.023  & 0.024  & 0.053  & 0.944  & 0.985  & 0.944  & 0.002  \\
		& $ p $=800 & 0.055  & 0.048  & 0.048  & 0.058  & 0.010  & 0.027  & 0.032  & 0.053  & 0.952  & 0.996  & 0.952  & 0.000  \\
		\bottomrule
	\end{tabular}%
	\label{tab:addlabel}%
\end{table}%

\begin{table}[htbp]
\centering
\addtolength{\tabcolsep}{-3pt}
\caption{Rejection frequencies for Example 3(a)-(d) under sparse alternative hypotheses}
\label{table3}
	\begin{tabular}{cccccccccccccc}
		\toprule
			$ n $     & $ p $     & $ S_r $    & $ S_\rho $  & $ S_\tau $  & $ S_\varphi $ & $ M_\rho $  & $ M_\tau $  & $ M_\varphi $ & $ C_\varphi $ & $ J_\xi $   &$  M_\xi $   & $ J_E $    & $ P (\hat{S}\neq\emptyset) $ \\
		\midrule
		\multicolumn{14}{c}{Example 3(\ref{example3(a)})} \\
		$ n $=50  & $ p $=100 & 0.076  & 0.069  & 0.080  & 0.069  & 0.988  & 0.996  & 0.996  & 0.993  & 0.117  & 0.737  & 0.291  & 0.215  \\
		& $ p $=200 & 0.065  & 0.056  & 0.053  & 0.062  & 0.998  & 1.000  & 1.000  & 1.000  & 0.071  & 0.935  & 0.467  & 0.430  \\
		& $ p $=400 & 0.054  & 0.048  & 0.051  & 0.057  & 1.000  & 1.000  & 1.000  & 1.000  & 0.079  & 0.997  & 0.882  & 0.868  \\
		& $ p $=800 & 0.043  & 0.039  & 0.036  & 0.049  & 1.000  & 1.000  & 1.000  & 1.000  & 0.063  & 1.000  & 1.000  & 1.000  \\
		&       &       &       &       &       &       &       &       &       &       &       &       &  \\
		$ n $=100 & $ p $=100 & 0.059  & 0.058  & 0.064  & 0.057  & 0.924  & 0.951  & 0.932  & 0.906  & 0.066  & 0.189  & 0.066  & 0.000  \\
		& $ p $=200 & 0.051  & 0.047  & 0.050  & 0.053  & 0.965  & 0.978  & 0.964  & 0.953  & 0.068  & 0.230  & 0.071  & 0.003  \\
		& $ p $=400 & 0.059  & 0.052  & 0.054  & 0.052  & 0.977  & 0.986  & 0.979  & 0.974  & 0.052  & 0.288  & 0.054  & 0.002  \\
		& $ p $=800 & 0.061  & 0.053  & 0.057  & 0.066  & 0.986  & 0.993  & 0.990  & 0.983  & 0.058  & 0.379  & 0.060  & 0.002  \\
		&       &       &       &       &       &       &       &       &       &       &       &       &  \\
		\multicolumn{14}{c}{Example 3(\ref{example3(b)})} \\
		$ n $=50  & $ p $=100 & 0.056  & 0.056  & 0.052  & 0.050  & 0.011  & 0.034  & 0.038  & 0.049  & 0.099  & 0.963  & 0.733  & 0.708  \\
		& $ p $=200 & 0.051  & 0.046  & 0.050  & 0.059  & 0.003  & 0.017  & 0.021  & 0.047  & 0.057  & 0.932  & 0.552  & 0.534  \\
		& $ p $=400 & 0.055  & 0.055  & 0.052  & 0.053  & 0.000  & 0.011  & 0.018  & 0.031  & 0.061  & 0.894  & 0.387  & 0.343  \\
		& $ p $=800 & 0.052  & 0.058  & 0.061  & 0.048  & 0.001  & 0.011  & 0.018  & 0.037  & 0.055  & 0.844  & 0.233  & 0.188  \\
		&       &       &       &       &       &       &       &       &       &       &       &       &  \\
		$ n $=100 & $ p $=100 & 0.058  & 0.054  & 0.052  & 0.047  & 0.019  & 0.036  & 0.044  & 0.050  & 0.172  & 1.000  & 0.996  & 0.996  \\
		& $ p $=200 & 0.067  & 0.050  & 0.054  & 0.051  & 0.024  & 0.039  & 0.029  & 0.038  & 0.122  & 1.000  & 0.985  & 0.984  \\
		& $ p $=400 & 0.055  & 0.056  & 0.057  & 0.060  & 0.011  & 0.020  & 0.021  & 0.040  & 0.070  & 1.000  & 0.956  & 0.954  \\
		& $ p $=800 & 0.048  & 0.043  & 0.048  & 0.047  & 0.009  & 0.027  & 0.029  & 0.035  & 0.057  & 1.000  & 0.935  & 0.928  \\
		&       &       &       &       &       &       &       &       &       &       &       &       &  \\
		\multicolumn{14}{c}{Example 3(\ref{example3(c)})} \\
		$ n $=50  & $ p $=100 & 0.048  & 0.051  & 0.050  & 0.040  & 0.003  & 0.019  & 0.024  & 0.043  & 0.149  & 1.000  & 1.000  & 1.000  \\
		& $ p $=200 & 0.035  & 0.043  & 0.041  & 0.048  & 0.004  & 0.020  & 0.027  & 0.039  & 0.075  & 1.000  & 1.000  & 1.000  \\
		& $ p $=400 & 0.052  & 0.068  & 0.066  & 0.065  & 0.001  & 0.010  & 0.024  & 0.032  & 0.065  & 1.000  & 1.000  & 1.000  \\
		& $ p $=800 & 0.057  & 0.062  & 0.060  & 0.053  & 0.000  & 0.010  & 0.017  & 0.042  & 0.065  & 1.000  & 1.000  & 1.000  \\
		&       &       &       &       &       &       &       &       &       &       &       &       &  \\
		$ n $=100 & $ p $=100 & 0.067  & 0.063  & 0.065  & 0.040  & 0.020  & 0.037  & 0.030  & 0.036  & 0.459  & 1.000  & 1.000  & 1.000  \\
		& $ p $=200 & 0.057  & 0.048  & 0.049  & 0.064  & 0.020  & 0.036  & 0.042  & 0.057  & 0.197  & 1.000  & 1.000  & 1.000  \\
		& $ p $=400 & 0.051  & 0.051  & 0.048  & 0.050  & 0.013  & 0.029  & 0.032  & 0.037  & 0.100  & 1.000  & 1.000  & 1.000  \\
		& $ p $=800 & 0.046  & 0.057  & 0.052  & 0.057  & 0.016  & 0.039  & 0.036  & 0.057  & 0.074  & 1.000  & 1.000  & 1.000  \\
		&       &       &       &       &       &       &       &       &       &       &       &       &  \\
		\multicolumn{14}{c}{Example 3(\ref{example3(d)})} \\
		$ n $=50  & $ p $=100 & 0.046  & 0.061  & 0.061  & 0.057  & 0.010  & 0.025  & 0.026  & 0.033  & 0.096  & 1.000  & 0.957  & 0.949  \\
		& $ p $=200 & 0.051  & 0.047  & 0.052  & 0.050  & 0.003  & 0.014  & 0.019  & 0.035  & 0.054  & 1.000  & 0.818  & 0.806  \\
		& $ p $=400 & 0.041  & 0.045  & 0.046  & 0.057  & 0.000  & 0.015  & 0.016  & 0.040  & 0.048  & 0.999  & 0.525  & 0.501  \\
		& $ p $=800 & 0.048  & 0.049  & 0.049  & 0.057  & 0.000  & 0.012  & 0.015  & 0.030  & 0.061  & 0.991  & 0.241  & 0.194  \\
		&       &       &       &       &       &       &       &       &       &       &       &       &  \\
		$ n $=100 & $ p $=100 & 0.053  & 0.052  & 0.055  & 0.048  & 0.018  & 0.033  & 0.025  & 0.041  & 0.257  & 1.000  & 1.000  & 1.000  \\
		& $ p $=200 & 0.040  & 0.041  & 0.045  & 0.046  & 0.017  & 0.027  & 0.033  & 0.041  & 0.114  & 1.000  & 1.000  & 1.000  \\
		& $ p $=400 & 0.052  & 0.047  & 0.047  & 0.051  & 0.013  & 0.034  & 0.030  & 0.034  & 0.094  & 1.000  & 1.000  & 1.000  \\
		& $ p $=800 & 0.046  & 0.042  & 0.045  & 0.049  & 0.010  & 0.021  & 0.032  & 0.043  & 0.055  & 1.000  & 1.000  & 1.000  \\
		\bottomrule
	\end{tabular}%
	\label{tab:addlabel}%
\end{table}%

\section{Real Data}\label{section6}
Below are two real datasets, the leaf dataset and the gene transcription dataset, to illustrate the usefulness and effectiveness of the
proposed test methods.
 
\subsection{ Leaf dataset}
A database extracted from digital images of leaf specimens of plant species was considered by \cite{silva2013evaluation} for the evaluation of measures using discriminant analysis and hierarchical clustering. The database contains 16 plant leaf attributes and 171 samples  and is available at \url{http://archive.ics.uci.edu/ml/datasets/Leaf}. An automatic plant recognition system requires a set of discriminating different attributes which is further applied to training statistical models, for instance, generative models in neural networks, and   naturally leading to the issue of testing independence.  We are currently considering  the independence between 7 shape attributes and 7 texture attributes. We chose   the sixth plant species, which includes 8 samples and 14 attributes to implement the tests in Section \ref{section5}. The $p$-values of all tests are presented in Table \ref{table4}. Our power enhancement test $ J_E $ provided strong rejection, indicating a strong dependency relationship between these attributes.  However, the $\mathcal{L}_\infty $-type test methods  except  $M_\varphi$ reluctantly accepted the null hypothesis.   In addition, we further used screening statistics to inspect the desired variables, resulting a total of ten pairs of variables being selected.  We present five pairs of asymmetric relationships in Figure \ref{figure1}   and it can be seen that these attributes are exhibiting strong linearity. This indicates that, as pointed out in Section \ref{section5}, our proposed tests also have a certain ability to detect linear relationships.

\begin{table}[htbp]
	\centering
	\addtolength{\tabcolsep}{-2pt}
	\caption{The test $ p $-value of different methods for leaf dataset }
	\label{table4}
	\begin{tabular}{ccccccc}
		\toprule
		Test method & $ S_r $    & $ S_\rho $  & $ S_\tau $  & $ S_\varphi $ & $ M_\rho $  & $ M_\tau $ \\
		\midrule
		$ p $-value & $ 2.23\times10^{-12} $ & $ 1.43\times10^{-9} $ & $ 1.47\times10^{-16} $ & $ 1.06\times10^{-26} $     & $ 5.17\times10^{-1} $ & $ 5.79\times10^{-2} $ \\
		\midrule
		Test method & $ M_\varphi $ & $ C_\varphi $ & $ J_\xi $   & $ M_\xi $   & $ J_E $    &  \\
		\midrule
		$ p $-value & $ 6.81\times10^{-3} $ & $ 1.06\times10^{-26} $      & $ 1.61\times10^{-12} $ & $ 6.37\times10^{-2} $ & $ <1.06\times10^{-26} $    &  \\
		\bottomrule
	\end{tabular}%
	\label{tab:addlabel}%
\end{table}%

\begin{figure}[!htb]
	\centering
	\includegraphics[width=1.02\columnwidth]{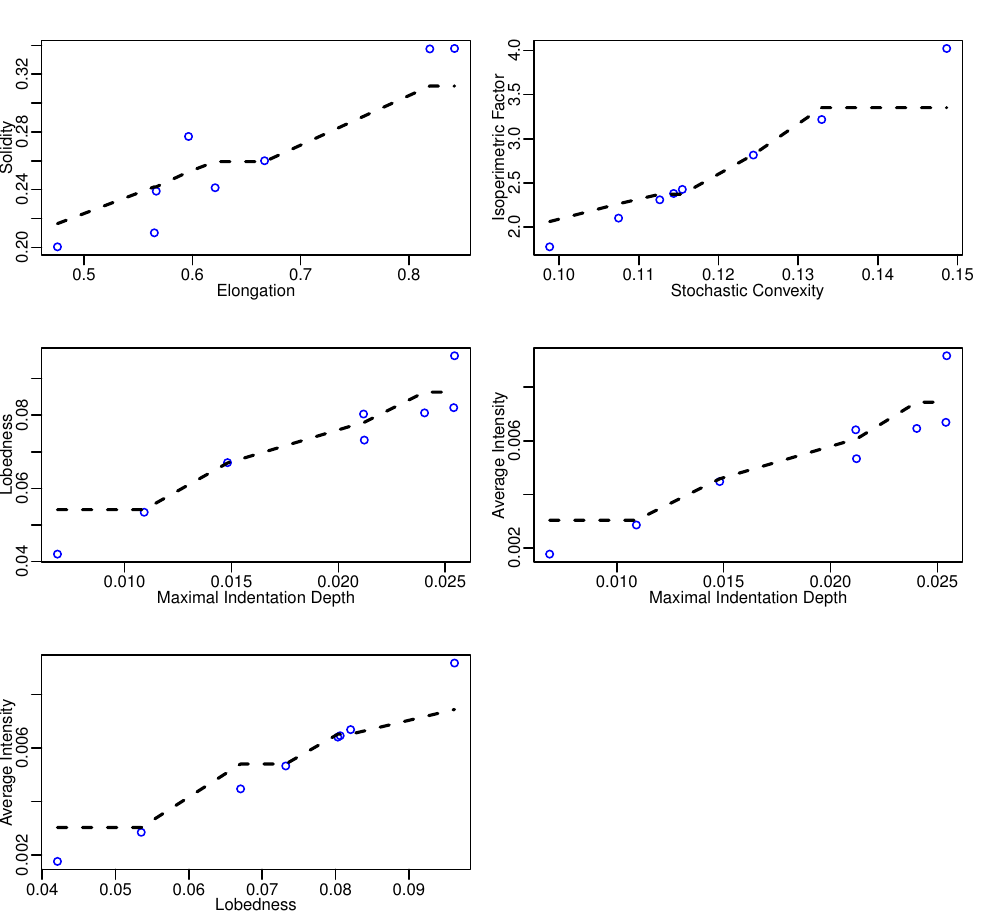}
	\caption{\textit{The dependency relationship of the selected five pairs of attributes in leaf dataset, with attributes as the names of the horizontal and vertical axes. The dashed line represents the curve fitted using the k-nearest neighbor method with k=3.} }
	\label{figure1}
\end{figure}

\subsection{Circadian gene expression dataset}

One oscillator presented in most mammalian tissues is the circadian clock, which allows organisms to align a great variety of rhythmic physiological behaviours between night and day, such as blood pressure, blood hormone levels, food consumption, metabolism  and locomotor activity. Disruption of normal circadian rhythms can lead to clinically numerous pathologies, including metabolic and cardiovascular disorders, neurodegeneration, aging  and cancer.  Remarkably, the circadian clock is judged to to be driven mainly by a transcription-translation feedback loop between several genes and proteins and capable of sustained oscillations outside of the body. These oscillating biological systems are being quantified using genomic techniques, but genomic data is high-dimensional and typically have many more features than observations. When predicting the periodic variable of the circadian clock, studying transcriptional rhythms or performing supervised learning on genomic data, we need to perform a dependence test on all features and identify rhythmic genes with oscillatory transcription levels whose intensity are associated with circadian time.

The circadian gene expression dataset (GSE11923) is from Gene Expression Omnibus (GEO) which is a public functional genomics data repository and data is available at  \url{https://www.ncbi.nlm.nih.gov/geo/}. We applied our proposed three methods and the  eight comparable methods to the dataset. 48 liver samples in GSE11923 were collected from 3-5 mice per time point every hour for 48 hours. For each sample, gene expression was measured for 17920 genes. The mice were initially entrained to a 12:12 hour light:dark schedule for one week, then released into constant darkness for 18 hours and took it as the starting point for the circadian time of the first sample, afterwards, the sample was collected every hour  and the termination time for the 48th sample was 65 hours.

The raw gene expression values were normalized using robust multi array average (RMA) (\cite{irizarry2003exploration}). Due to the overwhelming number of genes, we randomly selected 500 genes from 16649 genes as features (the genes with ties have been removed from the total of 17920 genes)  and added a clock periodic variable (circadian time) with a 24-hour oscillation. This resulted in a 48 $\times$ 501 high-dimensional data matrix for statistical testing. 

The $p$-value of all quadratic test statistics are 0, and the $p$-values of the remaining four extreme value tests $ M_\rho $, $ M_\tau $, $ M_\varphi $  and $ M_\xi $ are $ 1.262\times10^{-8} $, $ 3.170\times10^{-13} $, $ 2.665\times10^{-14} $  and $ 3.560\times10^{-12} $, respectively.  This indicates that there is a strong dependence relationship among all features. In addition, our screening statistic screened out 13 pairs of dependent features, of which seven pairs are related to clock variables. We show six of them in Figure \ref{figure2}.  The remaining seven  pairs exhibit strong oscillatory or linear relationships  and we will no longer showcase  here. From Figure \ref{figure2}, these genes exhibit strong clock oscillations, indicating that the high-dimensional power enhancement test based on Chatterjee coefficient does have excellent applicability for oscillation characteristics.

\begin{figure}[!htb]
	\centering
	\includegraphics[width=1.02\columnwidth]{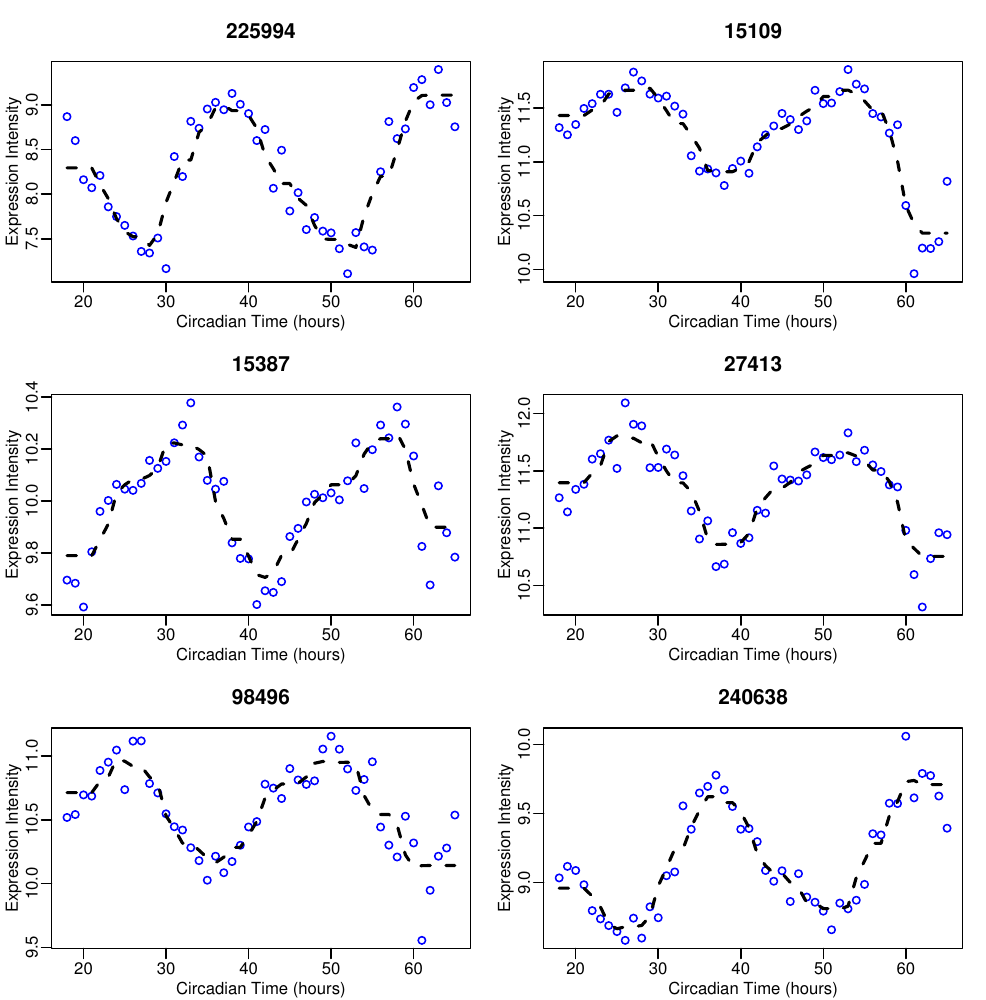}
	\caption{\textit{Gene transcriptional oscillation patterns changing with circadian time.  Each subgraph presents the gene ID and the dashed line represents the curve fitted using the k-nearest neighbor method with k=6.} }
	\label{figure2}
\end{figure}

\section{Discussion}\label{section7}
Chatterjee's rank correlation coefficient, as a newly developed coefficient, can be used to detect the nonlinear dependence between two scale  variables, especially oscillatory dependencies. The extension of Chatterjee's correlation coefficient to high-dimensional independence test is necessary and significant although the deriving of high-order moments and covariance is highly challenging. In this article, we focus on high-dimensional independence test based on Chatterjee's correlation coefficient. To this end,  we propose the quadratic test and extreme value test which  follow normal distribution and Gumbel distribution  under $ H_0 $, respectively. In addition, the consistencies of the proposed two test methods are established. However, the $L_2$-type and $L_\infty$-type statistics only perform well for dense and sparse alternative problems in high dimensions, respectively. In order to balance the drawbacks of these two tests, we add a screening statistic on the basis of the quadratic statistic  and propose a power enhancement test. Under mild conditions, we show that the resulting power is not lower than the quadratic test under   dense alternative hypotheses, and can be significantly enhanced under sparse alternative hypotheses. The proposed test methods are based on rank correlation and so that they have the following advantages:
 \begin{itemize}
\item[1.] The proposed tests are distribution free;
\item[2.] The proposed tests are more robust than tests based on the Pearson correlation coefficients; 
\item[3.] The proposed tests do  not include tuning parameters;
\item[4.] The proposed tests are  easy to implement using existing software programs like R, matlab, and so on. 
 \end{itemize}

\appendix
\section{Appendix}\label{appendix}
\begin{lemma}\label{lemmaA.1}
	Let $\Phi(\cdot)$ be the cumulative distribution function of standard normal distribution, under $ H_0 $, for any $k\neq l$ and $x\in  \mathbb{R}  $, there exists constant  $ C>0 $ such that
	$$ \sup_{x\in \mathbb{R}}\left|\P\left(\frac{1}{\sqrt{u_{n}}}\hat{\xi}_{kl} \leqslant x\right)-\Phi(x)\right|=C L_{n}^{1 /5}, $$
	where $ L_{n}=\dfrac{3(224n-225)}{14(4n-3)^2} . $
\end{lemma}

\begin{proof}[\textbf{Proof of Lemma \ref{lemmaA.1}.}]  Let $F_l(x)$ be the cumulative distribution function of  $X_l$.
	For presentation convenience, we hide the index $ kl $ and abbreviate the form of $ \hat{\xi}_{kl} $ as
	$$\xi_{n}=1- \dfrac{3}{n^2-1}\sum_{k=1}^{n-1}|R_k-R_{k+1}|.  $$	
	Denote $U_i=F_l(X_{il}),i=1,2,\ldots,n.$  Obviously,  $\{U_i\}_{i=1}^n$  are independent identically distributed from uniform distribution $ U(0,1) $ by Rosenblatt Transformation.   
	Similar to \cite{angus1995coupling}, under $ H_0 $, there exists an asymptotically equivalent representation for $ \xi_{n} $ as follows,
	$$\xi_{n}^{\prime}=-\frac{3}{n}\sum_{k=1}^{n-1}\left(|U_k-U_{k+1}|+2U_k(1-U_k)-\frac{2}{3} \right). $$
	
	Let
	$$\frac{1}{\sigma_{n}}\xi_{n}=\frac{1}{\sigma_{n}}\left(1- \dfrac{3}{n^2-1}\sum_{k=1}^{n-1}|R_k-R_{k+1}| \right)   $$
	and
	$$\frac{1}{\sigma_{n}'}\xi_{n}^{\prime}=-\frac{1}{\sigma''_{n}}\sum_{k=1}^{n-1}\left(|U_k-U_{k+1}|+2U_k(1-U_k)-\frac{2}{3} \right),   $$
	where $ \sigma_{n}^{2} $ and $ \sigma_{n}'^{2} $ are the variances of $ \xi_{n} $ and $ \xi_{n}^{\prime} $, respectively, $ \sigma_{n}^{2}=u_n=\dfrac{(n-2)(4n-7)}{10(n-1)^2(n+1)}$,  $ \sigma_{n}'^{2}=\dfrac{9}{n^2} \sigma_{n}''^{2} $, $ \sigma_{n}''^{2}=\dfrac{2}{45}n- \dfrac{1}{30}$. Then $\E\left( \frac{1}{\sigma_{n}}\xi_{n}\right) = \E\left( \frac{1}{\sigma_{n}'}\xi_{n}^{\prime}\right) =0, \Var\left( \frac{1}{\sigma_{n}}\xi_{n}\right) =\ \Var\left( \frac{1}{\sigma_{n}'}\xi_{n}^{\prime}\right) =1 $.
	
	Next, we use the relevant results of martingales in \cite{hall2014martingale} to obtain the Berry-Esseen bound for $  \dfrac{1}{\sigma_{n}'}\xi_{n}^{\prime}. $		
	Denote  $ \mathcal{F}_{l}=\sigma\left\{U_{1}, \ldots, U_{l}\right\}, l=1, \ldots, n , $ with $  \mathcal{F}_{0}=\emptyset  $, i.e.  $ \mathcal{F}_{l} $  is the $  \sigma $-field generated by  $ \left\{U_{1}, \ldots, U_{l}\right\} $. Set  $ Y_{l}=\mathrm{E}\left(\dfrac{\xi_{n}^{\prime}}{\sigma_{n}'} \Big| \mathcal{F}_{l}\right), Y_{0}=0, Z_{l}=Y_{l}-Y_{l-1},  V_{n}^{2}=\sum_{l=1}^{n} \mathrm{E}\left(Z_{l}^{2} \mid \mathcal{F}_{l-1}\right)  $.
	
	By simple calculation, one has
	$$  
	\frac{1}{\sigma_{n}'}\xi_{n}^{\prime}=\sum_{l=1}^{n}\left[\mathrm{E}\left(\frac{1}{\sigma_{n}'}\xi_{n}^{\prime} \mid \mathcal{F}_{l}\right)-\mathrm{E}\left(\frac{1}{\sigma_{n}'}\xi_{n}^{\prime} \mid \mathcal{F}_{l-1}\right)\right]=\sum_{l=1}^{n}\left(Y_{l}-Y_{l-1}\right)=\sum_{l=1}^{n} Z_{l}=Y_n  
	$$
	and
	$$E(Y_l|\mathcal{F}_{l-1})=E\left( \operatorname{E}(\xi_n''|\mathcal{F}_l)|\mathcal{F}_{l-1}\right)=\operatorname{E}(\xi_n''|\mathcal{F}_{l-1})=Y_{l-1} ,$$
	indicating $Y_l$ is a martingale. Further, $\operatorname{E}(Z_l|\mathcal{F}_{l-1})=\operatorname{E}(Y_l-Y_{l-1}|\mathcal{F}_{l-1})=0$, thus $Z_l$ is a martingale difference. 	For $2\leqslant l\leqslant n-1 $, one has
	$$   Y_{l}=\mathrm{E}\left(\frac{1}{\sigma_{n}'}\xi_{n}^{\prime} \mid \mathcal{F}_{l}\right)=-\frac{1}{\sigma_{n}''}\left( \sum_{k=1}^{l-1}\left[|U_k-U_{k+1}|+2U_k(1-U_k)-\frac{2}{3} \right] + U_l(1-U_l)-\frac{1}{6}\right) .$$
	For $l=1$ and $l=n$, we have 
	$$   Y_{1}=-\frac{1}{\sigma_{n}''}\left( U_1(1-U_1)-\frac{1}{6}\right),  Y_{n}=\frac{1}{\sigma_{n}'}\xi_{n}^{\prime} .$$
	In this way, we can further calculate the martingale difference $ Z_l $ through $ Y_l $.  For $2\leqslant l\leqslant n-1 $, 
	$$Z_{l}=Y_l-Y_{l-1}=-\frac{1}{\sigma_{n}''}\left(U_l(1-U_l)+U_{l-1}(1-U_{l-1})+|U_{l-1}-U_{l}|-\frac{2}{3} \right), \E\left( Z_{l}^4\right) =\dfrac{8}{1575}\frac{1}{\sigma_{n}''^4}. $$
	For $l=1$ and $l=n$, 
	$$Z_1=Y_1,\ \ \E\left( Z_{1}^4\right) =\dfrac{1}{15120}\frac{1}{\sigma_{n}''^4},$$
	and  $$Z_n=Y_n-Y_{n-1}=-\frac{1}{\sigma_{n}''}\left( U_{n-1}(1-U_{n-1})-\frac{1}{2}+ |U_{n-1}-U_{n}|\right),\ \ \E\left( Z_{n}^4\right) = \dfrac{3}{560}\frac{1}{\sigma_{n}''^4}.$$ 
	Further,
	\begin{eqnarray*} 
		\sum_{l=1}^{n}\E\left( Z_{l}^4\right)=\dfrac{1}{15120}\frac{1}{\sigma_{n}''^4}+\dfrac{8}{1575}\frac{1}{\sigma_{n}''^4}(n-2)+ \dfrac{3}{560}\frac{1}{\sigma_{n}''^4}=\left( \dfrac{8}{1575}n-\dfrac{179}{37800}\right) \frac{1}{\sigma_{n}''^4}=\dfrac{3(192n-179)}{14(4n-3)^2} .
	\end{eqnarray*} 	
	Denote $	V_{n}^{2}=\sum_{l=1}^{n} \mathrm{E}\left(Z_{l}^{2} \mid \mathcal{F}_{l-1}\right),$ one has 
	\begin{eqnarray*} 
		V_{n}^{2}&=& \frac{1}{\sigma_{n}''^2}\times\frac{1}{180}+ \frac{1}{\sigma_{n}''^2}\sum_{l=2}^{n-1}\left[ - \dfrac{4}{45} (-1 + 15 U_{l-1}^2 - 30 U_{l-1}^3 + 15 U_{l-1}^4)\right]  \\
		&+& \frac{1}{\sigma_{n}''^2}\times\dfrac{1}{12}(1 - 12 U_{n-1}^2 + 24 U_{n-1}^3 - 12 U_{n-1}^4).
	\end{eqnarray*} 
	It is easy to show that 
	\begin{eqnarray*} 
		\E \left( V_{n}^{2}\right) = \frac{1}{\sigma_{n}''^2}\times\frac{1}{180}+ \frac{1}{\sigma_{n}''^2}\times\dfrac{2}{45}(n-2) 
		+ \frac{1}{\sigma_{n}''^2}\times\frac{1}{20}=\frac{1}{\sigma_{n}''^2}\left( \dfrac{2}{45}n- \dfrac{1}{30}\right) =1 
	\end{eqnarray*} 
	and
	\begin{eqnarray*} 
		\Var \left( V_{n}^{2}\right) &=&\frac{1}{\sigma_{n}''^4}(n-2) \Var\left(  - \dfrac{4}{45} (-1 + 15 U_{l-1}^2 - 30 U_{l-1}^3 + 15 U_{l-1}^4)\right)\\
		&+& \frac{1}{\sigma_{n}''^4}\Var\left( \dfrac{1}{12}(1 - 12 U_{n-1}^2 + 24 U_{n-1}^3 - 12 U_{n-1}^4)\right)\\
		&=& \frac{1}{\sigma_{n}''^4}(n-2)\left(\dfrac{8}{2835}-\left( \dfrac{2}{45}\right)^2 \right)+\frac{1}{\sigma_{n}''^4}\left(\dfrac{1}{336} -\left(\dfrac{1}{20}^2 \right) \right) \\
		&=&\frac{1}{\sigma_{n}''^4}(n-2)\dfrac{4}{4725}+\dfrac{1}{2100}\frac{1}{\sigma_{n}''^4}=\frac{1}{\sigma_{n}''^4}\left( \dfrac{4}{4725}n-\dfrac{23}{18900}\right)=\dfrac{3(16n-23)}{7(4n-3)^2} .
	\end{eqnarray*} 
	Thus, based on the above calculation results, one has
	\begin{eqnarray*} 
		\E\left| V_n^2-1\right| ^2=\Var\left( V_n^2-1\right) +\left( \E\left(  V_n^2-1\right)  \right) ^2=\Var\left( V_n^2\right)=\dfrac{3(16n-23)}{7(4n-3)^2}.
	\end{eqnarray*} 
	
	Let $ L_{n}:=\sum_{i=1}^{n} E\left|X_{i}\right|^{4}+E\left|V_{n}^{2}-1\right|^{2},$
	one has  $ L_{n}=\dfrac{3(224n-225)}{14(4n-3)^2} . $
	According to Theorem 3.9 in \cite{hall2014martingale}, there exists constant  $ C>0 $ such that  for all  $ x \in \mathbb{R}$,  
	$$ 
	\left|\P\left(\frac{1}{\sigma_{n}'}\xi_{n}^{\prime} \leqslant x\right)-\Phi(x)\right| \leqslant C L_{n}^{1 /5}\left[1+|x|^{16 /5}\right]^{-1}\leqslant C L_{n}^{1 /5}. $$
	Then,
	$$ 
	\sup_{x\in \mathbb{R}}\left|\P\left(\frac{1}{\sigma_{n}'}\xi_{n}^{\prime} \leqslant x\right)-\Phi(x)\right|= C L_{n}^{1 /5}. $$
	
	Next, we will utilize the Berry-Esseen bound for $ \frac{1}{\sigma_{n}'}\xi_{n}^{\prime} $ and Lemma from page 228 of \cite{serfling2009approximation} to calculate the Berry-Esseen bound of $ \frac{1}{\sigma_{n}}\xi_{n} $. For positive constant sequence  $ {a_n} $,
	\begin{eqnarray*}   
		\sup_{x\in \mathbb{R}}\left|\P\left(\frac{1}{\sigma_{n}}\xi_{n} \leqslant x\right)-\Phi(x)\right|&=&\sup_{x\in \mathbb{R}}\left|\P\left(\frac{1}{\sigma_{n}'}\xi_{n}^{\prime}+\left( \frac{1}{\sigma_{n}}\xi_{n}-\frac{1}{\sigma_{n}'}\xi_{n}^{\prime}\right)  \leqslant x\right)-\Phi(x)\right|\\
		&=&C L_{n}^{1 /5}+O\left( a_n\right)+\P\left(\left|\frac{1}{\sigma_{n}}\xi_{n}-\frac{1}{\sigma_{n}'}\xi_{n}^{\prime}\right| >a_n \right)\\
		&\leqslant&C L_{n}^{1 /5} +O\left( a_n\right)+\dfrac{\E\left( \frac{1}{\sigma_{n}}\xi_{n}-\frac{1}{\sigma_{n}'}\xi_{n}^{\prime}\right) ^2}{a_n^2}. 
	\end{eqnarray*} 
	Now, we consider the  expectation of $ \dfrac{\xi_{n} }{\sigma_n} \dfrac{\xi_{n}' }{\sigma_n'} $.  Let $\left[R_{i}^{U}\right]_{i=1}^{n}$ be the ranks of $\left[U_{i}\right]_{i=1}^{n}$. Noting that $\left[R_{i}^{U}\right]_{i=1}^{n}$ and $\left[R_{i}\right]_{i=1}^{n}$ are equal under $ H_0 $, we remove the superscript $U$ of $\left[R_{i}^{U}\right]_{i=1}^{n}$ for notation simplicity in the following proof without causing any ambiguity. Similar to Lemma S8 in the supplementary material of \cite{lin2023boosting} and $ \operatorname{min}(a,b)=\dfrac{a+b-|a-b|}{2} $, it is easy to show that\\
	$ \operatorname{Cov}(R_1,U_1)=\dfrac{n-1}{12} $, \ $ \operatorname{Cov}(R_1,U_2)=-\dfrac{1}{12}   $, $ \operatorname{Cov}(R_1,U_1^2)=\dfrac{n-1}{12} $, \ $ \operatorname{Cov}(R_1,U_2^2)=-\dfrac{1}{12}   $, \ $ \operatorname{Cov}(|R_1-R_2|,|U_1-U_2|)=\dfrac{n-2}{18} $, \
	$ \operatorname{Cov}(|R_1-R_2|,|U_1-U_3|)=\dfrac{n-8}{180} $,\ $ \operatorname{Cov}(|R_1-R_2|,|U_3-U_4|)=-\dfrac{1}{45} $, \ $ \operatorname{Cov}(R_1,|U_1-U_2|)=0 $,
	$ \operatorname{Cov}(R_3,|U_1-U_2|)=0 $, \ $ \operatorname{Cov}(U_1,|R_1-R_2|)=0 $, \
	$ \operatorname{Cov}(U_3,|R_1-R_2|)=0 $, $ \operatorname{Cov}(U_1^2,|R_1-R_2|)=\dfrac{n-2}{180} $, $ \operatorname{Cov}(U_3^2,|R_1-R_2|)=-\dfrac{1}{90} $, $ \operatorname{Cov}(U_1(1-U_1),|R_1-R_2|)=-\dfrac{n-2}{180} $, $ \operatorname{Cov}(U_3(1-U_3),|R_1-R_2|)=\dfrac{1}{90} $. 	
	
	Based on above facts  and with simple calculation, it has
	\begin{eqnarray*} 
		&&\E\left( \dfrac{\xi_{n} }{\sigma_n}\dfrac{\xi_{n}' }{\sigma_n'}\right)=\Cov\left( \dfrac{\xi_{n} }{\sigma_n},\dfrac{\xi_{n}' }{\sigma_n'}\right)\\
		&=&\dfrac{1}{\sigma_n\sigma_n'}\dfrac{9}{n(n^2-1)}\Cov\left(\sum_{k=1}^{n-1}|R_k-R_{k+1}|,\ \sum_{k=1}^{n-1}\left( |U_k-U_{k+1}|+2U_k(1-U_k)\right)  \right)\\
		&=& \dfrac{1}{\sigma_n\sigma_n'}\dfrac{9}{n(n^2-1)}\Cov\left( \sum_{k=1}^{n-1}|R_k-R_{k+1}|,\ \sum_{k=1}^{n-1} |U_k-U_{k+1}|\right) \\
		&+&\dfrac{1}{\sigma_n\sigma_n'}\dfrac{9}{n(n^2-1)}2\Cov\left( \sum_{k=1}^{n-1}|R_k-R_{k+1}|,\ \sum_{k=1}^{n-1} U_k(1-U_k)\right)\\
		&:=& \dfrac{1}{\sigma_n\sigma_n'}\dfrac{9}{n(n^2-1)}V_1+\dfrac{1}{\sigma_n\sigma_n'}\dfrac{9}{n(n^2-1)}2V_2.
	\end{eqnarray*} 
	Straightforward computation yields
	\begin{eqnarray*} 
		&&V_1=\Cov\left( \sum_{k=1}^{n-1}|R_k-R_{k+1}|,\ \sum_{k=1}^{n-1} |U_k-U_{k+1}|\right)\\
		&=& \sum_{k=1}^{n-1}\Cov\left(|R_k-R_{k+1}|,\ \sum_{k=1}^{n-1} |U_k-U_{k+1}|\right)\\
		&=&2\times\left(\Cov\left(|R_1-R_2|,\ |U_1-U_2|\right) +\Cov\left(|R_1-R_2|,\ |U_2-U_3|\right) +(n-3)\Cov\left(|R_1-R_2|,\ |U_3-U_4|\right)  \right) \\
		&+&\sum_{k=2}^{n-2}\left(\Cov\left(|R_1-R_2|,\ |U_1-U_2|\right) +2\Cov\left(|R_1-R_2|,\ |U_2-U_3|\right) +(n-4)\Cov\left(|R_1-R_2|,\ |U_3-U_4|\right)  \right) \\
		&=&2\left( \dfrac{n-2}{18}+1\times\dfrac{n-8}{180}+(n-3)\times\left(- \dfrac{1}{45}\right) \right)+\sum_{k=2}^{n-2}\left( \dfrac{n-2}{18}+2\times\dfrac{n-8}{180}+(n-4)\times\left(- \dfrac{1}{45}\right) \right)\\
		&=&\dfrac{(n-2)(4n-7)}{90} 
	\end{eqnarray*} 
	and 
	\begin{eqnarray*} 
		&&V_2=\Cov\left( \sum_{k=1}^{n-1}|R_k-R_{k+1}|,\ \sum_{k=1}^{n-1} U_k(1-U_k)\right)\\
		&=&\sum_{k=1}^{n-1}\Cov\left(U_k(1-U_k),  \sum_{k=1}^{n-1}|R_k-R_{k+1}|\right)\\
		&=&\Cov\left(U_1(1-U_1), |R_1-R_2|\right)+(n-2)\Cov\left(U_3(1-U_3),\  |R_1-R_2|\right)\\
		&+&\sum_{k=2}^{n-1}\Cov\left(U_k(1-U_k),\  \sum_{k=1}^{n-1}|R_k-R_{k+1}|\right)\\
		&=&\Cov\left(U_1(1-U_1),\  |R_1-R_2|\right)+(n-2)\Cov\left(U_3(1-U_3),\  |R_1-R_2|\right)\\
		&+&\sum_{k=2}^{n-1}\left(2 \Cov\left(U_1(1-U_1),\  |R_1-R_2|\right)+(n-3)\Cov\left(U_3(1-U_3),\  |R_1-R_2|\right)\right) \\
		&=&-\dfrac{n-2}{180}+(n-2)\times\dfrac{1}{90}+(n-2)\left( 2\times\left( -\dfrac{n-2}{180}\right)  +(n-3)\dfrac{1}{90}\right) =-\dfrac{n-2}{180}.
	\end{eqnarray*} 
	Combining $ V_1 $ with $ V_2 $, we have
	\begin{eqnarray*} 
		\E\left( \dfrac{\xi_{n} }{\sigma_n}\dfrac{\xi_{n}' }{\sigma_n'}\right)
		&=&\dfrac{1}{\sigma_n\sigma_n'}\dfrac{9}{n(n^2-1)}V_1+\dfrac{1}{\sigma_n\sigma_n'}\dfrac{9}{n(n^2-1)}2V_2\\
		&=&\dfrac{1}{\sigma_n\sigma_n'}\dfrac{9}{n(n^2-1)}\times\dfrac{2(n-2)^2}{45} \\
		&=&\sqrt{1-\dfrac{72n^2-211n+149}{(4n-7)(4n-3)(n+1)}}.	 
	\end{eqnarray*} 
	Set $ a_n^3=\E\left( \frac{1}{\sigma_{n}}\xi_{n}-\frac{1}{\sigma_{n}}\xi_{n}^{\prime}\right) ^2, $  one has 
	$a_n^3=2-2\E\left( \dfrac{\xi_{n} }{\sigma_n}\dfrac{\xi_{n}' }{\sigma_n'}\right)= O\left(\frac{1}{n} \right) $ and $a_n=O\left(n^{-\frac{1}{3}}\right) .$
	
	Ultimately,
	$$ \sup_{x\in \mathbb{R}}\left|\P\left(\frac{1}{\sigma_{n}}\xi_{n}\leqslant x\right)-\Phi(x)\right|=C L_{n}^{1 /5}+O\left( a_n\right)=C L_{n}^{1 /5}. $$
\end{proof}

Lemma \ref{lemmaA.2} presents the uniformity of our quadratic test $ J_\xi $ in dense alternative region $ \boldsymbol{\Xi}( J_\xi) $.
\begin{lemma}\label{lemmaA.2} Under $H_a$, as $ n,p\xrightarrow{}\infty $, the quadratic test statistic $J_\xi $ has high power uniformly on $ \boldsymbol{\Xi}( J_\xi) $,that is, for any $ q\in(0,1) $,
	$$\inf_{\boldsymbol{\xi}_p\in\boldsymbol{\Xi}( J_\xi) } \P(J_\xi>z_q\mid\boldsymbol{\xi}_p)\xrightarrow{}1,$$
	where $ z_q $ is the $ q $th upper quantile of standard normal distribution for  significance level $ q $. 
\end{lemma}
\begin{proof}[\textbf{Proof of Lemma \ref{lemmaA.2}.}]
	For some constant $ C>1 $, define event $$  B =\left\lbrace\sum_{k\neq l}^{p} \left(\hat{\xi}_{kl}- \xi_{kl}\right) ^2< C^2p^2u_n\delta_{np}^2 \right\rbrace , $$ then by Lemma 4.1, $\inf_{\boldsymbol{\xi}_p \in \boldsymbol{\Xi}}\P\left(B\mid\boldsymbol{\xi}_p \right)\rightarrow1$. On the event $ B $, according to the H\"{o}lder inequality, we have uniformly in $ \boldsymbol{\xi}_p $, 
	\begin{eqnarray*} 
		\sum_{k\neq l}^{p} \left(\xi_{kl}- \hat{\xi}_{kl}\right) \xi_{kl} \leqslant \left( \sum_{k\neq l}^{p} \left(\hat{\xi}_{kl}- \xi_{kl}\right) ^2\right) ^{1/2}\left( \sum_{k\neq l}^{p} \xi_{kl}^2\right) ^{1/2}\leqslant  Cpu_n^{1/2}\delta_{np} \left( \sum_{k\neq l}^{p} \xi_{kl}^2\right) ^{1/2}.
	\end{eqnarray*}
	Therefore,  when $ \sum_{k\neq l}^{p} \xi_{kl}^2\geqslant 16C^2p^2u_n\delta_{np}^2 $,
	\begin{eqnarray*} 
		\sum_{k\neq l}^{p} \hat{\xi}_{kl}^2&=&\sum_{k\neq l}^{p}\left[\left(\hat{\xi}_{kl}- \xi_{kl}\right)^2+\xi_{kl}^2+2\left(\hat{\xi}_{kl}- \xi_{kl}\right)\xi_{kl}  \right]\\
		&\geqslant&\sum_{k\neq l}^{p}\left[\xi_{kl}^2+2\left(\hat{\xi}_{kl}- \xi_{kl}\right)\xi_{kl}  \right]\\ 
		&\geqslant& \sum_{k\neq l}^{p}\xi_{kl}^2-2Cpu_n^{1/2}\delta_{np}\left( \sum_{k\neq l}^{p}\xi_{kl}^2\right) ^{1/2} \\
		&\geqslant& \frac{1}{2} \sum_{k\neq l}^{p}\xi_{kl}^2.
	\end{eqnarray*}
	With $ \sum_{k\neq l}^{p} \xi_{kl}^2\geqslant 16C^2p^2u_n\delta_{np}^2 $, $ \sigma_{np}=O(\frac{p}{n}) $ and $ u_n=O(\frac{1}{n}) $, it further follows that
	\begin{eqnarray*} 
		\sup_{\boldsymbol{\xi}_p\in\boldsymbol{\Xi}( J_\xi)}\P\left(J_\xi<z_q \mid \boldsymbol{\xi}_p\right)&=&\sup_{\boldsymbol{\xi}_p\in\boldsymbol{\Xi}( J_\xi)} \P\left( \sum_{k\neq l}^{p} \hat{\xi}_{kl}^2<\sigma_{np}z_q+p(p-1)u_n \right) \\
		&\leqslant&\sup_{\boldsymbol{\xi}_p\in\boldsymbol{\Xi}( J_\xi)} \P\left(  \frac{1}{2} \sum_{k\neq l}^{p}\xi_{kl}^2<\sigma_{np}z_q+p(p-1)u_n\right)\\
		&\rightarrow0&. 
	\end{eqnarray*}
	Ultimately, $ 	\inf_{\boldsymbol{\xi}_p\in\boldsymbol{\Xi}( J_\xi)}\P\left(J_\xi>z_q\mid \boldsymbol{\xi}_p\right) \rightarrow1. $
\end{proof}

\begin{proof}[\textbf{Proof of Lemma 4.1.}]
	Using the Bonferroni inequality,  uniformly for $ \boldsymbol{\xi}_p \in \boldsymbol{\Xi} $,
	\begin{eqnarray*}
		\P\left(\max _{1\leqslant k\neq l \leqslant p}\left|\hat{\xi}_{kl}-\xi_{kl}\right| > u_{n}^{1 / 2}\delta_{np} \right)
		&\leqslant&\sum\limits_{ 1\leqslant k\neq l \leqslant p }  \P\left(\left|\hat{\xi}_{kl}-\xi_{kl}\right| > u_{n}^{1 / 2}\delta_{np} \right)\\ 
		&\leqslant& p(p-1)\max _{1\leqslant k\neq l \leqslant p} \P\left(\left|\hat{\xi}_{kl}-\xi_{kl}\right| > u_{n}^{1 / 2}\delta_{np} \right)\\
		&\leqslant& p(p-1)\max _{1\leqslant k\neq l \leqslant p} \P\left(\left|\hat{\xi}_{kl}-\E\hat{\xi}_{kl}\right| > u_{n}^{1 / 2}\delta_{np}/2 \right)\\
		&+& p(p-1)\max _{1\leqslant k\neq l \leqslant p} \P\left(\left|\E\hat{\xi}_{kl}-\xi_{kl}\right| > u_{n}^{1 / 2}\delta_{np}/2 \right).
	\end{eqnarray*}
	
	We first handle the first term on the right side of the above inequality.  Rewrite $ \hat{\xi}_{kl} $ as
	\begin{eqnarray*}
		\hat{\xi}_{kl}=1-\dfrac{3\sum^{n-1}_{i=1}|R_{[i+1]l}^k-R_{[i]l}^k|}{n^2-1}=1-\dfrac{3\sum^{n}_{i=1}|R_{il}-R_{N_k(i),l}|}{n^2-1}, 1\leqslant k \neq l\leqslant p,
	\end{eqnarray*}
	where $ N_k(i) $ is the unique index $ j $ such that $ X_{jk} $ is immediately to the right of $ X_{ik} $ if we arrange the $ X_k $'s in increasing order. If there is no such $ j $, set $ N_k(i)=i. $ 
	
	For each $  t \in \mathbb{R} $, let
	$ F_{l}^n(t):=\frac{1}{n} \sum_{i=1}^{n} I_{\left\{X_{il} \leqslant t\right\}}$ be the empirical  cumulative distribution function of $ X_l .$ Obviously $  R_{il}=nF_{l}^n\left(X_{il}\right).  $ Set $ G_{l}^n(t):=\frac{1}{n} \sum_{i=1}^{n} I_{\left\{X_{il} \geqslant t\right\}}$, $  L_{il}=nG_{l}^n\left(X_{il}\right) $, and define
	$$ 
	Q_{kl}^n:=\frac{1}{n} \sum_{i=1}^{n} \min \left\{F_{l}^n\left(X_{il}\right), F_{l}^n\left(X_{N_k(i),l}\right)\right\}-\frac{1}{n} \sum_{i=1}^{n} G_{l}^n\left(X_{il}\right)^{2} . $$
	$$S_{kl}^n :=\frac{1}{n^3}\sum_{i=1}^{n}L_{il}\left( n-L_{il}\right). $$ 
	Some elementary calculation shows  that 
	$$ \frac{Q_{kl}^n}{S_{kl}^n}=\hat{\xi}_{kl}+\frac{ R_{nl}- R_{1l}}{2n^2S_{kl}^n},\ \ \ \ S_{kl}^n =\frac{(n+1)(2n+1)}{6n^2} .$$ 
	Invoking Lemma A.11 in the supplementary material of \cite{chatterjee2021new}, for any $ n $ and any  $ t \geqslant 0 $, there is a positive universal constant $  C  $ such that 
	$$  
	\P\left(\left|Q_{kl}^n-\E\left(Q_{kl}^n\right)\right| \geqslant t\right) \leqslant 2 e^{-C n t^{2}},  $$
	which follows by the bounded difference concentration inequality  in \cite{mcdiarmid1989method}.
	
	Since $ S_{kl}^n $ converges to a nonzero constant, then 
	$$  
	\P\left(\left|\frac{Q_{kl}^n}{S_{kl}^n}-\E\left(\frac{Q_{kl}^n}{S_{kl}^n}\right)\right| \geqslant t\right) \leqslant 2 e^{-C n t^{2}} ,  $$ 
	which is further equivalent to 
	$$  
	\P\left(\left|\hat{\xi}_{kl}-\E\hat{\xi}_{kl}+\frac{ R_{nl}- R_{1l}}{2n^2S_{kl}^n}\right| \geqslant t\right) \leqslant 2 e^{-C n t^{2}} .  $$	
	With a simple derivation,
	\begin{eqnarray*}
		\P\left(\left|\hat{\xi}_{kl}-\E\hat{\xi}_{kl}+\frac{ R_{nl}- R_{1l}}{2n^2S_{kl}^n}\right| \geqslant t\right)&\leqslant&\P\left(\left|\hat{\xi}_{kl}-\E\hat{\xi}_{kl}\right|+\left|\frac{ R_{nl}- R_{1l}}{2n^2S_{kl}^n}\right| \geqslant t\right)\\
		&\leqslant&\P\left(\left|\hat{\xi}_{kl}-\E\hat{\xi}_{kl}\right| \geqslant t/2\right)+\P\left(\left|\frac{ R_{nl}- R_{1l}}{2n^2S_{kl}^n}\right| \geqslant t/2\right).
	\end{eqnarray*}
	Taking $ t=\delta_{np} u_n^{1 / 2}$ and noting that $ \left|\frac{ R_{nl}- R_{1l}}{2n^2S_{kl}^n}\right|=O_p(n^{-2})  $ and $ \delta_{np} u_n^{1 / 2}=O\left(\left(\dfrac{\log p}{n}\right)^{1/2}  \log\log n\right) $,  one has  $ \P\left(\left|\frac{ R_{nl}- R_{1l}}{2n^2S_{kl}^n}\right| \geqslant \delta_{np} u_n^{1 / 2}/2\right)=0$.
	Based on the above discussions, when $ n $ and $ p $ are sufficiently large,
	\begin{eqnarray*}
		\P\left(\left|\hat{\xi}_{kl}-\E\hat{\xi}_{kl}\right| \geqslant \delta_{np} u_n^{1 / 2}/2\right) \leqslant 2 e^{-C n \delta_{np}^2u_n}.
	\end{eqnarray*}
	Further verification reveals that as $ n,  p\xrightarrow{}\infty $,
	\begin{eqnarray*}
		p(p-1)\P\left(\left|\hat{\xi}_{kl}-\E\hat{\xi}_{kl}\right| > u_{n}^{1 / 2}\delta_{np} \right)\leqslant 2 p(p-1)e^{-C n \delta_{np}^2u_n}\rightarrow 0.
	\end{eqnarray*}
	
	Next, we consider the second term. By  Proposition 1.1 in  \cite{lin2022limit},  for any  constant  $\beta>0$ and some constants $ C $, one has 
	\begin{eqnarray*}
		\left|\mathrm{E}\hat{\xi}_{kl}-\xi_{kl}\right|\leqslant\frac{C(\log n)^{\beta+3}}{n}< u_{n}^{1 / 2}\delta_{np}/2 .
	\end{eqnarray*}
	Therefore, for large $n$, one has  
	$$  \P\left(\left|\E\hat{\xi}_{kl}-\xi_{kl}\right| > u_{n}^{1 / 2}\delta_{np}/2 \right)=0. $$ 
	Combining the above results, one has
	$$  \P\left(\max _{1\leqslant k\neq l \leqslant p}\left|\hat{\xi}_{kl}-\xi_{kl}\right| > u_{n}^{1 / 2}\delta_{np} \right)\rightarrow0. $$ 
	Ultimately, uniformly for $ \boldsymbol{\xi}_p \in \boldsymbol{\Xi}$,
	$$\P\left(\max _{1\leqslant k\neq l \leqslant p}\left|\hat{\xi}_{kl}-\xi_{kl}\right| / u_{n}^{1 / 2}<\delta_{np}\right) \rightarrow 1. $$
\end{proof}

\begin{proof}[\textbf{Proof of Lemma 2.1.}]
	For notational simplicity, we abbreviate the symbol $ R_{[i]l}^k $ to $R_i$. To avoid ambiguity, we only use it in the proofs of Lemma 2.1 and Lemma 2.2. Therefore, $ \hat{\xi}_{kl} $ can be rewritten as
	\begin{eqnarray*}
		\hat{\xi}_{kl}=1-\dfrac{3}{n^2-1}\sum^{n-1}_{i=1}A_i,
	\end{eqnarray*}
	where $A_i=|R_{i+1}-R_i|$.
	
	Next, we will handle the relevant results of $ \hat{\xi}_{kl} $ under $H_0$.
	Some elementary calculations show that
	\begin{eqnarray*}
		\E\left( \sum^{n-1}_{i=1}A_i\right)=\dfrac{n^2-1}{3}.
	\end{eqnarray*}
	Consequently, one has 
	$$\E(\hat{\xi}_{kl})=0.$$
	According to Lemma 2 in \cite{zhang2023asymptotic}, one has  $$\Var(\hat{\xi}_{kl})=\E(\hat{\xi}_{kl}^2)=\dfrac{(n-2)(4n-7)}{10(n-1)^2(n+1)}.$$
	With the help of this equation, we can further obtain
	\begin{eqnarray*}
		\E\left( \sum^{n-1}_{i=1}A_i\right)^2 =\dfrac{(n+1)(10n^3-6n^2-25n+24)}{90}.
	\end{eqnarray*}
	
	We  provide the following facts for calculating $ \E\left( \sum^{n-1}_{i=1}A_i\right)^3 $. Going forward, let $ C_n^m  $ represent the number of combinations of $ m $ elements selected from $ n $ elements, and $  A_n^m $ represent the number of permutations of $ m $ elements selected from $ n $ elements. By some calculations, we can derive that
	\begin{eqnarray*}
		\E\left(  A_1^3\right)=\frac{1}{A_n^2}\sum_{i_1\neq i_2}^n\left|i_1-i_2 \right|^3=\dfrac{(n+1)(3n^2-2)}{30} , 
	\end{eqnarray*}
	\begin{eqnarray*}
		\E \left( A_1^2A_2\right)=\frac{1}{A_n^3}\sum_{i_1\neq\dots\neq i_3}^n\left|i_1-i_2 \right|^2\left|i_2-i_3 \right|=\dfrac{(n+1)(11n^2+4n-6)}{180},  
	\end{eqnarray*}
	\begin{eqnarray*}
		\E \left( A_1^2A_3\right)=\frac{1}{A_n^4}\sum_{i_1\neq\dots\neq i_4}^n\left|i_1-i_2 \right|^2\left|i_3-i_4 \right|=\dfrac{(n+1)^2(5n-2)}{90} , 
	\end{eqnarray*}
	\begin{eqnarray*}
		\E\left( A_1A_2A_3\right)=\frac{1}{A_n^4}\sum_{i_1\neq\dots\neq i_4}^n\left|i_1-i_2 \right|\left|i_2-i_3 \right|\left|i_3-i_4 \right|=\dfrac{(n+1)(51n^2+59n-2)}{1260} , 
	\end{eqnarray*}
	\begin{eqnarray*}
		\E\left( A_1A_2A_4\right)=\frac{1}{A_n^5}\sum_{i_1\neq\dots\neq i_5}^n\left|i_1-i_2 \right|\left|i_2-i_3 \right|\left|i_4-i_5 \right|	=\frac{(n+1)(49n^2+59n+6)}{1260},
	\end{eqnarray*}
	\begin{eqnarray*}
		\E \left( A_1A_3A_5\right)=\frac{1}{A_n^6}\sum_{i_1\neq\dots\neq i_6}^n\left|i_1-i_2 \right|\left|i_3-i_4 \right|\left|i_5-i_6 \right| =\frac{(n+1)(35n^2+49n+12)}{945}.
	\end{eqnarray*}
	Based on these facts, we have
	\begin{eqnarray*}
		\E\left( \sum^{n-1}_{i=1}A_i\right)^3 
		&=&C_{n-1}^1\E\left(  A_1^3\right) +C_{n-2}^1C_{2}^1C_{3}^2\E \left( A_1^2A_2\right) +C_{n-2}^2C_{2}^1C_{3}^2\E \left( A_1^2A_3\right) \\
		&+&C_{n-3}^1A_3^3\ E\left( A_1A_2A_3\right) +C_{n-3}^2C_2^1A_3^3\ E\left( A_1A_2A_4\right) +C_{n-3}^3A_3^3\E\left(  A_1A_3A_5\right) \\
		&=&\dfrac{(n+1)(70n^5+14n^4-463n^3+397n^2+387n-612)}{1890}.
	\end{eqnarray*}
	
	Now, we consider  $ \E\left( \hat{\xi}_{kl}^4\right) $. 	Some elementary calculation  shows that	
	\begin{eqnarray*}
		\E\left(A_1A_3A_5A_7 \right)=\frac{1}{A_n^8}\sum_{i_1\neq\dots\neq i_8}^n\left|i_1-i_2 \right|\left|i_3-i_4 \right|\left|i_5-i_6 \right|\left|i_7-i_8 \right| 
		=\frac{(n+1)(5n+6)(35n^2+21n-8)}{14175},
	\end{eqnarray*}
	\begin{eqnarray*}
		\E\left(A_1A_2A_4A_6 \right)=\frac{1}{A_n^7}\sum_{i_1\neq\dots\neq i_7}^n\left|i_1-i_2 \right|\left|i_2-i_3 \right|\left|i_4-i_5 \right|\left|i_6-i_7 \right| 
		=\frac{(n+1)(245n^3+401n^2+40n-104)}{18900},
	\end{eqnarray*}
	\begin{eqnarray*}
		\E\left(A_1A_2A_4A_5 \right)=\frac{1}{A_n^6}\sum_{i_1\neq\dots\neq i_6}^n\left|i_1-i_2 \right|\left|i_2-i_3 \right|\left|i_4-i_5 \right|\left|i_5-i_6 \right|
		=\frac{(n+1)(3087n^3+4523n^2-218n-1392)}{226800}, 
	\end{eqnarray*}
	\begin{eqnarray*}
		\E\left(A_1A_2A_3A_5 \right)=\frac{1}{A_n^6}\sum_{i_1\neq\dots\neq i_6}^n\left|i_1-i_2 \right|\left|i_2-i_3 \right|\left|i_3-i_4 \right|\left|i_5-i_6 \right| 
		=\frac{(n+1)(153n^3+247n^2+2n-84)}{11340} , 
	\end{eqnarray*}
	\begin{eqnarray*}
		\E\left(A_1A_2A_3A_4 \right)=\frac{1}{A_n^5}\sum_{i_1\neq\dots\neq i_5}^n\left|i_1-i_2 \right|\left|i_2-i_3 \right|\left|i_3-i_4 \right|\left|i_4-i_5 \right| 
		=\frac{(n+1)(319n^3+477n^2-22n-156)}{22680},   
	\end{eqnarray*}
	\begin{eqnarray*}
		\E\left(A_1^2A_3A_5 \right) =\frac{1}{A_n^6}\sum_{i_1\neq\dots\neq i_6}^n\left|i_1-i_2 \right|^2\left|i_3-i_4 \right|\left|i_5-i_6 \right|
		=\frac{(n+1)(35n^3+35n^2-28n-24)}{1890} ,
	\end{eqnarray*}
	\begin{eqnarray*}
		\E\left(A_1^2A_3A_4 \right)=\frac{1}{A_n^5}\sum_{i_1\neq\dots\neq i_5}^n\left|i_1-i_2 \right|^2\left|i_3-i_4 \right|\left|i_4-i_5 \right|
		=\dfrac{(n+1)(49n^3+41n^2-42n-24)}{2520},
	\end{eqnarray*}
	\begin{eqnarray*}
		\E\left(A_1^2A_2A_4 \right)=\frac{1}{A_n^5}\sum_{i_1\neq\dots\neq i_5}^n\left|i_1-i_2 \right|^2\left|i_2-i_3 \right|\left|i_4-i_5 \right|
		=\dfrac{(n+1)(77n^3+63n^2-86n-60)}{3780} ,
	\end{eqnarray*}
	\begin{eqnarray*}
		\E\left(A_1A_2^2A_3 \right)=\frac{1}{A_n^4}\sum_{i_1\neq\dots\neq i_4}^n\left|i_1-i_2 \right|\left|i_2-i_3 \right|^2\left|i_3-i_4 \right|
		=\dfrac{(n+1)(14n^3+12n^2-17n-12)}{630},  
	\end{eqnarray*}
	\begin{eqnarray*}
		\E\left(A_1^2A_2A_3 \right)=\frac{1}{A_n^4}\sum_{i_1\neq\dots\neq i_4}^n\left|i_1-i_2 \right|^2\left|i_2-i_3 \right|\left|i_3-i_4 \right|
		=\dfrac{(n+1)(53n^3+45n^2-54n-32)}{2520} , 
	\end{eqnarray*}
	\begin{eqnarray*}
		\E\left(A_1^2A_2^2 \right)=\frac{1}{A_n^3}\sum_{i_1\neq\dots\neq i_3}^n\left|i_1-i_2 \right|^2\left|i_2-i_3 \right|^2
		=\dfrac{n(n+1)(2n^2-3)}{60}, 
	\end{eqnarray*}
	\begin{eqnarray*}
		\E\left(A_1^2A_3^2 \right)=\frac{1}{A_n^4}\sum_{i_1\neq\dots\neq i_4}^n\left|i_1-i_2 \right|^2\left|i_3-i_4 \right|^2
		=\dfrac{n(n+1)(n-1)(5n+6)}{180}, 
	\end{eqnarray*}
	\begin{eqnarray*}
		\E\left(A_1^3A_3 \right)=\frac{1}{A_n^4}\sum_{i_1\neq\dots\neq i_4}^n\left|i_1-i_2 \right|^3\left|i_3-i_4 \right|
		=\dfrac{(n+1)(21n^3+9n^2-32n-16)}{630}, 
	\end{eqnarray*}
	\begin{eqnarray*}
		\E\left(A_1^3A_2 \right)=\frac{1}{A_n^3}\sum_{i_1\neq\dots\neq i_3}^n\left|i_1-i_2 \right|^3\left|i_2-i_3 \right|
		=\dfrac{(n+1)(16n^3+4n^2-25n-8)}{420},  
	\end{eqnarray*}
	\begin{eqnarray*}
		\E\left(A_1^4 \right)=\frac{1}{A_n^2}\sum_{i_1\neq i_2}^n\left|i_1-i_2 \right|^4
		=\dfrac{n(n+1)(2n^2-3)}{30}.  
	\end{eqnarray*}
	
	Based on these facts, although there are many types of expectations that need to be calculated, we can still obtain
	\begin{eqnarray*}
		&&\E\left( \sum^{n-1}_{i=1}A_i\right)^4\\
		&=&C_{n-1}^1\E\left( A_1^4\right)+C_{n-2}^1C_2^1C_4^3\E\left(A_1^3A_2 \right) +C_{n-2}^2C_2^1C_4^3\E\left(A_1^3A_3 \right)+C_{n-2}^1C_4^2\E\left(A_1^2A_2^2 \right)\\ &+&C_{n-2}^2C_4^2\E\left(A_1^2A_3^2 \right)	+C_{n-3}^1\frac{A_4^4}{A_2^2}\E\left(A_1A_2^2A_3 \right)+C_{n-3}^1C_2^1\frac{A_4^4}{A_2^2}\E\left(A_1^2A_2A_3 \right) +C_{n-3}^2C_4^1\frac{A_4^4}{A_2^2}\E\left(A_1^2A_2A_4 \right)\\ &+&C_{n-3}^2C_2^1\frac{A_4^4}{A_2^2}\E\left(A_1^2A_3A_4 \right)+C_{n-3}^3C_3^1\frac{A_4^4}{A_2^2}\E\left(A_1^2A_3A_5 \right)+C_{n-4}^1A_4^4\E\left(A_1A_2A_3A_4 \right)\\ &+&C_{n-4}^2C_2^1A_4^4\E\left(A_1A_2A_3A_5 \right)+C_{n-4}^2A_4^4\E\left(A_1A_2A_4A_5 \right) +C_{n-4}^3C_3^1A_4^4\E\left(A_1A_2A_4A_6 \right)\\ &+&C_{n-4}^4A_4^4\E\left(A_1A_3A_5A_7 \right) \\
		&=&\dfrac{(n+1)(350n^7+490n^6-4192n^5+1990n^4+10835n^3-14765n^2-5553n+18000)}{28350}.
	\end{eqnarray*}
	Thus, combining the results of $ \E\left( \sum^{n-1}_{i=1}A_i\right)^3 $ and $ \E\left(  \sum^{n-1}_{i=1}A_i\right)^4 $, we have
	\begin{eqnarray*}
		\E\left( \hat{\xi}_{kl}^4\right)&=&\E\left(1-\dfrac{3}{n^2-1}\sum^{n-1}_{i=1}A_i\right)^4 \\
		&=&1-4\times\dfrac{3}{n^2-1}\E\left( \sum_{i=1}^{n-1}A_i\right)+6\times \left(\dfrac{3}{n^2-1} \right) ^2\E\left( \sum_{i=1}^{n-1}A_i\right)^2
		-4\left(\dfrac{3}{n^2-1} \right) ^3\E\left( \sum_{i=1}^{n-1}A_i\right)^3\\
		&+&\left(\dfrac{3}{n^2-1} \right)^4\E\left( \sum_{i=1}^{n-1}A_i\right)^4=\dfrac{3(56n^5-420n^4+1095n^3-925n^2-671n+3250)}{350(n-1)^4(n+1)^3} .
	\end{eqnarray*}
	Ultimately, one can obtain
	\begin{eqnarray*}
		\Var(\hat{\xi}_{kl}^2)&=&\E\left( \hat{\xi}_{kl}^4\right)-\left( \E\left( \hat{\xi}_{kl}^2\right)\right) ^2\\
		&=&\dfrac{224n^5-1792n^4+5051n^3-4969n^2-2458n+18128}{700(n-1)^4(n+1)^3}.
	\end{eqnarray*}
\end{proof}

\begin{proof}[\textbf{Proof of Lemma 2.2.}]
	Similar to the proof of Lemma 2.1, we abbreviate the symbol $ R_{[j]k}^l $ to $S_j.$ Then, 
	\begin{eqnarray*}
		\hat{\xi}_{lk}=1-\dfrac{3\sum^{n-1}_{j=1}|R_{[j+1],k}^l-R_{[j]k}^l|}{n^2-1}=1-\dfrac{3}{n^2-1}\sum^{n-1}_{j=1}B_j,
	\end{eqnarray*}
	where $B_j=|S_{j+1}-S_j|$. Recall that $\hat{\xi}_{kl}=1-\dfrac{3}{n^2-1}\sum^{n-1}_{i=1}A_i,$
	where $A_i=|R_{i+1}-R_i|$.
	
	Our goal here is to calculate the covariance of $\hat{\xi}_{kl}^2$ and $\hat{\xi}_{lk}^2$, 
	with a series of calculation,
	\begin{eqnarray*}
		&&\operatorname{Cov}\left(\hat{\xi}_{kl}^2,\hat{\xi}_{lk}^2 \right) \\
		&&=4\left( \dfrac{3}{n^2-1}\right) ^2\operatorname{Cov}\left(\sum^{n-1}_{i=1}A_i,\ \sum^{n-1}_{j=1}B_j \right)-2\left( \dfrac{3}{n^2-1}\right) ^3 \operatorname{Cov}\left(\sum^{n-1}_{i=1}A_i,\ \left(\sum^{n-1}_{j=1}B_j \right)^2\right)\\
		&&-2\left( \dfrac{3}{n^2-1}\right) ^3\operatorname{Cov}\left(\left(\sum^{n-1}_{i=1}A_i\right)^2,\ \sum^{n-1}_{j=1}B_j \right)+\left( \dfrac{3}{n^2-1}\right)^4\operatorname{Cov}\left(\left(\sum^{n-1}_{i=1}A_i\right)^2,\left(\sum^{n-1}_{j=1}B_j \right)^2\right).
	\end{eqnarray*}
	According to the proof of Lemma \ref{lemma2.1}, we have
	\begin{eqnarray*}
		\E\left( \sum^{n-1}_{i=1}A_i\right)=\E\left( \sum^{n-1}_{j=1}B_j\right)=\dfrac{n^2-1}{3},
	\end{eqnarray*}
	\begin{eqnarray*}
		\E\left( \sum^{n-1}_{i=1}A_i\right)^2 =\E\left( \sum^{n-1}_{j=1}B_j\right)^2 =\dfrac{(n+1)(10n^3-6n^2-25n+24)}{90}.
	\end{eqnarray*}
	To calculate the expectation of $ \left(\sum^{n-1}_{i=1}A_i\right)\left(\sum^{n-1}_{j=1}B_j \right) $, based on Lemma 1 in \cite{zhang2023asymptotic} which provided  an unrefined covariance result and a series of calculations, we can obtain
	\begin{eqnarray*}
		\operatorname{Cov}\left(\sqrt{n} \hat{\xi}_{kl},\sqrt{n}\hat{\xi}_{lk}\right)=\dfrac{(n-2)(2n-3)}{(n+1)^2(n-1)}, 
	\end{eqnarray*}
	furthermore,
	\begin{eqnarray*} \operatorname{E}\left(\sum^{n-1}_{i=1}A_i\right)\left(\sum^{n-1}_{j=1}B_j \right)=\dfrac{(n-1)(n^4+n^3+n^2-8n+6)}{9n}, \end{eqnarray*}
	then,
	\begin{eqnarray*}
		\operatorname{Cov}\left(\sum^{n-1}_{i=1}A_i,\ \sum^{n-1}_{j=1}B_j \right)=\dfrac{(n-2)(n-1)(2n-3)}{9n}.
	\end{eqnarray*}
	
	Next, we calculate the expectations for the remaining three terms.
	
	We first deal with the most tedious $\E\left[ \left(\sum^{n-1}_{i=1}A_i\right)^2\left(\sum^{n-1}_{j=1}B_j \right)^2\right]. $  Straightforward calculation  shows that
	\begin{eqnarray*}
		\E\left[ \left(\sum^{n-1}_{i=1}A_i\right)^2\left(\sum^{n-1}_{j=1}B_j \right)^2\right]&=&\sum^{n-1}_{i=1}\sum^{n-1}_{j=1}\sum^{n-1}_{k=1}\sum^{n-1}_{l=1}\E\left( A_iA_jB_kB_l\right)\\
		&=&\sum^{n-1}_{i=1}\sum^{n-1}_{j=1}\sum^{n-1}_{k=1}\sum^{n-1}_{l=1}\E\left( |R_{i+1}-R_{i}||R_{j+1}-R_{j}||S_{k+1}-S_{k}||S_{l+1}-S_{l}|\right). 
	\end{eqnarray*}\\
	Since indexes $ i $, $ j $, $ k $  and $ l $ affect the expected values, we must divide them into the following six categories.
	
	\textbf{Case 1}: $i\neq j+1$, $j\neq i+1$, $i\neq j,$ $k\neq l+1$, $l\neq k+1$, $k\neq l$.\\
	
	Denote  $\Omega=\{k+1,k,l+1,l\}$ and $\Omega'=\{i+1,i,j+1,j\}$.  Similar to   the proof of Lemma 1 in \cite{zhang2023asymptotic}, in order to  derive $ \operatorname{E}|R_{i+1}-R_{i}||R_{j+1}-R_{j}||S_{k+1}-R_{k}||S_{l+1}-S_{l}| $ by means of  the law of total expectation, we  divide the space of quaternion  $\left(R_{i+1}, R_{i},R_{j+1},R_{j}\right)$  into the following five parts.
	
	(1)There are  $A_4^4$ types that $  \left(R_{i+1}, R_{i},R_{j+1},R_{j}\right)$ takes from   $ \Omega$ with the probability of 
	$\frac{1}{A_{n}^4} $.
	
	Assuming the $ m $-th type occurs as event $  Z_{m}^{1}$ $ (m=1,\cdots,A_4^4)$, given condition $  Z_{m}^{1} $, $  |R_{i+1}-R_{i}||R_{j+1}-R_{j}| $ and $ |S_{k+1}-S_{k}||S_{l+1}-S_{l}|  $ are independent, then
	\begin{eqnarray*}
		&&T_{klij}^{(1)}:=\sum_{m=1}^{A_4^4}\E\left( |R_{i+1}-R_{i}||R_{j+1}-R_{j}||S_{k+1}-S_{k}||S_{l+1}-S_{l}||Z_{m}^{1}\right)\\
		&&=\sum_{m=1}^{A_4^4}\E\left( |R_{i+1}-R_{i}||R_{j+1}-R_{j}||Z_{m}^{1}\right)\times\E\left(|S_{k+1}-S_{k}||S_{l+1}-S_{l}||Z_{m}^{1}\right)\\
		&&=8+4\left(|k-l|^2+|k+1-l||l+1-k| \right)\left(|i-j|^2+|i+1-j||j+1-i| \right),  
	\end{eqnarray*}
	thus, for all $ i, j, k, l $,
	\begin{eqnarray*}
		&&\sum_{\substack{i\neq j+1 \\ j\neq i+1 \\ i\neq j}}^n\sum_{\substack{k\neq l+1 \\ l\neq k+1 \\ k\neq l}}^n	T_{klij}^{(1)}=8\left( \sum_{\substack{i\neq j+1 \\ j\neq i+1 \\ i\neq j}}^n1\right) ^2+4\left( \sum_{\substack{i\neq j+1 \\ j\neq i+1 \\ i\neq j}}^n\left(|i-j|^2+|i+1-j||j+1-i| \right) \right) ^2\\
		&&=8\left[(n-3)(n-2) \right]^2+4\left[\dfrac{(n-3)(n-2)(n^2+n+1)}{3} \right] ^2\\
		&&=\dfrac{4(n-3)^2(n-2)^2(n^4+2n^3+3n^2+2n+19)}{9}.
	\end{eqnarray*}
	(2) There are  $A_4^3C_4^3$ types that any three components of $\left(R_{i+1}, R_{i},R_{j+1},R_{j}\right)$ takes three different values from $ \Omega=\{k+1,k,l+1,l\}$, the remaining elements do  not take any of $ k+1 $, $ k $, $ l+1 $ or $ l $, the probability of each type is
	$ \frac{A_{n-4}^1}{A_{n}^4}  $.
	Assuming the $ m $-th type occurs as event $  Z_{m}^{2} $ $ (m=1,\cdots,A_4^3C_4^3)$, define
	\begin{eqnarray*}
		C_{kl}=\dfrac{1}{A_{n-4}^1}\sum_{i_1\neq \Omega}^{n}|i_1-(k+1)|
		=\dfrac{1}{A_{n-4}^1}\left( \sum_{i_1=1}^{n}|i_1-(k+1)|-1-|l-k|-|k+1-l|\right), 
	\end{eqnarray*}
	\begin{eqnarray*} 
		C_{ij}=\dfrac{1}{A_{n-4}^1}\sum_{j_1\neq \Omega'}^{n}|j_1-(i+1)|=\dfrac{1}{A_{n-4}^1}\left( \sum_{j_1=1}^{n}|j_1-(i+1)|-1-|j-i|-|i+1-j|\right) ,
	\end{eqnarray*}
	\begin{eqnarray*}
		D_{kl}=\dfrac{1}{A_{n-4}^1}\sum_{i_1\neq \Omega}^{n}|i_1-k|=\dfrac{1}{A_{n-4}^1}\left( \sum_{i_1=1}^{n}|i_1-k|-1-|l+1-k|-|k-l|\right), 
	\end{eqnarray*}
	\begin{eqnarray*}
		D_{ij}=\dfrac{1}{A_{n-4}^1}\sum_{j_1\neq \Omega'}^{n}|i_1-i|=\dfrac{1}{A_{n-4}^1}\left( \sum_{j_1=1}^{n}|j_1-i|-1-|j+1-i|-|i-j|\right) ,
	\end{eqnarray*}
	\begin{eqnarray*}
		E_{kl}=\dfrac{1}{A_{n-4}^1}\sum_{i_1\neq \Omega}^{n}|i_1-(l+1)|=\dfrac{1}{A_{n-4}^1}\left( \sum_{i_1=1}^{n}|i_1-(l+1)|-|k-l|-|l+1-k|-1\right), 
	\end{eqnarray*}
	\begin{eqnarray*}
		E_{ij}=\dfrac{1}{A_{n-4}^1}\sum_{j_1\neq \Omega'}^{n}|j_1-(j+1)|=\dfrac{1}{A_{n-4}^1}\left( \sum_{j_1=1}^{n}|j_1-(j+1)|-|i-j|-|j+1-i|-1\right) ,
	\end{eqnarray*}
	\begin{eqnarray*}
		F_{kl}=\dfrac{1}{A_{n-4}^1}\sum_{i_1\neq \Omega}^{n}|i_1-l|=\dfrac{1}{A_{n-4}^1}\left( \sum_{i_1=1}^{n}|i_1-l|-|k+1-l|-|k-l|-1\right) ,
	\end{eqnarray*}
	\begin{eqnarray*}
		F_{ij}=\dfrac{1}{A_{n-4}^1}\sum_{j_1\neq \Omega'}^{n}|i_1-j|=\dfrac{1}{A_{n-4}^1}\left( \sum_{j_1=1}^{n}|j_1-j|-|i+1-j|-|i-j|-1\right) ,
	\end{eqnarray*}
	\begin{eqnarray*}
		\sum_{i_1=1}^{n}|j_1-k|=\dfrac{(k-1)k+(n-k)(n-k+1)}{2},
	\end{eqnarray*}
	\begin{eqnarray*}
		\sum_{i_1=1}^{n}|j_1-(k+1)|=\dfrac{k(k+1)+(n-k-1)(n-k)}{2}.
	\end{eqnarray*}
	Applying the above equations, one has
	\begin{eqnarray*}
		&&T_{klij}^{(2)}:=\sum_{m=1}^{A_4^3C_4^3}\E\left( |R_{i+1}-R_{i}||R_{j+1}-R_{j}||S_{k+1}-S_{k}||S_{l+1}-S_{l}||Z_{m}^{2}\right)\\
		&&=2(C_{kl}+D_{kl}+E_{kl}+F_{kl})(C_{ij} + D_{ij}+E_{ij}+F_{ij}) \\
		&&+\bigg[|i-j|(C_{ij}+D_{ij}+E_{ij}+F_{ij})+|j+1-i|(C_{ij}+F_{ij})+|i+1-j|(D_{ij}+E_{ij})\bigg]\\
		&&\times\bigg[|k-l|(C_{kl}+D_{kl}+E_{kl}+F_{kl})+|l+1-k|(C_{kl}+F_{kl})+|k+1-l|(D_{kl}+E_{kl})\bigg].  
	\end{eqnarray*}
	Thus, for all $ i$ ,  $j$ ,  $k$  and  $l $,
	\begin{eqnarray*}
		&&\sum_{\substack{i\neq j+1 \\ j\neq i+1 \\ i\neq j}}^n\sum_{\substack{k\neq l+1 \\ l\neq k+1 \\ k\neq l}}^n	T_{klij}^{(2)}= 2\left( \sum_{\substack{i\neq j+1 \\ j\neq i+1 \\ i\neq j}}^n\left( C_{ij} + D_{ij}+E_{ij}+F_{ij}\right)\right)^2 \\
		&&+\left( \sum_{\substack{i\neq j+1 \\ j\neq i+1 \\ i\neq j}}^n\left(|i-j|(C_{ij}+D_{ij}+E_{ij}+F_{ij})+|j+1-i|(C_{ij}+F_{ij})+|i+1-j|(D_{ij}+E_{ij})\right)\right)^2 \\
		&&= 2\left[\dfrac{2(n-3)(n-2)(2n+3)}{3} \right]^2 +\left[ \dfrac{2(n-3)(n-2)(7n^2+18n+20)}{15}\right]^2 \\
		&&=\dfrac{4(n-3)^2(n-2)^2(49 n^4+ 252 n^3+ 804 n^2+ 1320 n +850  )}{225}.
	\end{eqnarray*}
	(3) There are  $A_4^2C_4^2$ types that any two of $  \left(R_{i+1}, R_{i},R_{j+1},R_{j}\right)$ take two different values from set $ \Omega=\{k+1,k,l+1,l\}$, the remaining  two elements do  not take any of $ k+1 $, $ k $, $ l+1 $, or $ l $, and the probability of each type is
	$ \frac{A_{n-4}^2}{A_{n}^4}  $.
	Assuming the $ m $-th type occurs as event $  Z_{m}^{3} $ $ (m=1,\cdots,A_4^2C_4^2)$, define
	\begin{eqnarray*}
		G_{kl}=\dfrac{1}{A_{n-4}^2}\sum_{\substack{(i_1,i_2)\neq \Omega\\i_1\neq i_2}}^{n}|i_1-i_2|,\ \ 	G_{ij}=\dfrac{1}{A_{n-4}^2}\sum_{\substack{(j_1,j_2)\neq \Omega'\\j_1\neq j_2}}^{n}|j_1-j_2|,
	\end{eqnarray*}
	\begin{eqnarray*}
		H_{kl}=\dfrac{1}{A_{n-4}^2}\sum_{\substack{(i_1,i_2)\neq \Omega\\i_1\neq i_2}}^{n}|i_1-(k+1)||i_2-k|,\ \ 	H_{ij}=\dfrac{1}{A_{n-4}^2}\sum_{\substack{(j_1,j_2)\neq \Omega'\\j_1\neq j_2}}^{n}|j_1-(i+1)||j_2-i|,
	\end{eqnarray*}
	\begin{eqnarray*}
		I_{kl}=\dfrac{1}{A_{n-4}^2}\sum_{\substack{(i_1,i_2)\neq \Omega\\i_1\neq i_2}}^{n}|i_1-(k+1)||i_2-(l+1)|,\ \ 	I_{ij}=\dfrac{1}{A_{n-4}^2}\sum_{\substack{(j_1,j_2)\neq \Omega'\\j_1\neq j_2}}^{n}|j_1-(i+1)||j_2-(j+1)|,
	\end{eqnarray*}
	\begin{eqnarray*}
		J_{kl}=\dfrac{1}{A_{n-4}^2}\sum_{\substack{(i_1,i_2)\neq \Omega\\i_1\neq i_2}}^{n}|i_1-(k+1)||i_2-l|,\ \ 	J_{ij}=\dfrac{1}{A_{n-4}^2}\sum_{\substack{(j_1,j_2)\neq \Omega'\\j_1\neq j_2}}^{n}|j_1-(i+1)||j_2-j|,
	\end{eqnarray*}
	\begin{eqnarray*}
		K_{kl}=\dfrac{1}{A_{n-4}^2}\sum_{\substack{(i_1,i_2)\neq \Omega\\i_1\neq i_2}}^{n}|i_1-k||i_2-(l+1)|,\ \  K_{ij}=\dfrac{1}{A_{n-4}^2}\sum_{\substack{(j_1,j_2)\neq \Omega'\\j_1\neq j_2}}^{n}|j_1-i||j_2-(j+1)|,
	\end{eqnarray*}
	\begin{eqnarray*}
		L_{kl}=\dfrac{1}{A_{n-4}^2}\sum_{\substack{(i_1,i_2)\neq \Omega\\i_1\neq i_2}}^{n}|i_1-k||i_2-l|,\ \ 	L_{ij}=\dfrac{1}{A_{n-4}^2}\sum_{\substack{(j_1,j_2)\neq \Omega'\\j_1\neq j_2}}^{n}|j_1-i||j_2-j|,
	\end{eqnarray*}
	\begin{eqnarray*}
		M_{kl}=\dfrac{1}{A_{n-4}^2}\sum_{\substack{(i_1,i_2)\neq \Omega\\i_1\neq i_2}}^{n}|i_1-(l+1)||i_2-l|,\ \ 	M_{ij}=\dfrac{1}{A_{n-4}^2}\sum_{\substack{(j_1,j_2)\neq \Omega'}}^{n}|j_1-(j+1)||j_2-j|.
	\end{eqnarray*}
	By applying the above equations, one has
	\begin{eqnarray*}
		&&T_{klij}^{(3)}:=\sum_{m=1}^{A_4^2C_4^2}\E\left( |R_{i+1}-R_{i}||R_{j+1}-R_{j}||S_{k+1}-S_{k}||S_{l+1}-S_{l}||Z_{m}^3\right)\\
		&&=\dfrac{1}{(A_{n-4}^2)^2}\left\lbrace 8G_{kl}G_{ij}+2(2|i-j|+|i+1-j|+|j+1-i|)(H_{kl}+M_{kl})G_{ij}\right. \\
		&&+\left. 2(2|k-l|+|k+1-l|+|l+1-k|)(H_{ij}+M_{ij})G_{kl}\right. \\
		&&+\left. 2(I_{kl}+J_{kl}+K_{kl}+L_{kl})( I_{ij}+J_{ij}+K_{ij}+L_{ij})\right\rbrace.
	\end{eqnarray*}
	For all $ i, j, k$ and  $l $,
	\begin{eqnarray*}
		&&\sum_{\substack{i\neq j+1 \\ j\neq i+1 \\ i\neq j}}^{n-1}\sum_{\substack{k\neq l+1 \\ l\neq k+1 \\ k\neq l}}^{n-1}	T_{klij}^{(3)}\\
		&&=8\left( \sum_{\substack{i\neq j+1 \\ j\neq i+1 \\ i\neq j}}^{n-1}Gij\right)^2+4\left( \sum_{\substack{i\neq j+1 \\ j\neq i+1 \\ i\neq j}}^{n-1}\left( 2|i-j|+|i+1-j|+|j+1-i|\right) Gij\right)\left( \sum_{\substack{i\neq j+1 \\ j\neq i+1 \\ i\neq j}}^{n-1}(H_{ij}+M_{ij})\right) \\
		&&+\left( \sum_{\substack{i\neq j+1 \\ j\neq i+1 \\ i\neq j}}^{n-1}\left( I_{ij}+J_{ij}+K_{ij}+L_{ij}\right)\right) ^2  \\
		&&=8\left[ \dfrac{(n-3)(n-2)(n+1)}{3}\right]^2+2\left[ \dfrac{4(n-3)(n-2)(5n^2+13n+12)}{45}\right]\left[\dfrac{(n-3)(n-2)(7n^2+17n+16)}{30}\right] \\
		&&+\left[ \dfrac{4(n-3)(n-2)(5n^2+13n+12)}{45}\right] ^2  \\
		&&=\dfrac{8(n-3)^2(n-2)^2(205n^4+1048n^3+2536n^2+2934n+1377)}{2025}.
	\end{eqnarray*}
	(4) There are  $A_4^1C_4^1$ types that any element of $  \left(R_{i+1}, R_{i},R_{j+1},R_{j}\right)$ takes one  value from   $ \Omega=\{k+1,k,l+1,l\}$, the remaining three elements do not take any of $ k+1 $, $ k $, $ l+1 $  or $ l $, and the probability of each type is
	$ \frac{A_{n-4}^3}{A_{n}^4}  $.
	Assuming the $ m $-th type occurs as event $  Z_{m}^{4} $ $ (m=1,\cdots,A_4^1C_4^1)$, define
	\begin{eqnarray*}
		O_{kl}=\dfrac{1}{A_{n-4}^3}\sum_{\substack{(i_1,i_2,i_3)\neq \Omega\\i_1\neq i_2\neq i_3}}^{n}|i_1-i_2||i_3-(k+1)|,\ \ 
		O_{ij}=\dfrac{1}{A_{n-4}^3}\sum_{\substack{(j_1,j_2,j_3)\neq \Omega'\\j_1\neq j_2\neq i_3}}^{n}|j_1-j_2||j_3-(i+1)|,
	\end{eqnarray*}
	\begin{eqnarray*}
		P_{kl}=\dfrac{1}{A_{n-4}^3}\sum_{\substack{(i_1,i_2,i_3)\neq \Omega\\i_1\neq i_2\neq i_3}}^{n}|i_1-i_2||i_3-k|,\ \ 
		P_{ij}=\dfrac{1}{A_{n-4}^3}\sum_{\substack{(j_1,j_2,j_3)\neq \Omega'\\j_1\neq j_2\neq i_3}}^{n}|j_1-j_2||j_3-i|,
	\end{eqnarray*}
	\begin{eqnarray*}
		U_{kl}=\dfrac{1}{A_{n-4}^3}\sum_{\substack{(i_1,i_2,i_3)\neq \Omega\\i_1\neq i_2\neq i_3}}^{n}|i_1-i_2||i_3-(l+1)|,\ \ 
		U_{ij}=\dfrac{1}{A_{n-4}^3}\sum_{\substack{(j_1,j_2,j_3)\neq \Omega'\\j_1\neq j_2\neq i_3}}^{n}|j_1-j_2||j_3-(j+1)|,
	\end{eqnarray*}
	\begin{eqnarray*}
		V_{kl}=\dfrac{1}{A_{n-4}^3}\sum_{\substack{(i_1,i_2,i_3)\neq \Omega\\i_1\neq i_2\neq i_3}}^{n}|i_1-i_2||i_3-l|,\ \ 
		V_{ij}=\dfrac{1}{A_{n-4}^3}\sum_{\substack{(j_1,j_2,j_3)\neq \Omega'\\j_1\neq j_2\neq i_3}}^{n}|j_1-j_2||j_3-j|,
	\end{eqnarray*}
	then,
	\begin{eqnarray*}
		&&T_{klij}^{(4)}:=\sum_{m=1}^{A_4^2C_4^2}\E\left( |R_{i+1}-R_{i}||R_{j+1}-R_{j}||S_{k+1}-S_{k}||S_{l+1}-S_{l}||Z_{m}^4\right)\\
		&&=(O_{kl}+P_{kl}+U_{kl}+V_{kl})(O_{ij}+P_{ij}+U_{ij}+V_{ij}).
	\end{eqnarray*}
	For all $ i, j, k$ and  $l $,
	\begin{eqnarray*}
		\sum_{\substack{i\neq j+1 \\ j\neq i+1 \\ i\neq j}}^{n-1}\sum_{\substack{k\neq l+1 \\ l\neq k+1 \\ k\neq l}}^{n-1}	T_{klij}^{(4)}=\left( \sum_{\substack{i\neq j+1 \\ j\neq i+1 \\ i\neq j}}^{n-1}\left[ O_{ij}+P_{ij}+U_{ij}+V_{ij}\right] \right) ^2
		=\left[ \dfrac{2(n-3)(n-2)(10n^2+22n+19)}{45} \right] ^2.
	\end{eqnarray*}
	(5) $ (R_{i+1},R_i,R_{j+1},R_{j}) $ does not take any element in the set $ \{k+1,k,l+1,l\} $ with probability $ \dfrac{A_{n-4}^4}{A_n^4}.  $ Define
	\begin{eqnarray*}
		Q_{kl}=\dfrac{1}{A_{n-4}^4}\sum^n_{\substack{(i_1,i_2,i_3,i_4)\neq (k+1,k,l+1,l)\\i_1\neq i_2\neq i_3\neq i_4}}|i_1-i_2||i_3-i_4|,\ \ 	Q_{ij}=\dfrac{1}{A_{n-4}^4}\sum^n_{\substack{(j_1,j_2,j_3,j_4)\neq \Omega'\\j_1\neq j_2\neq j_3\neq j_4}}|j_1-j_2||j_3-j_4|.
	\end{eqnarray*}
	Then 
	\begin{eqnarray*}
		T_{klij}^{(5)}:=\E\left( |R_{i+1}-R_{i}||R_{j+1}-R_{j}||S_{k+1}-S_{k}||S_{l+1}-S_{l}||Z_2^5\right)=Q_{kl}Q_{ij}.
	\end{eqnarray*}
	For all $ i, j, k$  and  $l $,
	\begin{eqnarray*}
		\sum_{\substack{i\neq j+1 \\ j\neq i+1 \\ i\neq j}}^{n-1}\sum_{\substack{k\neq l+1 \\ l\neq k+1 \\ k\neq l}}^{n-1}	T_{klij}^{(5)}=\left( \sum_{\substack{i\neq j+1 \\ j\neq i+1 \\ i\neq j}}^{n-1}Q_{ij}\right) ^2
		=\left[ \dfrac{(n-3)(n-2)(5n^2+9n+8)}{45} \right] ^2.
	\end{eqnarray*}
	Finally, combining (1), (2), (3), (4) and (5) with the law of total expectation, we can obtain the following result,
	\begin{eqnarray*}
		&&\sum_{\substack{i\neq j+1 \\ j\neq i+1 \\ i\neq j}}^{n-1}\sum_{\substack{k\neq l+1 \\ l\neq k+1 \\ k\neq l}}^{n-1}\E\left( |R_{i+1}-R_{i}||R_{j+1}-R_{j}||S_{k+1}-S_{k}||S_{l+1}-S_{l}|\right)\\
		&&=\sum_{\substack{i\neq j+1 \\ j\neq i+1 \\ i\neq j}}^{n-1}\sum_{\substack{k\neq l+1 \\ l\neq k+1 \\ k\neq l}}^{n-1} \left( T_{klij}^{(1)}\times \dfrac{1}{A_n^4}+T_{klij}^{(2)}\times \dfrac{A_{n-4}^1}{A_n^4}+T_{klij}^{(3)}\times \dfrac{A_{n-4}^2}{A_n^4}+T_{klij}^{(4)}\times \dfrac{A_{n-4}^3}{A_n^4}+T_{klij}^{(5)}\times \dfrac{A_{n-4}^4}{A_n^4}\right) \\
		&&=\dfrac{(n-3)(n-2)(25 n^8- 60 n^7+ 56 n^6- 1194 n^5+ 3587 n^4- 3906 n^3+ 5872 n^2- 1680 n-4500)}{2025 n(n-1)}.
	\end{eqnarray*}
	
	Cases 2-6 are similar to Case 1 and  we only provide the final results  to save space.
	
	\textbf{Case 2}: $i= j+1,\ k\neq l+1, l\neq k+1, k\neq l$,  or $k= l+1,i\neq j+1,j\neq i+1,i\neq j$ or $ j= i+1,k\neq l+1,l\neq k+1,k\neq l,$ or $l= k+1,j\neq i+1, i\neq j+1,i\neq j.$ 
	
	The result of  $i= j+1,\ k\neq l+1, l\neq k+1, k\neq l$ is the same as the remaining three situations, only differing in symbols.
	\begin{eqnarray*}
		&&\sum_{\substack{i=2}}^{n-1}\sum_{\substack{k\neq l+1 \\ l\neq k+1 \\ k\neq l}}^{n-1}\E\left( |R_{i+1}-R_{i}||R_{i}-R_{i-1}||S_{k+1}-S_{k}||S_{l+1}-S_{l}|\right)\\
		&&=\dfrac{(n-3)(n-2)(35n^7+13n^6+83n^5-829n^4+568n^3-954n^2-1556n+4080)}{2700n(n-1)}.
	\end{eqnarray*}
	
	\textbf{Case 3}: $i= j,k\neq l+1,l\neq k+1,k\neq l,$ or $k=l,j\neq i+1,i\neq j+1,i\neq j.$ 
	
	The result of  $i= j,k\neq l+1,l\neq k+1,k\neq l$ is the same as $k=l,j\neq i+1,i\neq j+1,i\neq j$.
	\begin{eqnarray*}
		&&\sum_{i=1}^{n-1}\sum_{\substack{k\neq l+1\\l\neq k+1\\ k\neq l}}^{n-1}\E\left( |R_{i+1}-R_{i}|^2|S_{k+1}-S_{k}||S_{l+1}-S_{l}|\right)\\
		&&=\dfrac{(n-3)(n-2)(5n^6+9n^5+19n^4-13n^3-90n^2+58n+60)}{270n}.
	\end{eqnarray*}
	
	\textbf{Case 4}: $i= j, k= l+1$, or $k=l,i= j+1$, or $i= j,l= k+1$, or $k=l,j= i+1.$ 
	
	The result of $i= j, k= l+1$ is the same as the remaining three situations.
	\begin{eqnarray*}
		&&\sum_{i=1}^{n-1}\sum_{k=2}^{n-1}\E\left( |R_{i+1}-R_{i}|^2|S_{k+1}-S_{k}||S_{k}-S_{k-1}|\right)\\
		&&=\dfrac{(n-2)(7n^6+11n^5+19n^4-23n^3-34n^2-76n-240)}{360n}.
	\end{eqnarray*}
	
	\textbf{Case 5}: $i= j+1,k= l+1$, or $k=l+1,i= j+1$, or   $i= j+1,l= k+1$, or $k=l+1,j= i+1$.
	
	The result of $i= j+1,k= l+1$ is the same as the remaining three situations.
	\begin{eqnarray*}
		&&\sum_{i=2}^{n-1}\sum_{k=2}^{n-1}\E\left( |R_{i+1}-R_{i}||R_{i}-R_{i-1}||S_{k+1}-S_{k}||S_{k}-S_{k-1}|\right)\\
		&&=\dfrac{(n-2)(49n^7+7n^6+55n^5-579n^4-1712n^3+7812n^2-11392n-3840)}{3600n(n-1)}.
	\end{eqnarray*}
	
	\textbf{Case 6}: $i=j$, $k=l.$
	\begin{eqnarray*}
		&&\sum_{i=1}^{n-1}\sum_{k=1}^{n-1}\E\left( |R_{i+1}-R_{i}|^2|S_{k+1}-S_{k}|^2\right)=\dfrac{(n-1)(n^6+n^5+n^4+3n^3-18n^2-8n+24)}{36n}.
	\end{eqnarray*}
	
	Taking the summation over Case 1 to Case 6,
	\begin{eqnarray*}
		&&\E\left[ \left(\sum^{n-1}_{i=1}A_i\right)^2\left(\sum^{n-1}_{j=1}B_j \right)^2\right]\\
		&&=\sum_{\substack{i\neq j+1 \\ j\neq i+1 \\ i\neq j}}^{n-1}\sum_{\substack{k\neq l+1 \\ l\neq k+1 \\ k\neq l}}^{n-1}\E\left( |R_{i+1}-R_{i}||R_{j+1}-R_{j}||S_{k+1}-R_{k}||S_{l+1}-S_{l}|\right)\\
		&&+4\sum_{\substack{i=2}}^{n-1}\sum_{\substack{k\neq l+1 \\ l\neq k+1 \\ k\neq l}}^{n-1}\E\left( |R_{i+1}-R_{i}||R_{i}-R_{i-1}||S_{k+1}-R_{k}||S_{l+1}-S_{l}|\right) \\
		&&+2\sum_{i=1}^{n-1}\sum_{\substack{k\neq l+1\\l\neq k+1\\ k\neq l}}^{n-1}\E\left( |R_{i+1}-R_{i}|^2|S_{k+1}-R_{k}||S_{l+1}-S_{l}|\right) +4\sum_{i=1}^{n-1}\sum_{k=2}^{n-1}\E\left( |R_{i+1}-R_{i}|^2|S_{k+1}-R_{k}||S_{k}-S_{k-1}|\right)\\
		&&+4\sum_{i=2}^{n-1}\sum_{k=2}^{n-1}\E\left( |R_{i+1}-R_{i}||R_{i}-R_{i-1}||S_{k+1}-R_{k}||S_{k}-S_{k-1}|\right)+\sum_{i=1}^{n-1}\sum_{k=1}^{n-1}\E\left( |R_{i+1}-R_{i}|^2|S_{k+1}-R_{k}|^2\right) \\
		&&=(100 n^{10}- 20 n^9+ 116 n^8- 5344 n^7+ 22469 n^6- 55879 n^5+ 99349 n^4-104141 n^3+33846 n^2\\
		&&- 79416 n+195480)/(8100  (-1 + n) n).
	\end{eqnarray*}
	Thus, it is easy to obtain
	\begin{eqnarray*}
		&&\operatorname{Cov}\left(\left(\sum^{n-1}_{i=1}A_i\right)^2,\left(\sum^{n-1}_{j=1}B_j \right)^2\right)\\
		&&=\dfrac{(n-2)(400n^7-2040n^6+6304n^5-14742n^4+21061n^3-10668n^2-4725n-48870)}{4050n(n-1)}.\\
	\end{eqnarray*}
	
	As for the derivation   of $\operatorname{E}\left[ \left(\sum^{n-1}_{i=1}A_i\right)\left(\sum^{n-1}_{j=1}B_j \right)^2\right] $, it is similar to $\operatorname{E}\left[ \left(\sum^{n-1}_{i=1}A_i\right)^2\left(\sum^{n-1}_{j=1}B_j \right)^2\right] $  and we only provide the final result.
	\begin{eqnarray*}
		\E\left[ \left(\sum^{n-1}_{i=1}A_i\right)\left(\sum^{n-1}_{j=1}B_j \right)^2\right]=\dfrac{10n^7+4n^6-n^5-217n^4+543n^3-621n^2+54n+540}{270 n}.
	\end{eqnarray*}
	Then,
	\begin{eqnarray*}
		\operatorname{Cov}\left(\sum^{n-1}_{i=1}A_i,\ \left(\sum^{n-1}_{j=1}B_j \right)^2\right)=\dfrac{(n-2)(20n^4-66n^3+112n^2-87n-135)}{135n}.
	\end{eqnarray*}
	
	For the expectation of $ \left(\sum^{n-1}_{i=1}A_i\right)^2\left(\sum^{n-1}_{j=1}B_j \right) $, we can condition  on $ X_l $, then  $\E\left(  \left(\sum^{n-1}_{i=1}A_i\right)^2\left(\sum^{n-1}_{j=1}B_j \right)\right) $  is the same as $\E\left(  \left(\sum^{n-1}_{i=1}A_i\right)\left(\sum^{n-1}_{j=1}B_j \right)^2\right) $.
	
	Based on the above results, we can obtain
	\begin{eqnarray*}
		&&\operatorname{E}\left(\hat{\xi}_{kl}^2\hat{\xi}_{lk}^2 \right)=\dfrac{(n-2)(16n^7-72n^6+1625n^5-15909n^4+54431n^3-48519n^2-9992n-88140)}{100n(n+1)^4(n-1)^5}.
	\end{eqnarray*}
	
	Therefore, we ultimately obtain the covariance of $ \hat{\xi}_{kl}^2 $ and $ \hat{\xi}_{lk}^2 $ as 
	\begin{eqnarray*}
		&&\operatorname{Cov}\left(\hat{\xi}_{kl}^2,\hat{\xi}_{lk}^2 \right)=\dfrac{(n-2)(784n^5-8022n^4+27301n^3-24228n^2-5045n-44070)}{50n(n+1)^4(n-1)^5}.
	\end{eqnarray*}
\end{proof}

\begin{proof}[\textbf{Proof of Theorem 2.1.}]
	According to the definition of $\hat{\xi}_{kl}$, $\hat{\xi}_{kl}$ is composed of ranks of concomitant  $  X_{l} $ by ordering $  X_{k} $. When $  X_{k} $ and $ X_{l} $ are independent, the component of $\hat{\xi}_{kl}$ only involves $  X_{l} $ and is not affected by $  X_{k} $, and the same goes for $\hat{\xi}_{km}$. Thus, $ \hat{\xi}_{kl} $ and $ \hat{\xi}_{km}$ ($ k\neq l\neq m $) are independent. Similarly, there are also $ \hat{\xi}_{lk} $ and $ \hat{\xi}_{mk} $ ($ k\neq l\neq m $) , $ \hat{\xi}_{kl} $ and $ \hat{\xi}_{mq} $ ($k\neq l\neq m \neq q$). On the contrary, $ \hat{\xi}_{kl} $ and $ \hat{\xi}_{lk} $ are not independent, see the proof of Lemma 2.2 for details. Thus, we cannot directly use the classical  central limit theorem to obtain the asymptotic normality of $ J_\xi $. Fortunately,   $ J_\xi $ can be written as
	\begin{eqnarray*}
		J_\xi=\sigma_{np}^{-1}\sum_{k\neq l}^{p} (\hat{\xi}_{kl}^2-u_n)= \dfrac{1}{\sigma_{np}}\sum_{k=1}^{p-1}\sum_{l=2}^{p}\varphi_{kl},
	\end{eqnarray*}
	where $\varphi_{kl}=\hat{\xi}_{kl}^2+\hat{\xi}_{lk}^2-2u_n$ and is symmetric for $ k $ and $  l $, obviously, for any $1<k\neq l<p$, $\hat{\xi}_{kl}^2+\hat{\xi}_{lk}^2$ are independent and identically distributed, and their expectation exist. Applying the classical Lindeberg-L\'{e}vy central limit theorem, one has 
	$$J_\xi \xrightarrow{d}N(0,1).$$
\end{proof}

\begin{proof}[\textbf{Proof of Theorem 2.2.}]
	Let  $ \mathcal{M}^c $ be the complement of $ \mathcal{M} $ with a cardinality $ p(p-1)-M $. Rewrite $ J_\xi $ as
	$$  J_\xi= \sigma_{np}^{-1}\sum_{(k,l)\in\mathcal{M}} (\hat{\xi}_{kl}^2-u_n)+\sigma_{np}^{-1}\sum_{(k,l)\in\mathcal{M}^c} (\hat{\xi}_{kl}^2-u_n):=J_M+J_M'. $$
	
	First, we consider $J_M$. One has  
	\begin{eqnarray*}
		J_M=\sigma_{np}^{-1}\sum_{(k,l)\in\mathcal{M}} (\hat{\xi}_{kl}^2-u_n)=\sigma_{np}^{-1}\sum_{(k,l)\in\mathcal{M}} (\hat{\xi}_{kl}^2-\xi_{kl}^2)+\sigma_{np}^{-1}\sum_{(k,l)\in\mathcal{M}} (\xi_{kl}^2-u_n):=T_1+T_2. 
	\end{eqnarray*}
	As for $T_1$, for any $ \varepsilon_{n}>0 $, one has 
	\begin{eqnarray*}
		\P\left(|T_1|>\varepsilon_{n} \right) &=&\P\left(\left|  \sigma_{np}^{-1}\sum_{(k,l)\in\mathcal{M}} (\hat{\xi}_{kl}^2-\xi_{kl}^2)\right| >\varepsilon_{n}\right) \\
		&\leqslant&\varepsilon_{n}^{-1}\E\left|  \sigma_{np}^{-1}\sum_{(k,l)\in\mathcal{M}} (\hat{\xi}_{kl}^2-\xi_{kl}^2)\right|\\
		&\leqslant&\varepsilon_{n}^{-1}\sigma_{np}^{-1}\sum_{(k,l)\in\mathcal{M}}\E\left| \hat{\xi}_{kl}^2-\xi_{kl}^2\right| \\
		&\leqslant&2\varepsilon_{n}^{-1}\sigma_{np}^{-1}\sum_{(k,l)\in\mathcal{M}}\E\left| \hat{\xi}_{kl}-\xi_{kl}\right| \\
		&\leqslant&2M\varepsilon_{n}^{-1}\sigma_{np}^{-1}\left[ \E( \hat{\xi}_{kl}-\xi_{kl})^2 \right]^{1/2},
	\end{eqnarray*}
	where the inequalities mentioned above employ Markov's Inequality, $ c_r $ Inequality  and Cauchy-Schwarz Inequality, respectively. Invoking  Proposition 1.1 and  Proposition 1.2 in \cite{lin2022limit}, one has   $$ \E( \hat{\xi}_{kl}-\xi_{kl})^2=\Var\left( \hat{\xi}_{kl}\right) +\left| \E \hat{\xi}_{kl}-\xi_{kl}\right|^2=O\left( \frac{1}{n}\right). $$ 
	Additionally, by Lemma 2.1 and Lemma  2.2, one has   $ u_n=O(n^{-1}) $, $ \sigma_{np}=O(\frac{p}{n}) $. Thus, let  $ \varepsilon_{n}= \dfrac{\sqrt{n}M\log n}{p}$, then there exists some constants $C$, as $ n,p \rightarrow \infty $, 
	\begin{eqnarray*}
		\P\left(|T_1|>\varepsilon_{n} \right)\leqslant 2M\varepsilon_{n}^{-1}\sigma_{np}^{-1}\left[ \E( \hat{\xi}_{kl}-\xi_{kl})^2 \right]^{1/2}=O\left(\dfrac{1}{\log n}\right) \rightarrow0.
	\end{eqnarray*}
	Consequently, one has $$T_1=O_p\left( \dfrac{\sqrt{n}M\log n}{p} \right).$$
	
	For term  $ T_2 $, one has
	$$ T_2=\sigma_{np}^{-1}\sum_{(k,l)\in\mathcal{M}} (\xi_{kl}^2-u_n)=\sigma_{np}^{-1}\sum_{(k,l)\in\mathcal{M}} \xi_{kl}^2-M \sigma_{np}^{-1}u_n\geqslant M \sigma_{np}^{-1}c_{0}^2-M\sigma_{np}^{-1}u_n. $$
	By the fact $ M\sigma_{np}^{-1}u_n=O(\frac{M}{p}) ,$  as $ n,p \rightarrow \infty $ and $ \frac{nM}{p}\rightarrow \infty$, one has
	\begin{eqnarray*}
		J_M=T_1+T_2\geqslant\dfrac{M}{ \sigma_{np}} \left( O_p\left(\dfrac{\log n}{\sqrt{n}} \right)+c_{0}^2-O\left(\dfrac{1}{n} \right) \right) .
	\end{eqnarray*}
	
	Finally, we consider $ J_M'$.  Recall that
	$$ J_M'=\sigma_{np}^{-1}\sum_{(k,l)\in\mathcal{M}^c} (\hat{\xi}_{kl}^2-u_n).$$ 
	Similar to  Lemma 2.1 and Lemma 2.2,   one has $\E \left( J_M'\right) =0$  and $\Var\left(  J_M'\right)  =O\left( \frac{p(p-1)-M }{p^2}\right)$. Therefore, it follows  $$ J_M'=O_p(1).$$

	Thus, as $ n,p\xrightarrow{}\infty $ and $\frac{nM}{p}\xrightarrow{}\infty $,
	$ J_\xi=J_M+J_M'\xrightarrow{p}\infty$. Therefore, $ 	\operatorname{P}(J_\xi>z_q)\xrightarrow{}1 . $ 
	The proof of Theorem 2.2 is completed.
\end{proof}

\begin{proof}[\textbf{Proof of Corollary 2.1.}]
	We continue to use the same symbols from proof of Theorem 2.2. Then
	\begin{eqnarray*}
		\operatorname{P}(J_\xi>z_q)&=&\operatorname{P}\left\lbrace \sigma_{np}^{-1}\left(\sum_{(k,l)\in\mathcal{M}} (\hat{\xi}_{kl}^2-u_n)+\sum_{(k,l)\in\mathcal{M}^c} (\hat{\xi}_{kl}^2-u_n) \right) >z_q\right\rbrace \\
		&=&\operatorname{P}\left\lbrace \sigma_{M}^{-1}\sum_{(k,l)\in\mathcal{M}^c} (\hat{\xi}_{kl}^2-u_n) >\dfrac{\sigma_{np}}{\sigma_M}z_q-\sigma_M^{-1}\sum_{(k,l)\in\mathcal{M}} (\hat{\xi}_{kl}^2-u_n)\right\rbrace\\
		&=&1-\operatorname{P}\left\lbrace \sigma_{M}^{-1}\sum_{(k,l)\in\mathcal{M}^c} (\hat{\xi}_{kl}^2-u_n) \leqslant\dfrac{\sigma_{np}}{\sigma_M}z_q-\sigma_M^{-1}\sum_{(k,l)\in\mathcal{M}} (\hat{\xi}_{kl}^2-u_n)\right\rbrace,   
	\end{eqnarray*}
	where $ \sigma_{M}^2 $ is the variance of $  \sum_{(k,l)\in\mathcal{M}^c} (\hat{\xi}_{kl}^2-u_n) $. When $ n,p\xrightarrow{}\infty $ and $\frac{nM}{p}\xrightarrow{}0 $, obviously, $\frac{M}{p}\xrightarrow{}0 $, and $ p(p-1)-M \rightarrow \infty$, similar to the proof of Theorem 2.2, $\sigma_{M}^{-1}\sum_{(k,l)\in\mathcal{M}^c} (\hat{\xi}_{kl}^2-u_n)\xrightarrow{d}N(0,1)$. With the assistance of $ \sigma_{M}=\sqrt{p(p-1)-M}\times\left( \dfrac{2\sqrt{2}}{5}\times\dfrac{1}{n}+O\left( \dfrac{1}{n^2}\right) \right) $ in Lemma 2.2 and $ u_n=O(\frac{1}{n}) $, we can deduce that
	$$ \sigma_M^{-1} \sum_{(k,l)\in\mathcal{M}} (\hat{\xi}_{kl}^2-u_n)<\sigma_M^{-1} \sum_{(k,l)\in\mathcal{M}} (1-u_n)=M\sigma_M^{-1}(1-u_n) =O\left( \dfrac{nM}{\sqrt{p(p-1)-M}}\right) \rightarrow 0 .$$
	Moreover, according to Lemma 2.1,
	$ \sigma_{np}=\sqrt{p(p-1)}\times\left( \dfrac{2\sqrt{2}}{5}\times\dfrac{1}{n}+O\left( \dfrac{1}{n^2}\right) \right) $, which imply $ \dfrac{\sigma_{np}}{\sigma_M}\rightarrow 1 $ as $ n,p\xrightarrow{}\infty $, further $ \dfrac{\sigma_{np}}{\sigma_M}z_q-\sigma_M^{-1}\sum_{(k,l)\in\mathcal{M}} (1-u_n)\xrightarrow{p} z_q $, then, by Corollary 11.2.3 in \cite{romano2005testing},
	\begin{eqnarray*}
		&&\operatorname{P}\left\lbrace \sigma_{M}^{-1}\sum_{(k,l)\in\mathcal{M}^c} (\hat{\xi}_{kl}^2-u_n) \leqslant\dfrac{\sigma_{np}}{\sigma_M}z_q-\sigma_M^{-1}\sum_{(k,l)\in\mathcal{M}} (\hat{\xi}_{kl}^2-u_n)\right\rbrace\\
		&\geqslant&\operatorname{P}\left\lbrace \sigma_{M}^{-1}\sum_{(k,l)\in\mathcal{M}^c} (\hat{\xi}_{kl}^2-u_n) \leqslant\dfrac{\sigma_{np}}{\sigma_M}z_q-\sigma_M^{-1}\sum_{(k,l)\in\mathcal{M}} (1-u_n)\right\rbrace\rightarrow \Phi(z_q)=1-q.   
	\end{eqnarray*}
	Thus, as $ n,p\xrightarrow{}\infty $ and $\frac{nM}{p}\xrightarrow{}0 $, $\lim_{n,p\xrightarrow{}\infty} \operatorname{P}(J_\xi>z_q) \leqslant q $.
\end{proof}

\begin{proof}[\textbf{Proof of Theorem 3.1.}]
	
	Invoking Theorem 1 in  \cite{arratia1989two}, let $ \psi_{kl}=\hat{\xi}_{kl}/\sqrt{u_n} $, $I =\{(k, l)$ : $1 \leqslant k\neq l \leqslant p\}$. For $\alpha=(k, l) \in I$, let $B_\alpha=\{(r, s) \in I:\{r, s\} \cap\{k, l\} \neq \emptyset\}$  and $A_\alpha=A_{k l}=\left\{\left|\psi_{k l}\right|>t\right\}$, then
	$$
	\left|\operatorname{P}\left(L_n / \sqrt{u_n} \leqslant t\right)-e^{-\lambda}\right| \leqslant b_{1 n}+b_{2 n}+b_{3 n},
	$$
	where
	$$
	\lambda=p(p-1)\operatorname{P}\left(A_{12}\right) .
	$$ 
	Invoking Lemma \ref{lemmaA.1}, we have
	$$
	\operatorname{P}\left(A_{12}\right)=\operatorname{P}\left(\left|\psi_{12}\right|>t\right)=2\{1-\Phi(t)(1+o(1))\} .
	$$ 
	Using the Gaussian distribution inequality, for any $t>0$,
	$$
	\frac{1}{t+1 / t}(2 \pi)^{-1 / 2} \exp \left(-\frac{t^2}{2}\right) \leqslant 1-\Phi(t) \leqslant \frac{1}{t}(2 \pi)^{-1 / 2} \exp \left(-\frac{t^2}{2}\right) .
	$$
	For notational convenience, let
	$$
	t=(4 \log \left( \sqrt{2}p\right) -\log \log \left( \sqrt{2}p\right)+y)^{1 / 2} \asymp(4 \log \left( \sqrt{2}p\right))^{1 / 2}.
	$$ 
	Obviously,  $ t  \rightarrow\infty$  as  $ p \rightarrow\infty$. Then, one has  
	$$ 
	1 / t-1 /(t+1 / t)=1 /\left\{t\left(t^{2}+1\right)\right\} \asymp 1 / t^{3}, $$
	which yields that
	$$
	1-\Phi(t)=\frac{1}{(2 \pi)^{1 / 2} t} \exp \left(-\frac{t^2}{2}\right)\left[1+O\left\{(\log p)^{-3 / 2}\right\}\right] .
	$$ 
	The above results imply 
	$$
	\begin{aligned}
		& \lambda =\frac{\left(\sqrt{2}p \right)^2}{\left[ 8 \pi \log \left(\sqrt{2}p \right)\right] ^{1 / 2}} \exp \left(-\frac{4 \log \left(\sqrt{2}p \right)-\log \log \left(\sqrt{2}p \right)+y}{2}\right)\{1+o(1)\} \\
		& =(8 \pi)^{-1 / 2} \exp \left(-\frac{y}{2}\right)\{1+o(1)\} .
	\end{aligned}
	$$
	
	Next, we will consider  $  b_{1n}$, $b_{2n}$ and $b_{3n} $, respectively. 
	
	$$
	b_{1n}=(4p-6)\times p(p-1) \operatorname{P}\left(A_{12}\right)^2\leqslant 4p^3 \operatorname{P}\left(A_{12}\right)^2,$$
	and
	$$  b_{2n}  \leqslant \sum_{\alpha \in I} \sum_{\beta \neq \alpha, \beta \in B_\alpha} \operatorname{P}\left(A_\alpha \right)\operatorname{P}\left(A_\beta\right)\leqslant (4p-7)p(p-1) \operatorname{P}\left(A_{12}\right)\operatorname{P}\left( A_{13}\right)\leqslant4p^3 \operatorname{P}\left(A_{12}\right)^2 .
	$$
	
	Note that  $ \hat{\xi}_{k l} $ and  $ \hat{\xi}_{l k} $ are not independent for different $ k $ and $  l  .$  Moreover, according to the definition of $A_\alpha$, for each $\alpha=(k,l) \in I$,  the set $\left\{A_\beta, \beta \notin B_\alpha\right\}$ does not contain any elements related to index $ k $ or $ l $, therefore, $A_\alpha$ is independent of $\left\{A_\beta, \beta \notin B_\alpha\right\}. $ Thus $  b_{3n}=0$.
	
	Accordingly, by  means of the Gaussian tail bound $\operatorname{P}\left\{\psi_{kl}>t\right\} \leqslant e^{-t^2 / 2} /\left\{(2 \pi)^{1 / 2} t\right\}$, as $ p \rightarrow \infty, $  we have
	$$
	\begin{aligned} 
		&b_{1n}+b_{2n}+b_{3n} \leqslant \frac{8}{2\pi t^2} p^3 \exp \left(-t^2\right) \\
		&= \frac{8p^3}{2\pi(4 \log \left( \sqrt{2}p\right) -\log \log \left( \sqrt{2}p\right)+y)} \exp \left( -4 \log \left( \sqrt{2}p\right)+\log \log \left( \sqrt{2}p\right)+p\right) \rightarrow 0,
	\end{aligned}
	$$
	which ultimately completes the proof of Theorem 3.1.
\end{proof}

\begin{proof}[\textbf{Proof of Theorem 3.2.}]
	We will continue to keep on symbol $ \mathcal{M}.$  Given $(k',l')\in \mathcal{M}$, one has
	\begin{eqnarray*}
		\P\left( L_{n p}^{2} / u_{n}-c_p>z'_q\right) &=&\P\left(  \max _{1 \leqslant k\neq l \leqslant p}\hat{\xi}_{k l}^{2} / u_{n}-c_p>z'_q\right) \\
		&\geqslant&\P\left(  \hat{\xi}_{k' l'}^{2} / u_{n}-c_p>z'_q\right).
	\end{eqnarray*}
	Recall that $ \hat{\xi}_{k'l'}\xrightarrow{p}\xi_{k'l'}> c_0>0$, $ u_n=O(\frac{1}{n}) $, $ c_p=4 \log ( \sqrt{2}p) -\log \log ( \sqrt{2}p)=O(\log p) $, these imply that as $ n,  p\xrightarrow{}\infty $ and $\frac{\log p}{n}\rightarrow0 $, $ \hat{\xi}_{k' l'}^{2} / u_{n}-c_p \rightarrow\infty.$
	Then, $$ \P\left( L_{n p}^{2} / u_{n}-c_p>z'_q\right)	\geqslant\P\left( \hat{\xi}_{k' l'} ^{2} / u_{n}-c_p>z'_q\right) \rightarrow1.  $$
	The proof of Theorem 3.2 is completed.
\end{proof}

\begin{proof}[\textbf{Proof of Theorem 4.1.}]	
	Let's deal with item (iii) first. Define event
	$$ 	A:=\left\{\max _{1\leqslant k\neq l \leqslant p}\left|\hat{\xi}_{kl}-\xi_{kl}\right| / u_n^{1 / 2}<\delta_{np}\right\},  $$
	then according to Lemma 4.1, $\inf_{\boldsymbol{\xi}_p \in \boldsymbol{\Xi}}\P\left(A\mid\boldsymbol{\xi}_p \right)\rightarrow1$.
	Invoking  the definition of $  S(\boldsymbol{\xi}_p) $, $\left|\xi_{kl}\right|>2 \delta_{np} u_n^{1 / 2} $,  
	for any  $ (k,l) \in S(\boldsymbol{\xi}_p) $, one can draw the following result,
	$$ 
	\frac{\left|\hat{\xi}_{kl}\right|}{ \sqrt {u_n}} \geqslant \frac{\left|\xi_{kl}\right|-\left|\hat{\xi}_{kl}-\xi_{kl}\right|}{\sqrt {u_n}} >\delta_{np} , $$
	which implies that  $ (k,l) \in \hat{S}\left(\boldsymbol{\hat{\xi}}_p \right) $, hence  $ S(\boldsymbol{\xi}_p) \subset \hat{S}\left(\boldsymbol{\hat{\xi}}_p \right) $. Thus,  for  $ \boldsymbol{\xi}_p \in \boldsymbol{\Xi} $,
	$$ \inf _{\boldsymbol{\xi}_p \in \boldsymbol{\Xi}} \P(S(\boldsymbol{\xi}_p) \subset \hat{S}\left(\boldsymbol{\hat{\xi}}_p \right) \mid \boldsymbol{\xi}_p) \rightarrow 1. $$ 
	Moreover, it is readily seen that, under  $ H_{0} $: $ \boldsymbol{\xi}_p=\mathbf{0} $, 
	$$ \P\left(J_{0}=0 \mid H_{0}\right)=\P\left(\hat{S}\left(\boldsymbol{\hat{\xi}}_p \right)=\emptyset \mid H_{0}\right)=\P\left(\max _{1\leqslant k\neq l \leqslant p}\left\{\left|\hat{\xi}_{kl}\right| / u_n^{1 / 2}\right\}<\delta_{np} \mid H_{0}\right) \rightarrow 1. $$
	Item (i) clearly holds true. As for the derivation of item (ii), with $ \boldsymbol{\Xi}_s $ being represented as $ \boldsymbol{\Xi}_s=\{\boldsymbol{\xi}_p\in \boldsymbol{\Xi }_a:S(\boldsymbol{\xi}_p) \neq \emptyset\}$, adopting the same approach as proof of Theorem 3.1 in \cite{fan2015power} with $  \inf _{\xi \in \boldsymbol{\Xi}} \P(S(\boldsymbol{\xi}_p) \subset \hat{S}\left(\boldsymbol{\hat{\xi}}_p \right) \mid \boldsymbol{\xi}_p) \rightarrow 1  $, we can obtain 
	\begin{eqnarray*}
		\sup _{\boldsymbol{\xi}_p \in \boldsymbol{\Xi}_s} \P\left(J_{0} \leqslant \sqrt{p(p-1)} \mid S(\boldsymbol{\xi}_p) \neq \emptyset\right)  \rightarrow 0.
	\end{eqnarray*}
	Therefore,  $ \inf _{\xi \in \boldsymbol{\Xi}_s} \P\left(J_{0}>\sqrt{p(p-1)} \mid S(\boldsymbol{\xi}_p) \neq \emptyset\right) \rightarrow 1 $ .
\end{proof}

\begin{proof}[\textbf{Proof of Theorem 4.2.}]
	Here we mainly verify Assumption 3.1 and the three conditions of Theorem 3.2 in \cite{fan2015power}, then Theorem 4.2 is naturally completed. The verification of Assumption 3.1 has been completed in the proof of Lemma 4.1, where we replaced $ \hat{v}_{kl} $ in \cite{fan2015power} with $ u_n $ and $ u_n $ is the exact variance of $ \hat{\xi}_{kl} $ under $ H_0 $. Conditions (i) and (ii) can be obtained using Theorem 2.1 and Lemma \ref{lemmaA.1}, respectively. Next we will prove Condition (iii).  Recall that  $ J_\xi=\sigma_{np}^{-1}\sum_{k\neq l}^{p} (\hat{\xi}_{kl}^2-u_n) $  and by Lemma 2.1 and Lemma 2.2, $u_n= \dfrac{2}{5}\times\dfrac{1}{n}+O\left( \dfrac{1}{n^2}\right), \sigma_{np}=\sqrt{p(p-1)}\left( \dfrac{2\sqrt{2}}{5}\times\dfrac{1}{n}+O\left( \dfrac{1}{n^2}\right) \right) $. Therefore, as $n,p \rightarrow \infty$, as long as $ c> \dfrac{\sqrt{2}}{2} $, we will have $  c\sqrt{p(p-1)}-\dfrac{p(p-1)u_n}{\sigma_{np}} \rightarrow  \infty.$   Consequently,
	\begin{eqnarray*}
		\inf _{\boldsymbol{\xi}_p \in \boldsymbol{\Xi}_{s}} \P\left(c\sqrt{p(p-1)} +J_{\xi}>z_q \mid \boldsymbol{\xi}_p\right) 
		&\geqslant& \inf _{\boldsymbol{\xi}_p \in \boldsymbol{\Xi}_{s}} \P\left( c\sqrt{p(p-1)}-\dfrac{p(p-1)u_n}{\sigma_{np}}>z_q \mid \boldsymbol{\xi}_p\right) \rightarrow 1 .
	\end{eqnarray*}
	As all conditions have been satisfied, we directly complete the proof of  Theorem 4.2 by applying Theorem 3.2 in \cite{fan2015power}.
\end{proof}

\section*{Acknowledgments}
This work was supported by the National Natural Science Foundation of China (Nos. 11971045, 12271014 and 12071457),  and the Science and Technology Project of Beijing Municipal Education Commission (No. KM202210005012).

\vskip15pt

\bibliographystyle{elsarticle-harv}
\bibliography{High_dimension_chatterjee}

\begin{thebibliography}{27}
\expandafter\ifx\csname natexlab\endcsname\relax\def\natexlab#1{#1}\fi
\providecommand{\url}[1]{\texttt{#1}}
\providecommand{\href}[2]{#2}
\providecommand{\path}[1]{#1}
\providecommand{\DOIprefix}{doi:}
\providecommand{\ArXivprefix}{arXiv:}
\providecommand{\URLprefix}{URL: }
\providecommand{\Pubmedprefix}{pmid:}
\providecommand{\doi}[1]{\href{http://dx.doi.org/#1}{\path{#1}}}
\providecommand{\Pubmed}[1]{\href{pmid:#1}{\path{#1}}}
\providecommand{\bibinfo}[2]{#2}
\ifx\xfnm\relax \def\xfnm[#1]{\unskip,\space#1}\fi
\bibitem[{Angus(1995)}]{angus1995coupling}
\bibinfo{author}{Angus, J.E.}, \bibinfo{year}{1995}.
\newblock \bibinfo{title}{A coupling proof of the asymptotic normality of the
  permutation oscillation}.
\newblock \bibinfo{journal}{Probability in the Engineering and Informational
  Sciences} \bibinfo{volume}{9}, \bibinfo{pages}{615--621}.
\bibitem[{Arratia et~al.(1989)Arratia, Goldstein and Gordon}]{arratia1989two}
\bibinfo{author}{Arratia, R.}, \bibinfo{author}{Goldstein, L.},
  \bibinfo{author}{Gordon, L.}, \bibinfo{year}{1989}.
\newblock \bibinfo{title}{Two moments suffice for poisson approximations: the
  chen-stein method}.
\newblock \bibinfo{journal}{The Annals of Probability} \bibinfo{volume}{17},
  \bibinfo{pages}{9--25}.
\bibitem[{Bao et~al.(2015)Bao, Lin, Pan and Zhou}]{bao2015spectral}
\bibinfo{author}{Bao, Z.}, \bibinfo{author}{Lin, L.}, \bibinfo{author}{Pan,
  G.}, \bibinfo{author}{Zhou, W.}, \bibinfo{year}{2015}.
\newblock \bibinfo{title}{Spectral statistics of large dimensional spearman’s
  rank correlation matrix and its application}.
\newblock \bibinfo{journal}{The Annals of Statistics} \bibinfo{volume}{43},
  \bibinfo{pages}{2588--2623}.
\bibitem[{Cai and Jiang(2011)}]{cai2011limiting}
\bibinfo{author}{Cai, T.T.}, \bibinfo{author}{Jiang, T.}, \bibinfo{year}{2011}.
\newblock \bibinfo{title}{Limiting laws of coherence of random matrices with
  applications to testing covariance structure and construction of compressed
  sensing matrices}.
\newblock \bibinfo{journal}{The Annals of Statistics} \bibinfo{volume}{39},
  \bibinfo{pages}{1496--1525}.
\bibitem[{Chatterjee(2021)}]{chatterjee2021new}
\bibinfo{author}{Chatterjee, S.}, \bibinfo{year}{2021}.
\newblock \bibinfo{title}{A new coefficient of correlation}.
\newblock \bibinfo{journal}{Journal of the American Statistical Association}
  \bibinfo{volume}{116}, \bibinfo{pages}{2009--2022}.
\bibitem[{Dette et~al.(2013)Dette, Siburg and Stoimenov}]{dette2013copula}
\bibinfo{author}{Dette, H.}, \bibinfo{author}{Siburg, K.F.},
  \bibinfo{author}{Stoimenov, P.A.}, \bibinfo{year}{2013}.
\newblock \bibinfo{title}{A copula-based non-parametric measure of regression
  dependence}.
\newblock \bibinfo{journal}{Scandinavian Journal of Statistics}
  \bibinfo{volume}{40}, \bibinfo{pages}{21--41}.
\bibitem[{Drton et~al.(2020)Drton, Han and Shi}]{drton2020high}
\bibinfo{author}{Drton, M.}, \bibinfo{author}{Han, F.}, \bibinfo{author}{Shi,
  H.}, \bibinfo{year}{2020}.
\newblock \bibinfo{title}{High-dimensional consistent independence testing with
  maxima of rank correlations}.
\newblock \bibinfo{journal}{The Annals of Statistics} \bibinfo{volume}{48},
  \bibinfo{pages}{3206--3227}.
\bibitem[{Fan et~al.(2015)Fan, Liao and Yao}]{fan2015power}
\bibinfo{author}{Fan, J.}, \bibinfo{author}{Liao, Y.}, \bibinfo{author}{Yao,
  J.}, \bibinfo{year}{2015}.
\newblock \bibinfo{title}{Power enhancement in high-dimensional cross-sectional
  tests}.
\newblock \bibinfo{journal}{Econometrica} \bibinfo{volume}{83},
  \bibinfo{pages}{1497--1541}.
\bibitem[{Hall and Heyde(1980)}]{hall2014martingale}
\bibinfo{author}{Hall, P.}, \bibinfo{author}{Heyde, C.C.},
  \bibinfo{year}{1980}.
\newblock \bibinfo{title}{Martingale limit theory and its application}.
\newblock \bibinfo{publisher}{Academic press}, \bibinfo{address}{New York}.
\bibitem[{Han et~al.(2017)Han, Chen and Liu}]{han2017distribution}
\bibinfo{author}{Han, F.}, \bibinfo{author}{Chen, S.}, \bibinfo{author}{Liu,
  H.}, \bibinfo{year}{2017}.
\newblock \bibinfo{title}{Distribution-free tests of independence in high
  dimensions}.
\newblock \bibinfo{journal}{Biometrika} \bibinfo{volume}{104},
  \bibinfo{pages}{813--828}.
\bibitem[{Irizarry et~al.(2003)Irizarry, Hobbs, Collin, Beazer-Barclay,
  Antonellis, Scherf and Speed}]{irizarry2003exploration}
\bibinfo{author}{Irizarry, R.A.}, \bibinfo{author}{Hobbs, B.},
  \bibinfo{author}{Collin, F.}, \bibinfo{author}{Beazer-Barclay, Y.D.},
  \bibinfo{author}{Antonellis, K.J.}, \bibinfo{author}{Scherf, U.},
  \bibinfo{author}{Speed, T.P.}, \bibinfo{year}{2003}.
\newblock \bibinfo{title}{Exploration, normalization, and summaries of high
  density oligonucleotide array probe level data}.
\newblock \bibinfo{journal}{Biostatistics} \bibinfo{volume}{4},
  \bibinfo{pages}{249--264}.
\bibitem[{Leung and Drton(2018)}]{leung2018testing}
\bibinfo{author}{Leung, D.}, \bibinfo{author}{Drton, M.}, \bibinfo{year}{2018}.
\newblock \bibinfo{title}{Testing independence in high dimensions with sums of
  rank correlations}.
\newblock \bibinfo{journal}{The Annals of Statistics} \bibinfo{volume}{46},
  \bibinfo{pages}{280--307}.
\bibitem[{Lin and Han(2022)}]{lin2022limit}
\bibinfo{author}{Lin, Z.}, \bibinfo{author}{Han, F.}, \bibinfo{year}{2022}.
\newblock \bibinfo{title}{Limit theorems of {C}hatterjee's rank correlation}.
\newblock \bibinfo{journal}{arXiv preprint arXiv:2204.08031} .
\bibitem[{Lin and Han(2023)}]{lin2023boosting}
\bibinfo{author}{Lin, Z.}, \bibinfo{author}{Han, F.}, \bibinfo{year}{2023}.
\newblock \bibinfo{title}{On boosting the power of {C}hatterjee's rank
  correlation}.
\newblock \bibinfo{journal}{Biometrika} \bibinfo{volume}{110},
  \bibinfo{pages}{283--299}.
\bibitem[{Mao(2014)}]{mao2014new}
\bibinfo{author}{Mao, G.}, \bibinfo{year}{2014}.
\newblock \bibinfo{title}{A new test of independence for high-dimensional
  data}.
\newblock \bibinfo{journal}{Statistics \& Probability Letters}
  \bibinfo{volume}{93}, \bibinfo{pages}{14--18}.
\bibitem[{Mao(2015)}]{mao2015note}
\bibinfo{author}{Mao, G.}, \bibinfo{year}{2015}.
\newblock \bibinfo{title}{A note on testing complete independence for high
  dimensional data}.
\newblock \bibinfo{journal}{Statistics \& Probability Letters}
  \bibinfo{volume}{106}, \bibinfo{pages}{82--85}.
\bibitem[{Mao(2017)}]{mao2017robust}
\bibinfo{author}{Mao, G.}, \bibinfo{year}{2017}.
\newblock \bibinfo{title}{Robust test for independence in high dimensions}.
\newblock \bibinfo{journal}{Communications in Statistics-Theory and Methods}
  \bibinfo{volume}{46}, \bibinfo{pages}{10036--10050}.
\bibitem[{Mao(2018)}]{mao2018testing}
\bibinfo{author}{Mao, G.}, \bibinfo{year}{2018}.
\newblock \bibinfo{title}{Testing independence in high dimensions using
  kendall’s tau}.
\newblock \bibinfo{journal}{Computational Statistics \& Data Analysis}
  \bibinfo{volume}{117}, \bibinfo{pages}{128--137}.
\bibitem[{McDiarmid et~al.(1989)}]{mcdiarmid1989method}
\bibinfo{author}{McDiarmid, C.}, et~al., \bibinfo{year}{1989}.
\newblock \bibinfo{title}{On the method of bounded differences}.
\newblock \bibinfo{journal}{Surveys in combinatorics} \bibinfo{volume}{141},
  \bibinfo{pages}{148--188}.
\bibitem[{Paindaveine and Verdebout(2016)}]{paindaveine2016high}
\bibinfo{author}{Paindaveine, D.}, \bibinfo{author}{Verdebout, T.},
  \bibinfo{year}{2016}.
\newblock \bibinfo{title}{On high-dimensional sign tests}.
\newblock \bibinfo{journal}{Bernoulli} \bibinfo{volume}{22},
  \bibinfo{pages}{1745--1769}.
\bibitem[{Romano and Lehmann(2005)}]{romano2005testing}
\bibinfo{author}{Romano, J.P.}, \bibinfo{author}{Lehmann, E.},
  \bibinfo{year}{2005}.
\newblock \bibinfo{title}{Testing statistical hypotheses}.
\newblock \bibinfo{publisher}{Springer}, \bibinfo{address}{Berlin}.
\bibitem[{Schott(2005)}]{Schott2005testing}
\bibinfo{author}{Schott, J.R.}, \bibinfo{year}{2005}.
\newblock \bibinfo{title}{Testing for complete independence in high
  dimensions}.
\newblock \bibinfo{journal}{Biometrika} \bibinfo{volume}{92},
  \bibinfo{pages}{951--956}.
\bibitem[{Serfling(1980)}]{serfling2009approximation}
\bibinfo{author}{Serfling, R.J.}, \bibinfo{year}{1980}.
\newblock \bibinfo{title}{Approximation theorems of mathematical statistics}.
\newblock \bibinfo{publisher}{John Wiley \& Sons}, \bibinfo{address}{New York}.
\bibitem[{Shi et~al.(2023)Shi, Xu and Du}]{shi2023max}
\bibinfo{author}{Shi, X.}, \bibinfo{author}{Xu, M.}, \bibinfo{author}{Du, J.},
  \bibinfo{year}{2023}.
\newblock \bibinfo{title}{Max-sum test based on spearman's footrule for
  high-dimensional independence tests}.
\newblock \bibinfo{journal}{Computational Statistics \& Data Analysis}
  \bibinfo{volume}{185}, \bibinfo{pages}{107768}.
\bibitem[{Silva et~al.(2013)Silva, Marcal and da~Silva}]{silva2013evaluation}
\bibinfo{author}{Silva, P.F.}, \bibinfo{author}{Marcal, A.R.},
  \bibinfo{author}{da~Silva, R.M.A.}, \bibinfo{year}{2013}.
\newblock \bibinfo{title}{Evaluation of features for leaf discrimination}, in:
  \bibinfo{booktitle}{International Conference Image Analysis and Recognition},
  \bibinfo{organization}{Springer}. pp. \bibinfo{pages}{197--204}.
\bibitem[{Yao et~al.(2018)Yao, Zhang and Shao}]{yao2018testing}
\bibinfo{author}{Yao, S.}, \bibinfo{author}{Zhang, X.}, \bibinfo{author}{Shao,
  X.}, \bibinfo{year}{2018}.
\newblock \bibinfo{title}{Testing mutual independence in high dimension via
  distance covariance}.
\newblock \bibinfo{journal}{Journal of the Royal Statistical Society Series B:
  Statistical Methodology} \bibinfo{volume}{80}, \bibinfo{pages}{455--480}.
\bibitem[{Zhang(2023)}]{zhang2023asymptotic}
\bibinfo{author}{Zhang, Q.}, \bibinfo{year}{2023}.
\newblock \bibinfo{title}{On the asymptotic null distribution of the
  symmetrized chatterjee’s correlation coefficient}.
\newblock \bibinfo{journal}{Statistics \& Probability Letters}
  \bibinfo{volume}{194}, \bibinfo{pages}{109759}.

\end{thebibliography}

\end{document}